\documentclass[10pt]{amsart}

\usepackage{mathtools}
\mathtoolsset{showonlyrefs}

\usepackage{graphicx}
\usepackage{amsmath,amssymb,amsthm}
\usepackage[english]{babel}

\usepackage[Symbol]{upgreek}

\usepackage{amsfonts}
\usepackage{amsmath}
\usepackage{latexsym}
\usepackage{amscd}

\usepackage{fancyvrb}
\usepackage{xcolor}
\usepackage{shadow}
\usepackage{a4wide}

\usepackage{endnotes}
\usepackage{amsopn}
\usepackage{url}
\usepackage{dsfont}

\usepackage{centernot}

\usepackage{enumitem}

\usepackage{tensor}

\usepackage{hyperref}
\hypersetup{colorlinks=true
}

\usepackage[margin=2cm,bottom=2.5cm,top=2.5cm]{geometry}

\numberwithin{figure}{section}
\theoremstyle{plain}
\newtheorem{thm}{Theorem}[section]
\newtheorem{theoreme}{Theorem}[section]
\newtheorem{prop}[thm]{Proposition}

\newtheorem{cor}[thm]{Corollary}
\newtheorem{lemma}[thm]{Lemma}
\numberwithin{equation}{section}

\theoremstyle{remark}

\newtheorem{rmq}[thm]{Remark}
\newtheorem{rem}[thm]{Remark}

\theoremstyle{definition}
\newtheorem{dfn}[thm]{Definition}
\newtheorem{definition}[thm]{Definition}

\newcommand{\eps}{\varepsilon}

\newcommand{\ceps}{\epsilon}

\newcommand{\hatun}[1]{\overset{\lower.9em\hbox{${\scriptscriptstyle 1 \wedge}$}}{#1}}
\newcommand{\hatdeux}[1]{\overset{\lower.9em\hbox{${\scriptscriptstyle \!\!2 \wedge}$}}{#1}}

\newcommand\R{{\mathbb R}} \newcommand\N{{\mathbb N}}
\newcommand\Z{{\mathbb Z}}

 \def\cdotv{\raise 2pt\hbox{,}}

\makeatletter
\def\@tvsp{\mathchoice{{}\mkern-4.5mu}{{}\mkern-4.5mu}{{}\mkern-2.5mu}{}}
\def\ltrivert{\left|\@tvsp\left|\@tvsp\left|}
\def\rtrivert{\right|\@tvsp\right|\@tvsp\right|}
\makeatother

 \def\cdotv{\raise 2pt\hbox{,}}

\newcommand{\Grond}{\mathcal{G}}
\newcommand{\Prond}{\mathcal{P}}

\newcommand{\Nrond}{\mathcal{N}}

\newcommand{\Ab}{\mathbb{A}}

\newcommand{\tprime}{\varrho}

\newcommand{\Lp}{\ell}
\newcommand{\Lo}{\mathfrak{L}}
\newcommand{\tiTheta}{\upsilon}

\newcommand{\canonchi}{\boldsymbol{\chi}}
\newcommand{\cutoffchi}{\chi}

\newcommand{\qtau}{\tau_{q}}

\usepackage{xcolor}

\makeatletter
\def\mathcolor#1#{\@mathcolor{#1}}
\def\@mathcolor#1#2#3{%
  \protect\leavevmode
  \begingroup
    \color#1{#2}#3%
  \endgroup
}
\makeatother

\begin{document}

\title[Dispersion for the wave equation with Neumann boundary condition inside general strictly convex domains]{Dispersion for the wave equation with Neumann boundary condition inside general strictly convex domains}

\author{Oana Ivanovici}

\author{Diego S\'anchez Sanz }
\address{Sorbonne Université, CNRS, LJLL, F-75005 Paris, France} 

  \email{${}^{*}$corresponding author: danela-oana.ivanovici@cnrs.fr}
  \email{danela-oana.ivanovici@cnrs.fr, diego.sanchez-sanz@sorbonne-universite.fr }

 \thanks{
  The authors were supported by ERC grant ANADEL 757 996.}
\date{}

\begin{abstract}
We consider the wave equation on a manifold $(\Omega,g)$ of dimension $d\geq 2$ with smooth
strictly convex boundary $\partial\Omega\neq\emptyset$, with Neumann boundary condition. 
We construct a sharp local in time parametrix for the Neumann wave equation near glancing, extending the classical Melrose-Taylor construction for the Dirichlet problem to the Neumann boundary condition. 
Our construction is based on the microlocal framework and the parametrix developed in \cite{ILLP} for the Dirichlet problem, together with the Melrose-Taylor argument for solving the transport equations under more general, affine-type boundary conditions. Once the Neumann parametrix is constructed, the dispersive and Strichartz estimates follow the analysis developed in \cite{ILLP} for the Dirichlet problem. We therefore only recall the main ingredients of that argument. In particular, the fixed time decay rate for the Green function exhibits the same $t^{1/4}$ loss with respect to the boundaryless case, associated with swallowtail type singularities in the wave front set, and this decay is optimal. Moreover, the corresponding Strichartz estimates are obtained by balancing lossy long time estimates at a given incidence with short time ones with no loss: for $d=3$, this heuristically means that, on average, the decay loss is only $t^{1/6}$.
\end{abstract}

\maketitle

\section{Introduction}\label{intro}

Let us consider the wave equation on a smooth $d$-dimensional Riemannian manifold $(\Omega,g)$, $d\geq 2$, with smooth strictly convex boundary $\partial\Omega\neq\emptyset$, and let $\Delta_g$ denote the Laplace--Beltrami operator. In the present work we consider the Neumann problem
\begin{equation}
\left\{
\begin{aligned}
& (\partial_t^2-\Delta_g)u=0 \text{ in }\Omega,\\
& u|_{t=0}=u_0,\quad \partial_tu|_{t=0}=u_1,\\
& \partial_\nu u|_{\partial\Omega}=0,
\end{aligned}
\right.
\label{WE}
\end{equation}
where $\partial_\nu$ denotes the normal derivative at the boundary. On a smooth Riemannian manifold without boundary, one may construct, to arbitrary order, an approximate solution to the wave equation by standard microlocal methods. Locally, such a parametrix is a Fourier integral operator whose phase solves the eikonal equation. The non-degeneracy properties of this phase then yield pointwise decay estimates for the kernel which are analogous to those in the flat case. More precisely, if $\varkappa\in C^\infty_0((0,+\infty))$, one has, at least for $|t|$ in a sufficiently small fixed interval,
\begin{equation}\label{disprd}
 \big\|\varkappa(-h^2\Delta_g)e^{\pm it\sqrt{-\Delta_g}}
 \big\|_{L^1\to L^\infty}
 \lesssim h^{-d}
 \min\left\{1,\left(\frac{h}{|t|}\right)^{\frac{d-1}{2}}\right\}.
\end{equation}
Such fixed-time estimates are a fundamental tool in the study of
space-time estimates for the wave equation, in particular Strichartz
and spectral projector estimates, and consequently in many nonlinear
problems.

In the presence of a boundary, even the microlocal construction of a
parametrix becomes considerably more involved. Reflections, and in particular glancing and gliding rays, make the geometry of the wave flow much more complicated. As recalled in detail in \cite{ILLP}, the analysis of propagation near the boundary led to major developments in microlocal analysis and to the construction of parametrices capable of describing propagation along the generalized bicharacteristic flow. Such constructions, however, do not in general retain sufficiently precise information on the amplitude of the wave to yield fixed-time dispersive estimates in the presence of gliding rays. For general boundaries, Strichartz estimates were nevertheless obtained, for both Dirichlet and Neumann boundary conditions, by reducing the problem to wave propagation for low regularity metrics and working on sufficiently short time intervals so as to control individual reflections; see in particular \cite{blsmso08} and the references in \cite{ILLP}. This approach does not provide the sharp fixed-time dispersion that requires following the wave through a large number of successive reflections.

For strictly convex domains, such successive reflections can be
described with much greater precision. In the Friedlander model,
explicit parametrices were constructed in \cite{ILP3,ilp12} for the
Dirichlet problem and, more recently, in \cite{DSS26} for Robin
boundary conditions. In a general strictly convex domain, the
Dirichlet problem was carried out in \cite{ILLP}. It relies on the Melrose-Taylor parametrix for the Dirichlet problem near glancing and on Melrose's theorem on glancing hypersurfaces. In particular, in \cite{ILLP} we made the associated canonical transformation sufficiently explicit to recover from the model case the phase functions needed in the general geometry. This allowed us to follow a large number of successive reflections and to obtain sharp fixed-time dispersive estimates, with the optimal $1/4$ loss caused by swallowtail singularities of the reflected wave fronts.\\

The starting point of the construction in \cite{ILLP} is the
Melrose-Taylor parametrix near a glancing point for the Dirichlet boundary condition. In boundary normal
coordinates $(x,y)$, $\Omega=\{x>0,y\in \mathbb{R}^{d-1}\}$ with $\partial\Omega=\{x=0\}$, the corresponding
quasimodes are written in terms of two phase functions $\psi$ and
$\zeta$ and Airy functions in the form

\begin{equation}\label{G}
G(x,y,\eta,\omega)= e^{i\psi(x,y,\eta,\omega)}
 \left(
 q_0(x,y,\eta,\omega) Ai(-\zeta(x,y,\eta,\omega))
 +
 i q_1(x,y,\eta,\omega)
 Ai'(-\zeta(x,y,\eta,\omega))
 \right).
\end{equation}

The phases $\psi$ and $\zeta$ satisfy the coupled eikonal equations
associated with the principal symbol of the Laplace-Beltrami
operator. An essential feature of the Melrose-Taylor construction is
that they may be chosen so that
$\zeta|_{x=0}=\omega$.
In the Dirichlet case, the amplitudes can moreover be constructed so
that
\[
q_0 \text{ is elliptic near $x=0$ and } q_1|_{x=0}=0.
\]
Consequently, when $\omega$ is chosen to be a zero of the Airy
function, the resulting quasimode satisfies the Dirichlet boundary
condition, modulo an error of arbitrarily high order.

A second essential ingredient in \cite{ILLP} is Melrose's
classification theorem for glancing hypersurfaces. It provides,
microlocally near a glancing point, a canonical transformation between
the general strictly convex geometry and the Friedlander model. A
precise form of its generating function was obtained in
\cite[Proposition~2.4 and Appendix]{ILLP}. This makes it possible to
transfer the explicit phase of the model problem to a general
strictly convex domain. The resulting representation is sufficiently
precise not merely to propagate singularities, but also to keep track
of the phase and amplitude through many successive reflections, which
is indispensable for sharp dispersive estimates.\\

The main purpose of the present paper is to construct a parametrix for
the wave equation with Neumann boundary condition microlocally near
glancing in an arbitrary smooth strictly convex domain. For the
Dirichlet problem, such a construction goes back to the fundamental
work of Melrose and Taylor in the 1970s and is the starting point of
the analysis carried out in \cite{ILLP}. To the best of our knowledge,
no corresponding glancing parametrix for the Neumann wave equation in
a general strictly convex domain has been constructed in the
literature. Once this parametrix is available, the dispersive and Strichartz estimates follow from the arguments developed in \cite{ILLP} for the Dirichlet problem.
There is an important distinction between the phase and amplitude
parts of the construction of the solution. The phase functions $\psi$ and $\zeta$ of \eqref{G}
are determined by the eikonal equations and therefore depend only on
the principal symbol of the wave operator and on the geometry of the
glancing hypersurfaces. They are consequently independent of the
choice of boundary condition. In particular, the canonical
transformation obtained from Melrose's equivalence theorem, and the
phase construction derived from it in \cite{ILLP}, remain unchanged
for the Neumann problem.
The new difficulty lies in the amplitudes. Inserting the Airy ansatz
\eqref{G} into the Helmholtz equation leads to the same hierarchy of
transport equations as in the Dirichlet construction, but the
boundary condition selecting their solutions is different. In the
Dirichlet case, Melrose-Taylor construct the amplitudes so that the
coefficient of $Ai'(-\zeta)$ vanishes at the boundary. For the Neumann
condition, one has instead to differentiate the Airy type ansatz
in the normal direction and arrange a cancellation of its $Ai(-\zeta)$
component. This produces a non-trivial boundary relation coupling the
two amplitudes and their normal derivatives.
The main point of the construction is to solve the transport equations
to all orders subject to this boundary relation. For this purpose we
use a more general part of the Melrose-Taylor transport construction:
their analysis of the two involutions associated respectively with the
boundary projection and with the bicharacteristic flow allows one to
solve the transport system while imposing an affine boundary relation
between its two components. This yields the Neumann amplitudes to all
orders, with an elliptic leading coefficient.
We may summarize the main construction as follows.
\begin{thm}[Neumann parametrix near glancing]
\label{thmN}
Let $(\Omega,g)$ be a smooth Riemannian manifold with smooth strictly
convex boundary. Microlocally near a glancing point, one can
construct, to arbitrary order, a parametrix for the wave
equation with Neumann boundary condition.
It is associated with the same
glancing canonical relation and uses the same phase functions
$\psi$ and $\zeta$ as the Melrose-Taylor Dirichlet parametrix,
with amplitudes determined by the Neumann boundary condition.
\end{thm}
\begin{rmq}
Thus, passing from the Dirichlet to the Neumann boundary condition
does not alter the microlocal geometry of the glancing construction;
the difference is entirely carried by the amplitudes.
\end{rmq}

The construction is explicit in the Friedlander model.
For that model, the wave equation with more general Robin boundary
conditions, including both the Dirichlet and Neumann cases, was recently
studied by Sanchez Sanz in \cite{DSS26}, extending 
\cite{ILP3,ilp12}. The explicit structure of the model makes it
possible to describe the corresponding modes and reflected waves
directly. Theorem~\ref{thmN} may
therefore be viewed as the passage from this explicit model to an
arbitrary smooth strictly convex geometry. This passage is not
formal and involves two distincts steps.

The first is the construction of the elementary Neumann quasimode
near glancing (an equivalent for $G$ in \eqref{G} for Neumann). 
Building on the Melrose-Taylor theory, we retain the
Dirichlet phase functions and solve the corresponding transport
equations to all orders with the boundary relation imposed by the
Neumann condition. This produces a Neumann quasimode \(G_N\) whose form is similar to \eqref{G}.
By itself, however, this local quasimode does not yet provide a
solution of the Neumann wave equation. The second step uses the
parametrix construction developed in \cite{ILLP}: starting from
\(G_N\), we construct a sequence of reflected waves and superpose
them so as to obtain a microlocal solution of the wave equation
satisfying the Neumann boundary condition. Thus the Melrose-Taylor
machinery provides the new Neumann building block at glancing,
whereas a successive-reflection construction similar to the one of \cite{ILLP}
turns this building block into the wave parametrix used throughout
the paper.

Finally, as in \cite{ILLP}, applying a Poisson type summation to the
successive-reflection representation leads to a dual representation, 
involving Neumann gallery modes. In the range of very small initial distances from the boundary we have to properly define these gallery modes and prove that their properties are uniform with respect to their discrete parameter, at least in a range useful for our purposes. To our knowledge, these gallery modes
have only been defined in such a uniform way in the general case in \cite{ILLP} for Dirichlet condition; then, one has to carefully construct the initial data by decomposing over the gallery modes, a delicate issue that was notably absent from the model case.
Combining the reflected-wave and gallery-mode
representations, exactly as in \cite{ILLP}, therefore yields a
Neumann parametrix uniformly for the whole range of distances to the
boundary needed in the dispersive analysis.\\

We now turn to the dispersive consequences of the parametrix.
Since its phase is the same as in the Dirichlet problem, the caustics
responsible for the loss of dispersion are the same geometric objects
as in \cite{ILLP}. The Neumann boundary condition changes the
amplitudes carried by the corresponding Lagrangian distributions but
not the underlying canonical geometry. In particular, the swallowtail singularities
responsible for the worst decay are unchanged. As a consequence, we obtain
the same local-in-time decay rate as in the Dirichlet case.

Before stating our main result, let us define strict convexity: our boundary $\partial\Omega\neq \varnothing$ is said to be strictly (geodesically) convex if the induced second fundamental form on $\partial\Omega$ is positive definite. If $\Omega$ is actually a domain in $\R^{d}$ with the identity metric, this definition is equivalent to strict positivity of all principal curvatures at any point of the boundary, and $\Omega$ is a strictly convex domain (it admits a gauge function that is strictly convex.)
\begin{thm}\label{disper}
Let $\varkappa \in C_{0}^{\infty}(]0,+\infty[)$. There exist
$C>0$, $T_{0}>0$ and $a_{0}>0$ such that, uniformly in $a\in ]0,a_{0}]$, $h\in (0,1)$ and $t\in [-T_{0},T_{0}]$, the solution $u_a$ to \eqref{WE} with $(u_0,u_1)=(\delta_a,0)$, $\delta_a$ being any Dirac mass at distance  $a$ from $\partial\Omega$, is such that
\begin{equation}\label{dispco}
\|\varkappa(-h^2 \Delta_{g})u_{a}(t,\cdot)\|_{L^{\infty}}\leq\frac C {h^{d}} \min\left\{1,\left(\frac h{|t|}\right)^{\frac{d-2}{2}+\frac{1}{4}}\right\}\,.
\end{equation}
\end{thm}
Here $\Delta_{g}$ denotes the Neumann realization of the
Laplace-Beltrami operator. As in the Dirichlet case, $T_0$
is independent of the distance of the source to the boundary and of
the semiclassical parameter $h$.

The estimate \eqref{dispco} exhibits the same $1/4$ loss in the
$h/t$ exponent as in the Dirichlet problem. This loss is geometric:
it is produced by the swallowtail singularities of the common
glancing phase and is therefore unaffected by the change of boundary
condition. We can moreover track these caustics and show that the
estimate is optimal.

\begin{thm}\label{disperoptimal}
Let  $u_a$ be the solution to \eqref{WE} with data
$(u_0,u_1)=(\delta_a,0)$. Let $h\in (0,1)$ and $a\geq h^{1/3}$. There exist a constant $C>0$, such that for all $\vartheta\in \mathbb{S}^{d-2}$, there exist a finite sequence
$(t_{n},x_{n},y_{n})_n$, $1\leq n\leq a^{-1/2}$ with $d(x_{n},\partial\Omega)\sim a$, $y_{n}/|y_{n}|\sim \vartheta$, such that
\begin{equation}\label{dispcooptimal}
h^{-d}(h/t_{n})^{\frac{d-2}{2}}n^{-1/4} a^\frac 1 8 h^{1/4}\sim a^\frac
1 4 h^{-d}(h/t_n)^{\frac{d-2}{2}+\frac 1 4}\leq C
 |\varkappa(-h^2 \Delta_{g})u_{a}(t_n,x_{n},y_{n})|\,.
\end{equation}
\end{thm}

As a consequence of the more precise estimates underlying
\eqref{dispco}, we also obtain Strichartz estimates for
the Neumann wave equation in the same spirit as those derived in
\cite{ILLP}. In particular, these estimates improve on what follows
from a direct use of the worst fixed-time decay by taking advantage
of the fact that the maximal loss occurs only for a restricted range
of incidence angles and time scales.

\begin{thm}\label{thStri}
 Let $d\geq 3$ and $u$ be a solution of \eqref{WE} on a manifold $\Omega$ with strictly convex boundary. Then there exist $T$ such that for all $\varepsilon>0$, there exists $C_{T,\varepsilon}$ such that
 \begin{equation}
   \label{eq:SEF}
\|u\|_{L^q(0,T) L^r(\Omega)}\leq
C_{T,\varepsilon} \bigl(\,||u_0||_{\dot{H}^{\beta}(\Omega)} +
||u_1||_{\dot{H}^{\beta-1}(\Omega)} \bigr)\,,
\end{equation}
where $\beta=d/2-1/q-d/r$ (scaling condition) and $(d,q,r)$ such that $q\geq 2$ ($q\neq 2$ for $d=3$),
\[
\frac{1}{q}\leq\Bigl(\frac{d-1}{2}-\upgamma(d)\Bigr)\Bigl(\frac{1}{2}-\frac{1}{r}\Bigr)\,, \text{ with } \upgamma(d)=\frac 1 4-\frac 1 {4d}+\varepsilon=\frac 1 6+\frac 1 4\Bigl(\frac 1 3-\frac 1 d\Bigr)+\varepsilon\,.
\]
\end{thm}

The proof of Theorem~1.5 is omitted, as it follows from the
dispersive estimates established below by exactly the same argument
as in the Dirichlet case; we refer to \cite[Section~5]{ILLP} for
details.
In dimension $d=2$ the known range of admissible indices for which
sharp Strichartz are already known to hold is in fact slightly larger, see
\cite{blsmso08} where $\upgamma(2)=1/6$ (which we may recover with our argument).  Especially noteworthy is $d=3$, for which we get $\upgamma(3)= 1/6+\varepsilon$: such a loss corresponds heuristically to a fixed time dispersion \eqref{dispco} where the $1/4$ loss would be replaced by a $1/6$ loss. In dimensions $d\geq 3$, Theorem
\ref{thStri} improves the known range of indices for which Strichartz estimates hold, and it does so in a uniform way with respect to dimension, in
contrast to \cite{blsmso08}, where $\upgamma(3)=2/3$ and $\upgamma(d)=(d-3)/2$ for $d\geq 4$. The results in \cite{blsmso08} however apply to any domain or manifold with non-empty boundary. 
\\

The paper is organized as follows. In Section \ref{parconstruction} we construct the
Neumann analogue of the Melrose-Taylor glancing quasimode \eqref{G}. We first
recall the canonical transformation associated with the equivalence
of glancing hypersurfaces and the Dirichlet construction of \eqref{G} from
\cite{ILLP}. We then solve the transport equations with the Neumann
boundary constraint and obtain quasimodes adapted to the Neumann
Laplacian. This is the main new microlocal ingredient of the paper.
We subsequently use these quasimodes to construct the corresponding
gallery modes and a parametrix for the wave propagator. 
Sections \ref{sec:dispersion-estimates} and \ref{secdispapetit} are devoted to the analysis of reflected waves and to fixed-time dispersive estimates. Throughout this paper, we use the notation and conventions of \cite{ILLP} and rely on several of its results without reproving them. 
Since the phase functions are unchanged in the Neumann setting, all the arguments in \cite{ILLP} involving only these phases apply verbatim and will merely be recalled when needed. 
The genuinely new part of the parametrix construction, namely the construction of the symbols satisfying the Neumann boundary condition, will instead be carried out in detail. Once this is achieved, these symbols enjoy properties analogous to those of their Dirichlet counterparts - in particular, the leading symbol remains elliptic - and the dispersive analysis follows along essentially the same lines as in \cite{ILLP}. We shall therefore give the main arguments and the necessary verifications, referring to \cite{ILLP} whenever the proofs are unchanged.\\

In the remaining of the paper, $A\lesssim B$ means that there exists a constant $C$ such that $A\leq CB$ and this constant may change from line to line but is independent of all parameters. It will be explicit when (very occasionally) needed. Similarly, $A\sim B$ means both $A\lesssim B$ and $B\lesssim A$. 

\section{A parametrix construction}\label{parconstruction}

We first recall some notations from \cite[Sect.2]{ILLP}.
By finite speed of propagation, we may work locally near the boundary and chose boundary normal coordinates $(x,y)$ on $\Omega$, with $x>0$ on $\Omega$, $y\in\mathbb{R}^{d-1}$ such that $\partial\Omega=\{(0,y): y\in \mathbb{R}^{d-1}\}$ (these coordinates may be interpreted as Fermi coordinates relative to the hypersurface that is the boundary); local coordinates on $\Omega\times\mathbb{R}_t$ are then $(x,y,t)$. Local coordinates on the base induce local coordinates on the cotangent bundle, namely $(x,y,t,\xi,\eta,\mathbf{\tau})$ on $T^*(\Omega\times\mathbb{R}_t)$. The corresponding local coordinates on the boundary are $(y,t,\eta,\mathbf{\tau})$. In this coordinates (and up to conjugation by a non vanishing smooth
factor $e_{g}(x,y)$), the Laplacian $\Delta_g$ can be written as (\cite[III, Appendix C]{Hormander})
\begin{equation}
  \label{eq:laptilde}
 \Delta=e_{g}^{-1} \Delta_g e_{g}=\partial^2_x+R(x,y,\partial_y)\,.
\end{equation}
We assume that the boundary is everywhere strictly (geodesically) convex: for every point $(0,y_0)\in\partial\Omega$ and every $(0,y_{0},0,\eta_0)\in T^*\Omega$ with $\eta_{0}\neq 0$,
\begin{gather}\label{glancingcond}
  \{\xi^2+R(x,y,\eta),x\}(0,y_0,0,\eta_{0})=0\,,\\
 \{\{\xi^2+R(x,y,\eta),x\},\xi^2+R(x,y,\eta)\}(0,y_0,0,\eta_{0})=2\partial_xR(0,y_0,\eta_0)>0,
\end{gather}
where $\{.,.\}$ denotes the Poisson bracket (see \cite[III, 24.3]{Hormander}). We assume (without loss of generality) that $y_0=0$, hence $\kappa_0=(0,0,0,\eta_0)$. On the boundary
and for $(0,y)$ near $(0,0)$, the metric reads $\xi^2+\sum_{j,k=1}^{d-1}R^{j,k}(0,y)\eta_j\eta_k$; using again \cite[III, Appendix C]{Hormander}, we assume moreover that $(R_{j,k}(0,0))_{j,k}$ is the identity matrix, and define
\begin{equation}
  \label{eq:R01}
  R_0(y,\partial_y)  :=R(0,y,\partial_y)=\sum_{j}\partial_{y_j}^2+O(|y|)\,,\quad 
  R_1(y,\partial_y)  :=\partial_x R(0,y,\partial_y)=\sum_{j,k}R_1^{j,k}(y)\partial_{y_j}\partial_{y_k}\,.
\end{equation}
Recall that strict convexity for $\partial\Omega$ is equivalent to $R_1$ being elliptic (the associated quadratic form is  positive definite). Define our model Laplacian $\Delta_{M}$ and (Fourier) multipliers $q, {\qtau}$
\begin{equation}
  \label{eq:LapM}
\Delta_M=\partial^2_x+\sum_{j}\partial_{y_j}^2+ x\sum_{j,k}R_1^{j,k}(0)\partial_{y_j}\partial_{y_k}\,,\, q(\eta)=\sum_{j,k}R_1^{j,k}(0)\eta_j \eta_k\,,\,  {\qtau}(\omega,\eta)=\sqrt{ |\eta|^2+\omega q(\eta)^\frac 2 3}\,.
\end{equation}
Later we will use various functions of variables $(x,y,\eta,\omega,\sigma)$ (where some variables may be omitted depending on context and both new variables $\omega, \sigma \in \R$)
that will be defined in a conic neighborhood of the set
\begin{equation}
  \label{eq:N0}
 N_0=\{x=0\,,\,\,y=0\,,\,\,\omega=0\,,\,\,\sigma=0\,,\,\,\eta\in \R^{d-1}\setminus\{0\}\}\,.
\end{equation}
Such a function $f$ is said to be homogeneous of degree $k$ if
\[
f(x,y, \lambda \eta, \lambda^{2/3}\omega, \lambda^{1/3}\sigma)=\lambda^{k}
f(x,y,\eta,\omega,\sigma)\,.
\]
\begin{definition}
A symbol $a(x,y,\eta,\omega,\sigma)$ is of order $m$ and type $((1, 2/3,1/3),0)$ if 
\[
\forall \beta=(\beta_0,\beta_1,\beta_2,\beta_3) \quad \exists C_{\beta} \quad |\partial^{\beta_0}_{(x,y)}\partial^{\beta_{1}}_{\eta} \partial^{\beta_{2}}_{\omega}\partial^{\beta_{3}}_{\sigma}a(x,y,\eta,\omega,\sigma)|\leq C_{\beta}(1+|\eta|)^{m-|\beta_1|-\frac 23|\beta_2|-\frac 13|\beta_3|}.
\]

\end{definition}
We now recall the Airy function, defined for $z\in \R$ as the oscillatory integral
\begin{equation}
  \label{eq:15}
  Ai(-z)=\int e^{i(\frac{\sigma^{3}} 3-\sigma z)}\,d\sigma\,.
\end{equation}
The choice of $\sigma$ as an integration variable is consistent with our later use of oscillatory integrals with related phases and with symbols within the class we just defined.

In the following we recall the Melrose-Taylor construction of a parametrix for the Dirichlet problem near a glancing point, and then we use its properties to obtain a similar form for the solution to the Neumann problem.

\subsection{The Dirichlet quasimode and its phase functions}
\subsubsection{A quasimode for the Dirichlet boundary problem }

Constructing a parametrix near glancing or gliding rays for the
Dirichlet boundary problem has a long history, starting with
Andersson-Melrose \cite{AndMel77} and Eskin \cite{esk77}. We also
refer to Melrose-Taylor \cite{meta} and the references therein, and
to Zworski \cite{Zwo} for the exterior problem. We recall below the
result that will be used throughout this paper. It provides the
Dirichlet quasimode near glancing and, in particular, the phase
functions that will also enter the Neumann construction.

\begin{theoreme}\label{thmMelrose}[\cite{meta}]
Let ${\qtau}(\omega,\eta)$ be defined in \eqref{eq:LapM}.  There exist a neighborhood $U$ of $(x,y,\eta,\omega)=(0,0,1,0)$, phase functions $\psi(x,y,\eta,\omega)$ and
  $\zeta(x,y,\eta,\omega)$, symbols $p_0(x,y,\eta,\omega)$ and
  $p_1(x,y,\eta,\omega)$ and a function $e_0(x,y,\eta,\omega)$  such that
  \begin{itemize}[leftmargin=5.5mm]
  \item the function $\psi$ is homogeneous of degree $1$, $(\nabla_y(\partial_{\eta_j}\psi))_{j=1,\cdots,d-1}$ are linearly independent;
  \item the function $\zeta$ is homogeneous of degree $2/3$, and
    \begin{equation}
      \label{eq:16}
      \zeta(x,y,\eta,\omega)=\omega-xq(\eta)^{1/3} e_0(x,y,\eta/|\eta|,\omega/q(\eta)^{1/3})\,,
    \end{equation}
i.e. $e_{0}$ is homogeneous of degree $0$;
\item the symbols    $p_0$, $p_1$ (which do not depend on $\sigma$) 
are of order $0$ and type $((1,2/3,1/3),0)$;
    \item the phase functions $\psi$ and $\zeta$ are solutions to the following eikonal equations 
    \begin{equation} \label{systeikeq} 
<\nabla_{(x,y)}\psi,\nabla_{(x,y)}\psi>+\zeta<\nabla_{(x,y)}\zeta,\nabla_{(x,y)}\zeta>={\qtau}^2(\omega,\eta),\quad    <\nabla_{(x,y)}\psi,\nabla_{(x,y)}\zeta>=0.
 \end{equation}
 Here $<.,.>$ is the symmetric bilinear form obtained by polarization of the principal symbol $\xi^2+R(x,y,\eta)$ of the operator $\Delta$ (which is a second order homogeneous polynomial).
\item Define the function $G_D(x,y,\eta,\omega)$ to be
  \begin{equation}
    \label{eq:defG}
    G_D(x,y,\eta,\omega)=e^{i\psi(\cdot)}\left(p_0(\cdot) Ai(-\zeta(\cdot))+ i p_1(\cdot ) q^{-1/6}(\eta) Ai'(-\zeta(\cdot))\right)\Big|_{(\cdot)=(x,y,\eta,\omega)}\,.
  \end{equation}
  Then, for $\Delta$ given in \eqref{eq:laptilde}, the following equation holds in $U$, 
  \begin{equation}
  \label{eq:eqG}
  -\Delta G_D={\qtau}^2 G_D+O_{C^\infty}({\qtau}^{-\infty})\,,
\end{equation}
with $p_0$, $e_0$ elliptic symbols, $e_0>0$ near any $(0,0,\eta,0)$ with $\eta\in \R^{d-1}\setminus\{0\}$ and $p_{1}=0$ on $\{x=0\}$. 
  \end{itemize}
\end{theoreme}

\begin{rmq}\label{rmqzeta}
Constructing an asymptotic solution to equation \eqref{eq:eqG} of the form \eqref{eq:defG} is a classical result in geometrical optics. However, that such a solution can be constructed with $\zeta\vert_{x=0}=\omega$ independent of $(y,\eta)$ is delicate and is a key point of the result (for both the Dirichlet or Neumann cases). Moreover, that the construction can be done such that the symbol $p_{1}$ in front of $Ai'$ in \eqref{eq:defG} vanishes on the boundary $\{x=0\}$ is not obvious and proved in \cite[Paragraph. 4.4, formula 4.4.6 and paragraph 7.1]{meta}. This last condition allows to obtain a parametrix for Dirichlet. Indeed, since $\zeta|_{x=0}=\omega$ and $p_1|_{x=0}=0$, one has
\[
G_D|_{x=0}=e^{i\psi}p_0Ai(-\omega).
\]
Thus, choosing $\omega$ among the zeros of $Ai(-\cdot)$ yields the
Dirichlet boundary condition.
\end{rmq}

Theorem \ref{thmMelrose} can be found, for the exterior problem, in
\cite{Zwo}, and for the interior problem in \cite{meta}.
Another proof of Theorem~\ref{thmMelrose} was given in
\cite[Sect.~2]{ILLP}, based on Melrose's equivalence theorem for
glancing hypersurfaces \cite{mel76}. We briefly recall this
construction below, as both the resulting canonical transformation
and the associated phase functions will also be used in the Neumann
case.

\subsubsection{The glancing canonical transformation and an oscillatory representation }

Theorem~\ref{thmMelrose} gives the quasimode in Airy normal form.
For the dispersive analysis, however, we need a representation in
which the phase is expressed explicitly in terms of the geometry of
the domain and can be compared with the corresponding phase in the
Friedlander model. This is provided by Melrose's equivalence theorem
for glancing hypersurfaces, together with the explicit description
of the associated canonical transformation obtained in \cite{ILLP}.

Observe that Melrose's classification Theorem for glancing
hypersurfaces (see \cite{mel76}) applies, in the non-homogeneous setting, locally near any point in the set $\Sigma_0$, defined as
\begin{equation}
  \label{eq:18}
\Sigma_{0}  =\{(X_{M},Y_{M},\Xi,\Theta)\,:\,\, X_{M}=0\,,\,\,Y_{M}=0\,,\,\,
\Xi=0\,,\,\,|\Theta|=1\}\,.
\end{equation}
Therefore, there exists a canonical transform $\canonchi_M$ such that, near $\Sigma_0$
\begin{equation}
  \label{eq:melrose}
  \canonchi_M(\{X_{M}=0\})=\{x=0\},\;\;\;\canonchi_{M}(\{\Xi^2+|\Theta|^2+X_{M}q(\Theta)=1\})=\{\xi^2+R(x,y,\eta)=1\}\,.
\end{equation}
The crucial fact that such a canonical transformation $\canonchi_{M}$ may actually be defined in a neighborhood of $\Sigma_{0}$ then follows from the transversality of the Hamiltonian flow with respect to $\Sigma_{0}$. The following proposition has been proved in \cite[Prop.2.4]{ILLP}.
\begin{prop}\label{lemgamma}
The generating function for $\canonchi_M$ may be written as $\varphi_{\Gamma}(x,y,\Xi,\Theta)=x\Xi+y\Theta+\Gamma(x,y,\Xi,\Theta)$, where $\Gamma(0,y,\Xi,\Theta)$ is independent of $\Xi$ (as $\canonchi_M(\{X_{M}=0\})=\{x=0\}$) and
\begin{equation}
  \label{eq:GAB}
\Gamma(x,y,\Xi,\Theta)=B_{\Gamma}(y,\Theta)+xA_{\Gamma}(x,y,\Xi,\Theta).
\end{equation}
The transformation $\canonchi_M$ is such that $\canonchi_M(\partial_{\Xi}\varphi_{\Gamma},\nabla_{\Theta}\varphi_{\Gamma},\Xi,\Theta)=(x,y,\partial_x\varphi_{\Gamma},\nabla_y\varphi_{\Gamma})$, and therefore generated by the following relations:
\begin{equation} \label{genchi} 
\left\{ \begin{array}{l}
X_{M}=x+x\frac{\partial A_{\Gamma}}{\partial \Xi}(x,y,\Xi,\Theta)\,,\,\,Y_{M}=y+\frac{\partial B_{\Gamma}}{\partial \Theta}(y,\Theta)+x\frac{\partial A_{\Gamma}}{\partial \Theta}(x,y,\Xi,\Theta)\\ 
\xi=\Xi+A_{\Gamma}(x,y,\Xi,\Theta)+x\frac{\partial A_{\Gamma}}{\partial x}(x,y,\Xi,\Theta)\,,\,\,
\eta=\Theta +\frac{\partial B_{\Gamma}}{\partial y}(y,\Theta)+x\frac{\partial A_{\Gamma}}{\partial y}(x,y,\Xi,\Theta)\,.
\end{array} \right.
 \end{equation}
There exists an elliptic symbol $p(x,y,\eta,\omega,\sigma)$ of order $0$ and type $((1,2/3,1/3),0)$
with support near $N_0$ (from \eqref{eq:N0}) 
and
\begin{equation}\label{eq:Gosc}
G_D(x,y,\eta,\omega):=\frac{1}{2\pi} e^{-i{\qtau} B_{\Gamma}(0,\eta/{\qtau})}\int e^{i(y\cdot \eta+\frac {s^3} 3 +s(xq^{\frac 13}(\eta)-\omega)+{\qtau}\Gamma(x,y,s q^{\frac 13}(\eta)/{\qtau},\eta/{\qtau}))}p(x,y,\eta,\omega,s)ds
\end{equation}
such that Theorem \ref{thmMelrose} holds with this $G_D$.
\end{prop}
\begin{rmq}
Proposition \ref{lemgamma} is one of the key ingredients in the
parametrix construction of \cite{ILLP}. The precise structure of
the generating function $\Gamma$ will also be needed in the
Neumann analysis. We therefore recall below the part of its
construction from \cite[Sect.~6.2]{ILLP} that will be used later.
\end{rmq}

Set $\Theta=\varrho\vartheta$ with $\varrho =|\Theta|$ near $1$ and $\vartheta=\Theta/|\Theta|$. 
The functions $A_{\Gamma}, B_{\Gamma}$ are to be defined near the glancing set $\mathcal{GL}=\{x=0,\Xi=0,\varrho-1=0\}$ and for $(y,\vartheta)$ near $\{0\}\times \mathbb{S}^{d-1}$. We work with formal Taylor expansions $F$ near $\mathcal{GL}$ such that $F=\sum_{a,b,c} f_{a,b,c}(y,\vartheta){X_{M}}^a(\varrho-1)^b\Xi^c$.
We attribute a degree to each factor $x, \varrho-1, \Xi$: a monomial of the form $x^a(\varrho-1)^b\Xi^c$ is homogeneous of degree $k$ if and only if $c+2(a+b)=k$. 
For such a  formal series $F(x,y,\Xi,\varrho,\vartheta)$, defined near $\mathcal{GL}$, we write $F= \sum_{k\geq 0} F_k$, where $F_k$ is homogeneous of degree $k$; we also write $F\in \mathcal{H}_{\geq j}$ if and only if $F=\sum_{k\geq j} F_k$. Therefore, $F_0=f_0(y,\vartheta)$, $F_1=\Xi f_1(y,\vartheta)$, $
F_2={X_{M}} f_2^0(y,\vartheta)+(\varrho-1)f_2^1(y,\vartheta)+\Xi^2f_2^2(y,\vartheta)$, and so on. Replacing ${X_{M}},\xi,\eta$ by their formulas \eqref{genchi} (as functions of $(x,y,\Xi,\Theta)$) and using that from \eqref{eq:melrose} 
\begin{equation}\label{eqiffchi}
\xi^2+R(x,y,\eta)=1\quad \text{ if and only if }\quad \Xi^2+|\Theta|^2+{X_{M}}q(\Theta)=1,
\end{equation}
(where we notice that there is no $Y_{M}$ in the second equation), we get
\begin{equation}\label{formalseriesBA}
B_{\Gamma}= \sum_{j\geq 0} (\varrho-1)^j B_{2j}(y,\vartheta)\,,\quad
A_{\Gamma}= \sum_{k\geq 1}A_k\,.
\end{equation}
Using the third equation from \eqref{genchi} and $\xi|_{\mathcal{GL}}=0$, we have $A_{0}=0$. We also have $\xi(x,y,\Xi,\Theta)\in \mathcal{H}_{\geq 1}$ and ${X_{M}}(x,y,\Xi,\Theta)\in \mathcal{H}_{\geq 2}$. Moreover, from the proof of Melrose's theorem of equivalence of glancing surfaces \cite{mel76}, if formal series of the form \eqref{formalseriesBA} satisfy \eqref{genchi} and \eqref{eqiffchi}, then there exist $C^{\infty}$ functions $A_{\Gamma}, B_{\Gamma}$ with the same Taylor development near $\mathcal{GL}$ satisfying \eqref{genchi} and \eqref{eqiffchi}.

We consider in \eqref{eqiffchi} homogeneous terms of order $\leq 5$ in the expansion of $\Gamma$:  using ${X_{M}}(x,y,\Xi,\Theta)\in \mathcal{H}_{\geq 2}$, we are to write the explicit form of $A_{\Gamma}$ up to $\mathcal{H}_{j\leq 3}$ and $B_{\Gamma}$ up to $\mathcal{H}_{j\leq 5}$ (that is, the terms of $A_\Gamma$ of degree at most $3$ and those of $B_\Gamma$ of degree at most $5$).  Let
$A_1=\Xi\Lp(y,\vartheta)$, $A_2=\alpha(y,\vartheta)x+\beta(y,\vartheta)(\varrho-1)+\mu(y,\vartheta)\Xi^2$, $A_3=\alpha_1(y,\vartheta)x\Xi+\beta_1(y,\vartheta)(\varrho-1)\Xi+\mu_1(y,\vartheta)\Xi^3$ and
$A_{\Gamma}=A_{1}+A_{2}+A_{3} 
+\mathcal{H}_{j\geq 4}$. We find

\begin{prop}\label{propimpformgamma}(\cite[Prop.6.5]{ILLP})
The phase function $\Gamma(x,\cdot)=B_{\Gamma}+x A_{\Gamma}$ is such that, near the glancing set $\mathcal{GL}$,
\begin{equation}\label{AGam}
\left\{ \begin{array}{l}
A_{\Gamma}(x,y,\Xi,\Theta)= \Xi\Lp(y,\vartheta)+\alpha(y,\vartheta)x+\mu(y,\vartheta)(\Xi^2+|\Theta|^2-1)+\mathcal{H}_{j\geq3}\,,\\
B_{\Gamma}(y,\Theta)=B_0(y,\vartheta)+(\varrho-1)B_2(y,\vartheta)+\mathcal{H}_{j\geq3}\,,
 \end{array} \right.
\end{equation}
where $\vartheta=\Theta/|\Theta|$, $\varrho=|\Theta|$ and $1+\Lp=\Big(\frac{R_1(y,\vartheta+\nabla_yB_0)}{q(\vartheta)}\Big)^{1/3}$, and where $R_{0}$, $R_{1}$ are defined in \eqref{eq:R01}. 
Moreover, 
\[
\beta(y,\vartheta)=2\mu\,, \quad \text{ and } \quad \mu(y,\vartheta)=\frac{\alpha}{q(\vartheta)(1+\Lp)}-\frac{\nabla_{\eta}R_0(y,\vartheta+\nabla_yB_0)}{4q(\vartheta)}\cdot \nabla_y\Big(\frac{1}{1+\Lp}\Big)\,.
\]
We also have $\Lp(0,\vartheta)=0$, $B_0(0,\vartheta)=0$, $\nabla_yB_0(0,\vartheta)=0$, $B_2(0,\vartheta)=0$ and $\nabla_yB_2(0,\vartheta)=0$, $B_0$ (resp. $B_2$) is homogeneous of order $1$ (resp. of order $0$) in its second variable.
\end{prop}

\subsubsection{Equivalence of phase functions for $G_{D}(x,y,\eta,\omega)$}

We finally make explicit the relation between the two representations
\eqref{eq:defG} and \eqref{eq:Gosc} of the same quasimode. The Airy
normal form \eqref{eq:defG} involves the phase
\[
 \psi(x,y,\eta,\omega)+\frac{s^3}{3}
 -s\zeta(x,y,\eta,\omega),
\]
whereas the oscillatory representation \eqref{eq:Gosc}, obtained
from the glancing canonical transformation, involves
\[
 y\cdot\eta+\frac{\sigma^3}{3}
 +\sigma\bigl(xq^{1/3}(\eta)-\omega\bigr)
 +\tau_q\Gamma\left(
 x,y,\frac{\sigma q^{1/3}(\eta)}{\tau_q},
 \frac{\eta}{\tau_q}\right).
\]
These two phase functions parametrize the same Lagrangian. The
following lemma, proved in \cite[Lemma~6.9]{ILLP}, gives an explicit
change of variable between them. This identification will be
particularly useful below, since the Neumann construction has the
same phase functions $\psi$ and $\zeta$.

\begin{lemma}\label{lemphaseG}
Set $\phi(x,y,\theta,\alpha,\sigma)=\sigma^3/3+\sigma(xq^{1/3}(\theta)-\alpha)+x\tau_{q} A_{\Gamma}(x,y,\sigma q^{1/3}(\theta)/\tau_{q},\theta/\tau_{q})$, with $\tau_q=\tau_q(\alpha,\theta)$. There exists a unique smooth non-degenerate change of variable $\sigma\rightarrow s$ and a smooth function $\Upsilon(x,y,\theta,\alpha)\in C^{\infty}$, independent of $s$, such that $\phi(x,y,\theta,\alpha,\sigma)=s^3/3-s\zeta(x,y,\theta,\alpha)+\Upsilon(x,y,\theta,\alpha)$
 and with $\frac{ds}{d\sigma}$ smooth and non-vanishing.
 
 Let $w:=(x,y,\theta,\alpha)$ and denote $\sigma_0(w)$ the unique solution to $\partial^2_{\sigma,\sigma}\phi(w,\sigma)=0$. The two critical points $\sigma_{\pm}(w)$ of $\phi$ correspond under this change of variable to $s_{\pm}(w):=\pm\sqrt{\zeta(w)}$. More precisely, $\sigma_{\pm}(w)=\sigma_0(w)\pm \sqrt{\zeta(w)}k(\pm\sqrt{\zeta(w)},w)$, where $k(u,w)= 1+\sum_{j\geq 1}k_j(w)u^j$, and the coefficients $k_j$ are smooth functions of $w$.
Moreover, 
\begin{equation}\label{formephaseUpsilon}
\frac 23 \zeta^{3/2}(w)=\frac 12\Big(\phi(w,\sigma_-(w))-\phi(w,\sigma_+(w))\Big),\quad \Upsilon(w):=\frac 12\Big(\phi(w,\sigma_+(w))+\phi(w,\sigma_-(w))\Big)\,.
\end{equation}
\end{lemma}
As a consequence, $\psi$ in the Airy representation can be recovered explicitly from the generating function $\Gamma$,
 \begin{equation}\label{formpsigamma}
 \psi(x,y,\theta,\alpha)=y\cdot \theta+\tau_{q}(\alpha,\theta)B_{\Gamma}(y,\theta/\tau_{q})+\Upsilon(x,y,\theta,\alpha), \text{ where } \Upsilon \in H_{j\geq 2}.
 \end{equation}

In particular, the passage from the Airy normal form
\eqref{eq:defG} to the oscillatory representation
\eqref{eq:Gosc} depends only on the phase functions $\psi,\zeta$,
or equivalently on the underlying Lagrangian and the generating
function $\Gamma$, and not on the amplitudes $p_0,p_1$. This
observation will be crucial for the Neumann problem: once a Neumann
quasimode has been constructed with the same phase functions
$\psi,\zeta$, the same change of variable yields an oscillatory
representation with exactly the phase appearing in
\eqref{eq:Gosc}; only the symbol is changed.

\subsection{A quasimode for the Neumann problem}

We now turn to the Neumann boundary condition. The phase functions
constructed in Theorem~\ref{thmMelrose} depend only on the eikonal
equations and on the geometry of the glancing hypersurfaces, and
therefore do not depend on the choice of boundary condition. The
Neumann construction can consequently be carried out with exactly
the same phase functions $\psi$ and $\zeta$ as in the Dirichlet
case. The new point is to construct amplitudes solving the
corresponding transport equations to all orders and satisfying the
boundary relation required by the Neumann condition.

The following theorem is the main new ingredient in the parametrix
construction.
\begin{thm}\label{thm:N}
Let $\tau_q(\omega,\eta)$ be defined in \eqref{eq:LapM}. There exist
a neighborhood $U$ of $(x,y,\eta,\omega)=(0,0,1,0)$ and symbols
$g_0(x,y,\eta,\omega)$ and $g_1(x,y,\eta,\omega)$ such that the
following properties hold.
The phase functions $\psi,\zeta$ and the elliptic function $e_0$
are exactly those of Theorem~\ref{thmMelrose}; in particular, they
satisfy all the conclusions of that theorem concerning homogeneity,
non-degeneracy and the eikonal equations, and
\begin{equation}\label{zeta}
 \zeta(x,y,\eta,\omega)
 =
 \omega-xq(\eta)^{1/3}
 e_0(x,y,\eta/|\eta|,
          \omega/q(\eta)^{1/3}).
\end{equation}
The symbols $g_0,g_1$ have the same orders and types as $p_0,p_1$
in Theorem~\ref{thmMelrose}, with $g_0$ elliptic, and satisfy on
$\{x=0\}$
\begin{equation}\label{eq:Neumann-symbol-boundary}
 \partial_xg_0
 +i q(\eta)^{-1/6}
   \zeta\,\partial_x\zeta\,g_1
 =0.
\end{equation}
Define
\begin{equation}\label{eq:GN}
 G_N(x,y,\eta,\omega)
 =
 e^{i\psi(\cdot)}
 \left(
 g_0(\cdot)Ai(-\zeta(\cdot))
 +
 i g_1(\cdot)q(\eta)^{-1/6}
 Ai'(-\zeta(\cdot))
 \right)\Big|_{(\cdot)=(x,y,\eta,\omega)}.
\end{equation}
Then, for $\Delta$ given in \eqref{eq:laptilde},
\begin{equation}\label{eq:GN-quasimode}
 -\Delta G_N
 =
 \tau_q^2 G_N
 +
 O_{C^\infty}(\tau_q^{-\infty})
 \qquad\text{in }U.
\end{equation}
Moreover, on $\{x=0\}$,
\begin{equation}\label{eq:dGN-boundary}
 \partial_xG_N
 =
 e^{i\psi} b_N Ai'(-\omega),
 \qquad
 b_N=
 -g_0\partial_x\zeta
 +i q(\eta)^{-1/6}\partial_xg_1.
\end{equation}
Consequently, if $\omega=\widetilde\omega_k$, where
\[
 Ai'(-\widetilde\omega_k)=0,
\]
then
\[
 \partial_xG_N|_{x=0}=0 \text{ modulo } O_{C^\infty}(\tau_q^{-\infty}).
 \]
\end{thm}
\begin{rmq}\label{rmq:Neumann-new}
The essential point in Theorem~\ref{thm:N} is that the Neumann
quasimode can be constructed with \emph{exactly the same phase
functions} $\psi$ and $\zeta$ as in the Dirichlet construction.
Thus all the geometric information encoded by the glancing
canonical transformation, and in particular the results recalled
in the previous subsection, remains unchanged. The new content is
entirely in the amplitudes: one has to solve the transport equations
to all orders while imposing
\eqref{eq:Neumann-symbol-boundary} and preserving the ellipticity of
$g_0$.
\end{rmq}
Let us first explain why \eqref{eq:Neumann-symbol-boundary} is the
appropriate boundary relation. Differentiating \eqref{eq:GN} with
respect to $x$ and using $Ai''(z)=zAi(z)$, the coefficient of
$Ai(-\zeta)$ in $\partial_xG_N$ is
\[
 \partial_xg_0+i(\partial_x\psi)g_0
 +i q(\eta)^{-1/6}
   \zeta\,\partial_x\zeta\,g_1.
\]
Since $\zeta|_{x=0}=\omega$, its tangential derivatives vanish on
the boundary. The eikonal equation
\[
 \langle\nabla\psi,\nabla\zeta\rangle=0
\]
then gives
\[
 \partial_x\psi|_{x=0}=0,
\]
because $\partial_x\zeta$ is elliptic there. Hence
\eqref{eq:Neumann-symbol-boundary} cancels the $Ai(-\zeta)$
component of the normal derivative. The remaining term is precisely
\eqref{eq:dGN-boundary}. Thus the zeros of $Ai'$ play for the
Neumann problem the role played by the zeros of $Ai$ in the
Dirichlet problem. In the next subsection we use the quasimode $G_N$ to construct a
solution to \eqref{WE} microlocally near a glancing point.

\begin{rmq}\label{rmq:GN-osc}
Once Theorem~\ref{thm:N} is established, the oscillatory
representation of the Neumann quasimode follows directly from
Lemma~\ref{lemphaseG}. Indeed, the change of variable in that lemma
depends only on the phase functions $\psi,\zeta$, which are the
same for $G_D$ and $G_N$, and not on their amplitudes. Absorbing
the Jacobian of this change of variable and the amplitudes
$g_0,g_1$ into a new symbol, one obtains
\begin{equation}\label{eq:GN-oscillatory}
 G_N(x,y,\eta,\omega)
 =
 \frac{1}{2\pi}
 e^{-i\tau_q B_\Gamma(0,\eta/\tau_q)}
 \int
 e^{\,i\left(
 y\cdot\eta+\frac{\sigma^3}{3}
 +\sigma(xq^{1/3}(\eta)-\omega)
 +\tau_q\Gamma\left(
 x,y,\frac{\sigma q^{1/3}(\eta)}{\tau_q},
 \frac{\eta}{\tau_q}\right)\right)}
 g(x,y,\eta,\omega,\sigma)\,d\sigma ,
\end{equation}
where $g$ has the same symbolic order and type as the symbol
$p$ in \eqref{eq:Gosc}. Thus the Dirichlet and Neumann quasi modes
are associated with the same Lagrangian and have oscillatory
representations with exactly the same phase; the boundary condition
affects only the symbol.
\end{rmq}

\begin{proof}
The construction follows the same geometric optics scheme as in the
Dirichlet case, the only difference being the boundary condition
imposed on the amplitudes. We first explain this point before turning
to the result of Melrose-Taylor which will allow us to solve the
corresponding boundary problem for the transport equations.

As in the proof of Theorem \ref{thmMelrose}, we look for a quasimode as an
oscillatory integral associated with the glancing canonical relation.
Equivalently, after integration in the Airy variable, it takes the
form
\[
 G_N(x,y,\eta,\omega)
 =
 e^{i\psi(x,y,\eta,\omega)}
 \left(
 g_0(x,y,\eta,\omega)Ai(-\zeta(x,y,\eta,\omega))
 +
 i q(\eta)^{-1/6}
 g_1(x,y,\eta,\omega)Ai'(-\zeta(x,y,\eta,\omega))
 \right).
\]
The phase of the underlying oscillatory integral is determined by the
principal symbol of the wave operator and by the glancing canonical
relation. It is therefore completely independent of the boundary
condition. In particular, the phase functions $\psi$ and $\zeta$ are
exactly the same as in the Dirichlet construction of
Theorem \ref{thmMelrose}. They satisfy
\[
 \langle \nabla\psi,\nabla\psi\rangle
 +
 \zeta\langle\nabla\zeta,\nabla\zeta\rangle
 =
 \tau_q^2,
 \qquad
 \langle\nabla\psi,\nabla\zeta\rangle=0,
\]
together with $\zeta|_{x=0}=\omega$.
Consequently, all the results used in \cite{ILLP} concerning the
construction and the properties of $\psi$ and $\zeta$ apply
unchanged. Once the eikonal equations have been solved, inserting the above
ansatz into
\[
 (-\Delta-\tau_q^2)G_N=O_{C^{\infty}}(\tau_q^{-\infty})
\]
and identifying successively the coefficients of the Airy function
and of its derivative gives a hierarchy of transport equations for
the amplitudes. These equations are local equations in the interior
and depend only on the differential operator and on the already
constructed phase functions. Hence they are again exactly the same
transport equations as in the Dirichlet construction.
The distinction between the Dirichlet and Neumann problems appears
only when one has to select, among the solutions of this system of
transport equations, those satisfying the required condition on
$\{x=0\}$. In the Dirichlet case, Melrose and Taylor construct the
symbols $p_0,p_1$ of $G_D$ in such a way that
$p_1|_{x=0}=0$,
with $p_0$ elliptic. Since $\zeta|_{x=0}=\omega$, this implies that
the trace on the boundary of the resulting quasimode is proportional to
$Ai(-\omega)$ (the coefficient of $Ai'(-\zeta)$ cancels), and choosing $\omega$ among the zeros of $Ai$ yields the Dirichlet boundary condition.

For the Neumann problem we must choose a different solution of the
same transport system. Indeed, differentiating $G_N$ with respect to
$x$, and using $Ai''(-\zeta)=-\zeta Ai(-\zeta)$
shows that the coefficient of $Ai(-\zeta)$ in $\partial_xG_N$ is
\begin{equation}\label{Ncond}
 \partial_x g_0
 +i(\partial_x\psi)g_0
 +i q(\eta)^{-1/6}\zeta\,\partial_x\zeta\,g_1.
\end{equation}
On the boundary, the second eikonal equation and the identity
$\zeta|_{x=0}=\omega$ implies
$\partial_x\psi|_{x=0}=0$.
Thus, in order that the boundary trace of $\partial_xG_N$ contain
only an $Ai'(-\omega)$ component, we have to impose
\begin{equation}
\label{eq:Neumann-boundary-amplitudes}
 \left.
 \left(
 \partial_xg_0
 +i q(\eta)^{-1/6}
   \zeta\,\partial_x\zeta\,g_1
 \right)\right|_{x=0}=0.
\end{equation}
Under this condition,
\begin{equation}\label{condN}
 \partial_xG_N|_{x=0}
 =
 e^{i\psi}
 \left(
 -g_0\partial_x\zeta
 +i q(\eta)^{-1/6}\partial_xg_1
 \right)_{x=0}
 Ai'(-\omega),
\end{equation}
and choosing $\omega$ to be a zero of $Ai'(-\cdot)$ gives the Neumann
 boundary condition.
We are therefore reduced to the following problem: solve exactly the
same hierarchy of transport equations as in the Dirichlet case, but
replace the boundary condition
$ p_1|_{x=0}=0$
by \eqref{eq:Neumann-boundary-amplitudes}, while preserving the
symbolic properties of the amplitudes and, in particular, the
ellipticity of the leading coefficient $g_0$.
At this point we use a further result of Melrose-Taylor. The
construction leading to the condition $p_1|_{x=0}=0$ is in fact part
of a more general result: their method allows one to prescribe a
larger class of boundary relations between the two amplitudes (and
their boundary derivatives), while solving the same transport system
to all orders. The Dirichlet choice $p_1|_{x=0}=0$ is only one
particular instance of this general construction. We shall apply
their result with the boundary operator chosen so that the resulting
amplitudes $g_0,g_1$ satisfy
\eqref{eq:Neumann-boundary-amplitudes}.

We now recall the argument of Melrose-Taylor which allows one to
solve the transport equations with more general boundary conditions that will be
needed for this purpose.
In order to avoid any confusion with the symbols $p_0,p_1$ of the
Dirichlet quasimode and $g_0,g_1$ which will be used below for the
Neumann problem, we denote here the two amplitudes by $q_0,q_1$ and
write
\[
 q_0\sim\sum_{j\geq0}i^j q_{0,j}|\eta|^{-j},
 \qquad
 q_1\sim |\eta|^{-1/3}
             \sum_{j\geq0}i^j q_{1,j}|\eta|^{-j},
\]
where all the coefficients $q_{0,j}$ and $q_{1,j}$ are real valued symbols of
order zero.
We consider
\[
 {G}
 =
 e^{i\psi}
 \left(
 q_0 Ai(-\zeta)+i q_1 Ai'(-\zeta)
 \right).
\]
Once the eikonal equations for $\psi$ and $\zeta$ are satisfied,
substitution into the equation for ${G}$ gives a coupled system
of transport equations for $q_0$ and $q_1$. With our convention
$Ai(-\zeta)$, the equations for the leading terms
$q_{0,0},q_{1,0}$ are
\begin{equation}
\label{eq:transport0-first}
\begin{split}
 2\langle \nabla\psi,\nabla q_{0,0}\rangle
 +2\zeta\langle \nabla\zeta,\nabla q_{1,0}\rangle
 -\langle \nabla\zeta,\nabla\zeta\rangle q_{1,0}
 +(\Delta^b\psi)q_{0,0}
 -\zeta(\Delta^b\zeta)q_{1,0}
 &=0,
\\
 2\langle \nabla\zeta,\nabla q_{0,0}\rangle
 +2\langle \nabla \psi,\nabla q_{1,0}\rangle
 +(\Delta^b\zeta)q_{0,0}
 +(\Delta^b\psi)q_{1,0}
 &=0.
\end{split}
\end{equation}
Here, as above, $\langle\cdot,\cdot\rangle$ denotes the symmetric
bilinear form obtained by polarization of the principal symbol and
$\Delta^b$ denotes the operator obtained by removing its zeroth order
term.
At the successive orders $j\geq1$, the same system has forcing terms
which only involve the coefficients already constructed:
\begin{equation}
\label{eq:transportj}
\begin{split}
 2\langle \nabla \psi,\nabla q_{0,j}\rangle
 +2\zeta\langle \nabla\zeta,\nabla q_{1,j}\rangle
 -\langle \nabla \zeta,\nabla\zeta\rangle q_{1,j}
 +(\Delta^b\psi)q_{0,j}
 -\zeta(\Delta^b\zeta)q_{1,j}
 &=
\Delta q_{0,j-1},
\\
 2\langle \nabla \zeta,\nabla q_{0,j}\rangle
 +2\langle \nabla \psi,\nabla q_{1,j}\rangle
 +(\Delta^b\zeta)q_{0,j}
 +(\Delta^b\psi)q_{1,j}
 &=
 \Delta q_{1,j-1}.
\end{split}
\end{equation}
Thus, at each step, one has to solve the same first order system,
with a source term completely determined by the preceding step.
For the Dirichlet construction, Melrose-Taylor solve this system with the boundary condition
\[
 q_1=0\qquad\text{on }x=0.
\]
Their construction in fact applies to a more general affine boundary
relation which is required for Neumann. 
We shall use the following consequence of \cite[Prop.~4.4.11]{meta}:
\begin{prop}\label{propTEmeta}(\cite[Prop. 4.4.11]{meta}; for details see \cite[Sections 2,4]{meta})
Near a glancing point, the transport system
\eqref{eq:transport0-first}--\eqref{eq:transportj} admits smooth
solutions satisfying at each step an affine boundary condition of
the form
\begin{equation}
\label{eq:general-MT-boundary}
        q_{1,j}|_{x=0}
        =
        c_j q_{0,j}|_{x=0}+d_j,
\end{equation}
where \(c_j\) and \(d_j\) are real-valued smooth symbols of the
appropriate orders. The solution may moreover be normalized by
prescribing the value of its leading component at the glancing
point.
\end{prop} 

This result, whose idea of proof is briefly recalled in the next section, is exactly what is needed for the Neumann
construction. Indeed, write the symbols $g_0,g_1$ of the
Neumann quasimode under the following form
\begin{equation}\label{symb}
 g_0\sim_{|\eta|^{-1}}\sum_{j\geq0}i^j g_{0,j}|\eta|^{-j},
 \qquad
  g_1\sim_{|\eta|^{-1}}
             \sum_{j\geq0}i^j g_{1,j}|\eta|^{-j}.
\end{equation}
The transport equations are precisely
\eqref{eq:transport0-first}-\eqref{eq:transportj}, with
$q_{0,j},q_{1,j}$ replaced by $g_{0,j},q(\eta)^{-1/6} g_{1,j}$. Notice that, as $\psi$, $\zeta$ are homogeneous of order $1$ and $2/3$, respectively, $d\psi$ yields a factor $|\eta|$ while $\zeta$ and $d\zeta$ yield, each one, a factor $q^{1/3}(\eta)$, which motivates the form of these asymptotic expansions. The boundary condition that we require is \eqref{Ncond}, where $\zeta$ is of the form \eqref{zeta}, homogeneous of order $2/3$. It follows that

\begin{align}\label{coeffq}
q^{-1/6}(\eta)(\zeta\partial_x\zeta)|_{x=0} &=-q^{-1/6}(\eta)\times q^{1/3+1/3}(\eta)\Big(\omega/ q^{1/3}(\eta)\Big)e_0(0,y,\eta/|\eta|, \omega/q(\eta)^{1/3})\\
&=-|\eta|\times q^{1/2}(\eta/|\eta|)\times \Big(\omega/ q^{1/3}(\eta)\Big)e_0(0,y,\eta/|\eta|, \omega/q^{1/3}(\eta))\\
&=- |\eta| E_0(y, \eta/|\eta|, \omega/q^{1/3}(\eta)),
\end{align}
where we set $E_0:= q^{1/2}(\eta/|\eta|)\times \Big(\omega/ q^{1/3}(\eta)\Big)e_0(0,y,\eta/|\eta|, \omega/q^{1/3}(\eta))$.

We use \cite[Prop. 4.4.11]{meta} recursively to construct $g_{0,j}(y,\eta/|\eta|, \omega/q^{1/3}(\eta))$ and $g_{1,j}(y,\eta/|\eta|, \omega/q^{1/3}(\eta))$. First, we can construct $g_{0,0}$ and $g_{1,0}$ exactly as in the Dirichlet case as they solve \eqref{eq:transport0-first} and can be chose such that $g_{0,0}(0,0,1,0)=1$ and $g_{1,0}|_{x=0}=0$. Then,
at each step, once
$g_{0,j}$ has been constructed, the contribution
$\partial_xg_{0,j}$ has to be cancelled in the coefficient of
$Ai(-\zeta)$ in the boundary trace of $\partial_xG_N$. Indeed, using \eqref{eq:Neumann-boundary-amplitudes} and \eqref{coeffq} we must have
\[
\left(\partial_x \Big( g_{0,0}+ \sum_{j\geq 1} (i/|\eta|)^j g_{0,j}\Big)-i|\eta|E_0\Big(g_{1,0}+ \sum_{j\geq 1} (i/|\eta|)^j g_{1,j}\Big)\right)\Big|_{x=0}=0,
\]
where we already have $g_{1,0}|_{x=0}=0$.
We therefore
prescribe the boundary value of the next coefficient by
\begin{equation}
\label{eq:Neumann-recursion}
 g_{1,j+1}|_{x=0}
 =
 -\frac{\partial_xg_{0,j}|_{x=0}}
 {E_0(y, \eta/|\eta|, \omega/q^{1/3}(\eta))}.
\end{equation}
This is exactly a boundary condition of the form
\eqref{eq:general-MT-boundary}, with
\[
 c_{j+1}=0,
 \qquad
 d_{j+1}=
 -\frac{\partial_xg_{0,j}|_{x=0}}
 {E_0(y, \eta/|\eta|, \omega/q^{1/3}(\eta))}.
\]
Proposition \cite[Prop. 4.4. 11]{meta} therefore provides the solution at the next step
of the transport construction.
Notice that the quotient in \eqref{eq:Neumann-recursion} is
well-defined in the region under consideration. Indeed, on the
boundary, $\zeta|_{x=0}=\omega>0$, 
while, from the explicit form \eqref{zeta} of
$ \zeta$ we deduce that 
$\partial_x\zeta|_{x=0}$ is elliptic. The vanishing of $\zeta$
occurs away from the boundary, at $x\sim a>0$, and hence plays no
role in prescribing the boundary data.
Iterating this procedure gives the full asymptotic expansions of
$g_0$ and $g_1$, satisfying the transport equations to all orders
and the required Neumann boundary relation. Moreover, the
normalization in \cite[Prop. 4.4.11]{meta} gives a non-vanishing leading
coefficient $g_{0,0}$ at the base point (which equals  the corresponding $p_{0,0}$ from the decomposition of $p_0\sim_{|\eta|^{-1}}\sum_{j\geq 0} i^j p_{0,j}|\eta|^{-j}$ from the Dirichlet case). After possibly shrinking
the microlocal neighborhood, $g_0$ is therefore elliptic.
\end{proof}

\subsubsection{The boundary condition - Sketch of the proof of Proposition \ref{propTEmeta}
(see \cite[Sections 2,4]{meta})}

We briefly recall the geometric argument of Melrose-Taylor which
allows one to solve the transport equations while imposing a boundary
condition. This will also clarify why the boundary condition may be
replaced by the more general affine condition used below.
Introducing, away from the glancing set, the two combinations
\begin{equation}\label{transp}
        a^\pm=q_0\pm(-\zeta)^{1/2}q_1,
\end{equation}
the coupled transport system is diagonalized into the transport
equations associated with the two branches of the characteristic
variety. 
Let \(\Sigma=\{(x,y,\xi,\eta), \xi^2+R(x,y,\eta)=1\}\) denote the characteristic hypersurface of the principal
symbol, and let \(Q=\{(x,y,\xi,\eta), x=0,\eta\neq 0\}\) denote the bounding hypersurface. Thus \(\Sigma\cap Q\) is the characteristic set over the
boundary, and the glancing set $\mathcal{GL}$ is the subset of \(\Sigma\cap Q\) where the
Hamilton vector field \(H_p\) is tangent to \(Q\), that is $\mathcal{GL}=\{(x,y,\xi,\eta), \xi^2+R(x,y,\eta)=1, x=0, H_p x=2\xi=0\}$. Equivalently, if \(q:=x=0\) is a defining function for \(Q\), the
glancing condition is $H_px=0$.
The strict convexity assumption  \eqref{glancingcond} implies the corresponding simple
glancing conditions of Melrose-Taylor, namely that the tangency of
the Hamilton flow of \(\Sigma\) to \(Q\) is of second order and that the
dual tangency condition for the Hamilton flow of \(q\) to \(\Sigma\)
is also non-degenerate.

After lifting the
transport equation to \(\Sigma\), \eqref{transp} has the form
$H_p a+F_1a=F_2$.
By solving successively along the Hamilton flow, one first removes the
inhomogeneous term and then the zeroth order term. The equation is
therefore reduced to
$H_p b=0$.                                       
Hence \(b\) is constant on the bicharacteristics of \(H_p\).
The point is now to choose such a solution so that its restriction to
\(\Sigma\cap Q\) satisfies the prescribed boundary condition. 

Near the
glancing set, two natural two-sheeted projections are defined on
\(\Sigma\cap Q\).
The first one is simply the projection onto the cotangent bundle of the boundary,
\[
        \pi_Q:\Sigma\cap Q\longrightarrow T^*\partial\Omega, \qquad
        (0,y,\xi,\eta)\longmapsto(y,\eta).
\]
Its two sheets can be seen explicitly in the present coordinates.
Indeed, on \(\Sigma\cap Q\) the characteristic equation reads
$\xi^2=1-R(0,y,\eta)$. Thus, on the hyperbolic side of the glancing set, a fixed boundary
covector \((y,\eta)\) has two lifts to \(\Sigma\cap Q\),
\[
        \rho_\pm(y,\eta)
        =
        \left(0,y,\pm\sqrt{1-R(0,y,\eta)},\eta\right).
\]
The corresponding involution \(I\) exchanges these two lifts; in the
present coordinates it is simply given by
\begin{equation}\label{Iinv}
        I(0,y,\xi,\eta)=(0,y,-\xi,\eta),
\end{equation}
and satisfies $\pi_Q\circ I=\pi_Q$.
At glancing, \(H_px=2\xi=0\), hence \(\xi=0\), and the two lifts
\(\rho_+\) and \(\rho_-\) coalesce. Consequently, the glancing set
is precisely the fixed point set of \(I\). In these coordinates, \(\xi\) is \(I\)-odd, and
the fixed point set of $I$ is characterized by $\xi=0$, which is precisely the glancing set.\\

The second involution is associated with the Hamilton flow. Let
\(\Phi_s=\exp(sH_p)\) denote the bicharacteristic flow on \(\Sigma\).
For \(\rho\in\Sigma\cap Q\) sufficiently close to the glancing set,
let \(s(\rho)\neq0\) denote the second small solution of $x\bigl(\Phi_s(\rho)\bigr)=0$.
The second involution is then given by
\[
        J(\rho)=\Phi_{s(\rho)}(\rho).
\]
Thus \(J\) exchanges the two points of \(\Sigma\cap Q\) which lie on
the same bicharacteristic, or equivalently the two points having the
same image under the local projection of \(\Sigma\cap Q\) onto the
space of bicharacteristics.
To see the local geometry, set $x(s)=x\bigl(\Phi_s(\rho)\bigr)$.
Since \(\rho\in Q\), \(x(0)=0\), while
$ x'(0)=H_px(\rho)=2\xi$.
At glancing \(\xi=0\), and the simple glancing condition gives
\(H_p^2x\neq0\). Hence the two intersections of the bicharacteristic
with \(Q\) coalesce at the glancing set, which is the fixed point set
of \(J\).
Since \(H_pb=0\), \(b\) is constant along the Hamilton flow. Therefore
\[
        b(J(\rho))=b(\rho),\text { or equivalently }
        J^*b=b.
\]
Thus \(J\)-invariance is precisely the condition on the boundary
values imposed by the transport equation.\\

The boundary condition is naturally expressed instead in terms of
the first involution \(I\), since \(I\) exchanges the two lifts of the
same boundary covector. Let \(\tau\) be an \(I\)-odd local coordinate
transverse to the glancing set, so that
\[
        \mathcal{GL}=\{\tau=0\},
        \qquad I^*\tau=-\tau.
\]
In the coordinates used above, one may take \(\tau=\xi\), since \eqref{Iinv} holds.
We keep the notation \(\tau\) in order to emphasize that the argument
is intrinsic and does not depend on this particular choice of
coordinates.
Every smooth function \(b\) on \(\Sigma\cap Q\) can then be decomposed
uniquely as
\[
        b=b_E+\tau b_O,
\]
where \(b_E\) and \(b_O\) are both \(I\)-even. Indeed,
\[
        b_E=\frac12(b+I^*b),
        \qquad
        \tau b_O=\frac12(b-I^*b).
\]
The numerator in the second expression is \(I\)-odd and therefore
vanishes on the fixed point set \(\{\tau=0\}\). It is consequently
divisible by \(\tau\), and the quotient \(b_O\) is smooth and
\(I\)-even. Thus \(b_E\) is the even part of \(b\), whereas
\(\tau b_O\) is its odd part. The original boundary condition couples the two branches of the
transport system over the same boundary covector. After the preceding
reductions, these two branches correspond precisely to the two points
interchanged by \(I\). Hence the boundary condition can be expressed
in terms of the \(I\)-even and \(I\)-odd parts of \(b\). The affine
boundary relation considered by Melrose-Taylor takes, after these
reductions, the form
\[
        b_O=f_1 b_E+f_2,
\]
where \(f_1\) and \(f_2\) are smooth \(I\)-even functions. 
We have therefore reduced the transport problem with boundary
condition to the following compatibility problem on
\(\Sigma\cap Q\):
\begin{equation}\label{TE}
        J^*b=b,
        \qquad
        b_O=f_1b_E+f_2.
\end{equation}
The two equations have different origins. The first identifies the
values at the two boundary points belonging to the same
bicharacteristic and is imposed by the transport equation
\(H_pb=0\). The second relates the values on the two characteristic
sheets lying over the same boundary covector and is imposed by the
boundary condition.

Melrose-Taylor construct a simultaneous local normal form for the
two involutions \(I\) and \(J\). 
In these normal coordinates,
\cite[Prop.~2.8.2]{meta} shows that, for arbitrary smooth
\(I\)-even functions \(f_1,f_2\), one can solve simultaneously \eqref{TE}
with the prescribed normalization at the glancing point.
Extending this boundary value
along the bicharacteristic flow then gives the required solution of
the transport equation.
This is the mechanism underlying \cite[Prop.~4.4.11]{meta} of
Melrose-Taylor. In particular, replacing the homogeneous boundary
condition used in the Dirichlet construction by an affine relation
between the two components only changes the even functions \(f_1\) and
\(f_2\) in $b_O=f_1\,b_E+f_2$; the normal form argument itself is unchanged.

\subsection{A parametrix for the wave equation with Neumann boundary condition}

Let $a_0>0$ be small and $a\in (0,a_0]$; denote by $\mathcal{G}(t,x,y,a)$ the
Green function for the wave equation with Neumann boundary condition, and $\delta_{(a,0)}:=\delta_{x=a,y=0}$ the source point,
\begin{equation}
\label{waveeq}
(  \partial_t^2 -\Delta) \Grond=0 \,,\text{ for }\, x>0\,,
 \, \partial_x\Grond_{{\textstyle |}x=0}=0, \Grond|_{t=0}= \delta_{(a,0)} \text{ and } \partial_t \Grond |_{t=0}=0.
\end{equation}
We will frequently need smooth cut-off functions $\varkappa\geq 0$ in $ C_{0}^{\infty}(\R^{m})$ with $m=1$ or with $m=d-1$. For $m=1$, $\varkappa$ will be such that $\varkappa=1$ near $1$, $\varkappa=0$ outside a small neighborhood of $1$, and for $m=d-1$, $\varkappa$ will be radial and such that ${\varkappa}=1$ near $\mathbb{S}^{m-1}$, ${\varkappa}=0$ outside a small neighborhood of $\mathbb{S}^{m-1}$. We will abuse notations and retain $\varkappa$ as a generic notation, irrespective of the value of $m$ (which will be clear from context) as well as the size of the (small) support of $\varkappa$, which we assume from now on to be smaller than $0<\ceps_{0}<1/100$.
\begin{dfn}\label{dfnparametrix}
  Let $h\in (0,1)$. A function $\mathcal{P}_{h,a}(t,x,y)$ is a parametrix for \eqref{waveeq} if and only if
 there exists $a_0>0$, $r>0$ and a neighborhood $V$ of $(t,y)=(0,0)$ such that for all $\alpha$ one has  
 \[
\sup_{0<a\leq a_0}\sup_{0<x\leq r}\sup_{(t,y)\in V} \Big|\partial^{\alpha}_{t,x,y}({\varkappa}(hD_{t}){\varkappa}(hD_{y})(\mathcal{P}_{h,a}-\Grond(\cdot,a))\Big| \in O(h^{\infty}).
 \]
\end{dfn}
\begin{rmq}
The operator $\varkappa(hD_{t})$ is really a spectral localization with respect to $\Delta$, if applied to a solution to the wave equation. The operator $\varkappa(hD_{y})$ further restricts this localization to spatial frequencies whose dominant part is tangential: the general heuristic is that waves propagating along the boundary are the most dangerous ones, whereas other waves are transverse and can be handled by simpler arguments (with a finite number of reflections). While $\varkappa(hD_{y})$ does not commute with  $\Delta$ (unlike in the model case), the support of $\eta$ in phase space will not significantly move over a finite time interval as a consequence of the Melrose-Sj\"ostrand propagation of singularities theorem. Therefore, up to $O_{C^{\infty}}(h^{\infty})$ terms, we may insert $\varkappa(hD_{y})$ operators before and after the propagator.
\end{rmq}
Rescale
$\omega=\frac{\alpha}{h^{2/3}}$ and $\eta=\frac{\theta}{h}$ in \eqref{eq:GN} (defining $G_N$), hence $\tau_q(\alpha,\theta)= h\tau_q(\omega,\eta)$. 
Let also $\cutoffchi^{\flat}\in C^{\infty}(\mathbb{R})$ such that $\cutoffchi^{\flat}=1$ on $(-\infty,1/2]$ and $\cutoffchi^{\flat}=0$ on $[1,\infty)$, and $\cutoffchi^{\sharp}=1-\cutoffchi^{\flat}$, where the relevance of all cut-off functions will reveal itself later on.
We now define an operator acting on smooth $f(y',\rho)$, with $\hat f$ its Fourier transform in all variables, 
\begin{equation}
  \label{eq:J}
    J(f)(x,y)=\int G_N(x,y,\eta,\omega)\cutoffchi^{\sharp}(\omega)q(\eta)^{1/6}\varkappa(h\eta)\varkappa(h\tau_q(\omega,\eta))\hat{f}(\eta,\omega/h^{1/3})d\eta d\omega\,.
  \end{equation}
For $i\in \{0,1\}$, $g_{i,j}=g_{i,j}(x,y,\eta/|\eta|,\omega/|\eta|^{2/3})=g_{i,j}(x,y,\theta,\alpha)$, hence $g_i=g_i(x,y, \theta,\alpha, h)=\sum_{j\geq 0}h^j g_{i,j}(x,y,\theta,\alpha)|\theta|^{-j}$.
  Using moreover that $\psi$ and $\zeta$ are homogeneous of order $1$ and $2/3$, we find after rescaling
\begin{multline}\label{Jform}
  J(f)(x,y) 
   = \frac{1}{2\pi h^d} \int e^{\frac ih (\psi(x,y,\theta,\alpha)-y'\cdot\theta-{\tprime}\alpha)}
   \Big(g_0Ai(-h^{-2/3}\zeta(x,y,\alpha,\theta))+ih^{1/3}g_1 q(\theta)^{-1/6} Ai'(-h^{-2/3}\zeta((x,y,\alpha,\theta))\Big)\\
   \times q(\theta)^{1/6}
   \varkappa(\theta)\varkappa(\tau_q(\alpha,\theta))\cutoffchi^{\sharp}(\alpha/h^{2/3})f(y',{\tprime}) \,dy'd{\tprime}d\theta d\alpha\,.
\end{multline}

\begin{lemma}\label{lemJinvert}
The operator $J$ is well defined from tempered distributions $\mathcal{S}'_{y',\tprime}$ into smooth functions of $(x,y)$ near $(0,0)$. In the semiclassical setting with $h$ as small parameter, $J$ is a semi-classical Fourier integral operator associated to a
canonical transform $\canonchi_J$, defined near the set $\{y'=0,\tprime=0, |\theta|=1,\alpha=0 \}$ and such that
 $ \canonchi_J(y'=0,{\tprime}=0,|\theta|=1,\alpha=0)=\{y=0,x=0,|\theta|=1,\xi=0\}$.
Moreover, $J$ is elliptic on this set and, microlocally near this set, an intertwining relation holds,
\begin{equation}
  \label{eq:LapJ}
  -h^2 \Delta J(f)= J(\tau_{q}^2(hD_{{\tprime}},hD_{y'}) f)+O(h^{\infty}).
\end{equation}
\end{lemma}
As the symbol $p_h$ is smooth and compactly supported in $(\theta,\alpha,\sigma)$, $J$ is easily extended to $\mathcal{S}'_{y',\tprime}$. In the case of the Dirichlet boundary condition, this result follows immediately from Theorem \ref{thmMelrose}, since $p_0$ is elliptic and $p_1$ vanishes on $\partial\Omega$. The Neumann case require some additional work, as the term involving $Ai'(-\zeta)$ is not longer negligible compared with the one involving $Ai(-\zeta)$, so the same argument doesn't apply directly. 
\begin{proof}
Write the term in brackets in \eqref{Jform} as an integral using \eqref{eq:15}, then make the change of variables $s=\sigma/h^{1/3}$ 
\begin{align}\label{AiAi'}
&\int e^{i(\frac{s^3}{3}-sh^{-2/3}\zeta(x,y,\theta,\alpha))}(g_0(x,y,\theta, \alpha, h)+sh^{1/3}q^{-1/6}(\theta)g_1(x,y,\theta,\alpha,h))ds\\
=& h^{-1/3}\int e^{\frac ih(\frac{\sigma^3}{3}-\sigma\zeta(x,y,\theta,\alpha))}(g_0(x,y,\theta, \alpha, h)+\sigma q^{-1/6}(\theta)g_1(x,y,\theta,\alpha,h))d\sigma.
\end{align}
Although $g_0$ is elliptic near the boundary, it is not immediately clear that the symbol of the last integral is nonvanishing. We therefore show that, on the essential support of the symbol in \eqref{Jform}, the second term in the sum remains a small perturbation of the first one. 

Recall from \eqref{eq:Neumann-recursion} that $g_{1,0}$ was constructed (as in the Dirichlet case) so that its boundary value satisfies $g_{1,0}|_{x=0}=0$. 
It follows that, for $x$ and $y$ small, the contribution of  $|g_{1,0}|$ is $O(x,y)$. 

We next observe that the integral is effectively localized to $|\sigma| \lesssim \alpha$. Indeed, stationary points of the phase satisfy $\sigma^2=\zeta(x,y,\theta,\alpha)\leq \alpha$. Consequently, on the essential support of the symbol in \eqref{Jform} we may restrict to $|\sigma| \leq 2\alpha$, since the contribution of the complementary region is $O(h^{\infty})$ by repeated integrations by parts in $\sigma$. More precisely, on the support of $\chi^{\#}(\alpha/h^{2/3})$ we have $\alpha\geq h^{2/3}$. Hence, if $|\sigma|>2\alpha$, $M$ integrations by parts in $\sigma$ yield a factor $O((\alpha/h)^{-M})\leq O(h^{-M/3})$ for every $M\geq 1$, and therefore a $O(h^{\infty})$ contribution. 

It follows that, on the essential support of \eqref{Jform}, the symbol $g_{0,0}+\sigma q^{-1/6}(\theta)g_{1,0}$ equals $1+O(x,y)$, and therefore, for $x$ and $y$ sufficiently small, it is elliptic.
\end{proof}

In the following we set 
\begin{equation}\label{symbgh}
g_h(x,y,\theta,\alpha,\sigma):=\big(g_0(x,y,\theta, \alpha, h)+\sigma q^{-1/6}(\theta)g_1(x,y,\theta,\alpha,h)\big) \varkappa(\theta)\varkappa(\tau_q(\alpha,\theta))\cutoffchi^{\sharp}(\alpha/h^{2/3}).
\end{equation}

\subsection{Some useful results on Airy functions} 
We now digress and present a variation on the Poisson summation
formula.
For $z\in\mathbb{C}$ we set $ A_\pm(z)=e^{\mp i\pi/3} Ai(e^{\mp i\pi/3} z)$, then $Ai(-z)=A_+(z)+A_-(z)$ and $\overline{A_+(z)}=A_-(\overline{z})$ and 
\begin{equation}
  \label{eq:Apm}
  \begin{split}
\frac{A_-(z)}{A_+(z)}=e^{-2i\chi(z)},\quad \chi(z)\sim_{1/z} -\frac{\pi}{4}+\frac 23 z^{3/2}\Big(1-\frac{5}{4(2z)^3}+\frac{1105}{96(2z)^6}-...\Big),
\\
 \frac{A'_-(z)}{A'_+(z)}=e^{-2i\tilde\chi(z)},\quad 
 \tilde\chi(z)\sim_{1/z} -\frac{3\pi}{4}+\frac 23 z^{3/2}\Big(1+\frac{7}{4(2z)^3}-\frac{1463}{96(2z)^6}+...\Big).
\end{split}
\end{equation}
The last two formulas follow from \cite[(2.106), (2.107)]{vaso}. We recall some standard facts about the behavior of the zeros of the Airy function and those of its derivative (see \cite[Section 2.2]{vaso}, \cite[Thm. A.2.17, (A.2.18), (A.3.27)]{meta}).
\begin{lemma}
All the zeroes of $Ai(z)$ and of $Ai'(z)$ are located on the negative part of the real axis. Let $\{-\omega_k\}_{k\geq 1}$ be the zeros of $Ai(z)$ and $\{-\tilde\omega_k\}_{k\geq 1}$ be the zeros of $Ai'(z)$, then $1<\tilde\omega_1<\omega_1<\tilde\omega_2<\omega_2<...<\tilde\omega_k<\omega_k<\tilde\omega_{k+1}<...$.
\end{lemma}
We have $\tilde\omega_1=1,018...$, while $\omega_1=2,3381...$. We will need the next two lemmas.
\begin{lemma}\label{lemL}
Define, for $\omega \in \R$, the function $\tilde L(\omega)=\pi+ i\log \frac{A'_-(\omega)}{A'_+(\omega)}=\pi+2\tilde \chi(\omega)$: $\tilde L$ is an analytic, real valued, strictly increasing function 
and, for $\omega>1$,
\begin{equation}
  \label{eq:propL}
    \tilde L(\omega)=\frac 4 3 \omega^{\frac 3 2} - \frac{\pi}{2}-B_{\tilde L}(\omega^{\frac 3
    2})\,,\quad
  B_{\tilde L}(u)= \sum_{k=1}^\infty b_k u^{-k}\,,\,\, (b_k)_{k}\in\R\,,\,\,
  b_1=-7/24(>0)\,.
\end{equation}
Finally, let $\{-\tilde\omega_k\}_{k\geq 1}$ denote the zeros of the derivative $Ai'(\cdot)$ of the Airy function in decreasing order,
\begin{equation}
  \label{eq:propL2}
 \tilde L(\tilde\omega_k)=2\pi k \text{ and }
  \tilde L'(\tilde \omega_k)= 2\pi \int_0^\infty Ai^2(x-\tilde\omega_k) \,dx\,= \tilde\omega_k Ai^2(-\tilde\omega_k).
\end{equation}
\end{lemma}
\begin{lemma}
Let $\N^{*}=\N\setminus\{0\}$.  In $\mathcal{D}'(\R_\omega)$, one has
  \begin{equation}
    \label{eq:AiryPoisson}
        \sum_{N\in \Z} e^{-i N\tilde L(\tilde \omega)}= 2\pi \sum_{k\in \N^*} \frac 1
    {\tilde L'(\tilde\omega_k)} \delta(\tilde\omega-\tilde\omega_k)\,.
  \end{equation}
\end{lemma}
Analogues of the last two lemmas are proved and used in \cite{ilpCE, ILLP} for the Dirichlet case, with $\omega_k$ in place of $\tilde\omega_k$, where $\{(-\omega_k)\}_{k\geq 1}$ denote the zeros of the Airy function and where the function $\tilde L$ is defined by $L(\omega)=\pi+ i\log \frac{A_-(\omega)}{A_+(\omega)}=\pi+2 \chi(\omega)$. 
The Neumann case instead requires the use of $\tilde \omega_k$. The two lemmas are proved, in a more general setting, in \cite[Lemmas B.16 and B.18]{DSS26}. One main difference between Dirichlet and Neumann is the sign of the first term of the asymptotic expansion of $B_{\tilde L}$ (for Dirichlet, the corresponding $B_L$ has $b_1>0$) . 

Let us define, for $\omega\in \R$, and without loss of generality, an arbitrary choice of $+$ sign for the time propagator $\exp(it{\qtau}(\omega,\eta))$,
\begin{equation}
  \label{eq:Kequiv}
  K_\omega(f)(t,x,y)=\int e^{it{\qtau}(\omega,\eta)} G_N(x,y,\eta,\omega)\cutoffchi^{\sharp}(\omega) q^{1/6}(\eta)
\varkappa(h\eta) \varkappa(h{\qtau}(\omega,\eta)) \hat f(\eta,\frac{\omega}{h^{1/3}}) \, d\eta\,.
\end{equation}
Due to both cut-off in $\omega$ as well as that in $\eta$, $K_{\omega}(f)$ is supported in $1\leq \omega\leq \ceps_0h^{-2/3}$ and so is $R(t,x,y,\omega,a,h):=((\partial^2_{t}-\Delta)K_{\omega}(f))(t,x,y)$. By design of $G_N$, using \eqref{eq:GN-quasimode}, we have moreover that, for small $r_0$ and $a_0$ and for all (large) $M\in \N$,
\begin{equation}\label{eq:Kom1}
\sup_{|a|<a_0}\sup_{|(t,x,y)|<r_0}\sup_{\omega}\Big|\nabla^{\alpha}_{t,x,y,\omega}R\Big|\leq C_{M,\alpha}h^M\,.
\end{equation}
Moreover, at $x=0$, we have $\partial_x K_{\omega_{k}}(f)(t,0,y)=0$ as $\partial_x G(0,y,\eta,\omega_k)=0$ (recall $\zeta(x,y,\eta,\omega)|_{x=0}=\omega$ and \eqref{condN}). In other words, $K_{\omega}(f)(t,x,y)$ is a solution to the wave equation,  up to $O(h^{\infty})$; and when $\omega=\omega_{k}$, it satisfies the Neumann boundary condition.

To get a sense of perspective, let us remark that, in the Friedlander model case, and then (up to normalization) $G_{M}(x,y,\eta,\omega):=\int_{\eta} \exp(i y\cdot \eta) Ai(x q^{1/3}(\eta)- \omega)\, d\eta$. Then the corresponding $K_{\omega}(f)$ is an exact solution to the half-wave equation, satisfying Dirichlet boundary condition if $\omega=\omega_k$ and the Neumann boundary condition if $\omega=\tilde\omega_{k}$, but $f$ should not be considered as its data: if one picks $f$ such that, on the model, $J(f)$ is a Dirac at $(x=a,y=0)$, then $ f=\int_{\eta} \exp(-iy\cdot \eta) Ai(aq^{1/3}(\eta)-\omega)\,d\eta$ and then integrating over $\omega$ recovers $\delta_{x=a,y=0}$ by a standard identity on Airy functions. For this $f$, the integral over $\omega$ of $K_{\omega}(f)(t,x,y)$ is then just an half-wave solution with no boundary condition. In \cite{ilp12,ilpCE} such a solution with $\omega=\omega_k$ in the Dirichlet case, is iterated by successive reflections at the boundary. For Robin boundary condition, the analogous construction is carried out in \cite{DSS26}, using a version of the Airy-Poisson formula \eqref{eq:AiryPoisson} (see \cite[Lemma B.18]{DSS26}) associated with the zeros $\{-\omega_{l,k}\}_{k\geq 1}$ of the Robin function $Rob_l(z)=\cos(l)Ai(z)-sin(l)Ai'(z)$. In the general setting of \cite{ILLP}, for the Dirichlet condition, the Airy-Poisson formula \eqref{eq:AiryPoisson} applied to $G_D(x,y,\eta,\omega)$ with $\omega_k$ in place of $\tilde\omega_k$ directly yields a sum over $N$ of waves that can subsequently be identified with the analogues of the reflected waves, while the spectral sum over $k$ provides a direct decomposition of the Dirac initial data.

We now revert to the general case for Neumann, where we follow the strategy from \cite{ILLP}, but replace $G_D(x,y,\eta,\omega)$ by $G_N(x,y,\eta,\omega)$. Recall we defined $J(f)(x,y)$ in \eqref{eq:J} and we may rewrite $J(f)(x,y)=\int_\R K_\omega(f)(0,x,y)\,d\omega$.
Let $\eta=\frac \theta h$, $q^{\frac 16}(\eta)=h^{-\frac 1 3}q^{\frac 1 6}(\theta)$ and $\alpha=h^{\frac 2 3}\omega$, then, using \eqref{AiAi'}, we write
\begin{multline}
  \label{eq:Kom}
 K_\omega(f)(t,x,y)=\frac{ h^{1/3}}{2\pi h^d} \int e^{\frac i
    h(t\tau_q(\alpha,\theta)+\psi(x,y,\theta,\alpha)-y'\cdot \theta-{\tprime} \alpha )} \int e^{\frac ih(\frac{\sigma^3}{3}-\sigma\zeta(x,y,\theta,\alpha))}g_h(x,y,\theta,\alpha,\sigma) d\sigma \\
  \times  q^{\frac 1 6}(\theta) f(y',{\tprime}) \,dy'd{\tprime}d\theta \,,
\end{multline}
with $g_h$ defined in \eqref{symbgh}. Recall from the proof of Lemma \ref{lemJinvert} that the symbol $g_h$ is elliptic.

As we will see later, both cut-off functions $\varkappa$ in $K_{\omega}(f)$ relate to localization operators from Definition \ref{dfnparametrix}. Moreover, $K_{\omega}(f)$ is a suitable test function in $\omega$ (smooth and compactly supported in $\omega$). Using \eqref{eq:AiryPoisson},
\begin{equation}\label{formulafAP}
 \langle \sum_{N\in \Z} e^{-iN \tilde L(\omega)}
    , K_\omega(f)(t,x,y)\rangle_{\omega}=2\pi  \sum_{k\in \N^*} \frac 1
      {\tilde L'(\tilde\omega_k)}  K_{\tilde\omega_k}(f)(t,x,y)\,.
\end{equation}
The $N=0$ term in the sum over $N$ is $J(f)$. Moreover, at $x=0$ the normal derivative of the RHS vanishes, as the sum over $k$ is finite and each term in this sum vanishes as we just observed, and this finite sum (on the RHS) satisfies the wave equation, up to $O(h^{\infty})$ terms, due to \eqref{eq:Kom1}.

These remarks will later be of crucial importance to verify that Definition \ref{dfnparametrix} will hold for the parametrix we shall introduce in the next sections, up to finding a suitable function $f$ that will recover the data at $t=0$ in \eqref{formulafAP}. This remaining step is far from trivial, unlike in the model case, for which we know explicitly the spectral resolution of the Laplacian and can therefore expand a Dirac mass over the eigenmodes, as alluded to earlier.

One may expect that it should be enough to consider initial data (at time $0$) $\chi_0(hD_{x})\varkappa(hD_{y})\delta_{(a,0)}$, for $\varkappa$ supported near $\mathbb{S}^{d-1}$ and $\chi_0\in C^{\infty}_0$ supported near $0$. Indeed,  classical geometric optics arguments provide a parametrix for data $(1-\chi_0(hD_{x}))\varkappa(hD_{y})\delta_{(a,0)}$:  due to the cut-off $(1-\chi_0(hD_{x}))$, singularities are transverse to the boundary at $x=0$ (there is at most one reflection). However, 
\begin{equation}\label{datachi0var}
  \chi_0(hD_{x})\varkappa(hD_{y})\delta_{(a,0)}=\int e^{i((x-a)\xi+y\cdot\eta)}\chi_0(h\xi)\varkappa(h\eta)d\xi d\eta
  =\frac{1}{h^d}\widehat{\chi}_0\Big(\frac{x-a}{h}\Big)\widehat{\varkappa}\Big(\frac{y}{h}\Big)\,.
\end{equation}
Therefore, at $x=0$, this data will be $O(h^{\infty})$ together with its derivative only if we assume that $a\geq h^{1-\ceps}$ for some $\ceps>0$. For smaller $a$, in the case of the Friedlander model operator $\Delta_M$, we can take advantage of the known, explicit, spectral resolution of $-\Delta_{M}$  in order to consider an initial data $\chi_0(-h^2\Delta_M)\varkappa(hD_y)\delta_{(a,0)}$ that can be further expressed as a sum of eigenfunctions whose normal derivatives vanish on the boundary. By contrast, in the general case, we only have quasimodes and this is a source of significant difficulties for these very small $a$. Nevertheless, we will decompose the parametrix construction according to the respective values of $a$ and $h^{2/3}$, with an overlap between the two regimes where any construction holds. In subsection \ref{ss22}, for $a\gg h^{2/3}$, we will use \eqref{datachi0var} as a data, and mainly proceed with the sum over $N$ in our Airy-Poisson formula \eqref{formulafAP}. In subsection \ref{lowparam}, dealing with $a \lesssim h^{2/3}$, stationary phase methods in this sum over $N$ break down (although one could push them down to $a\gg h$, matching the heuristic above, but with no obvious benefit) in addition to the problem of defining a suitable initial data. We will solve the data issue in subsection \ref{sss231} by choosing the model initial data $\chi_0(-h^2\Delta_M)\varkappa(hD_y)\delta_{(a,0)}$. In some sense, in the very narrow strip where it is located, the spectral localizations with respect to $\Delta_{M}$ or $\Delta$ are close enough that gallery modes are good substitutes to the quasimodes in defining said data.  One then proceeds with a parametrix construction where such data is, again, split according to the values of $k$ in the spectral sum defining it: either $k$ is large enough and we recover a large parameter and can proceed as in the previous regime, or we have the relatively small value of $k$ for which we proceed with the spectral sum, proving in subsections \ref{sectpseudocalcul} and \ref{sectproofs} that terms appearing in that expansion are close enough to the model gallery modes and therefore retain enough of their properties to provide a parametrix. This part of the construction is quite delicate and obviously absent in the model case, while of independent interest as far as uniform estimates on quasimodes are concerned as these will be proved in the range $k\ll h^{-1/4}$, exceeding by far what we need in our construction.

\subsection{Parametrix construction for $a\geq h^{\frac 23-\ceps}$, $0<\ceps<2/3$}\label{ss22}
An initial data \eqref{datachi0var} is $O(h^{\infty})$ on the boundary for any $\chi_0$, compactly supported near $0$. 
Let $\chi_0\in C^{\infty}_0(-2\ceps_0,2\ceps_0)$ with $\chi_{0\big|[-\ceps_0,\ceps_0]}=1$.
\begin{lemma}\label{lemgab}
Let $a_0>0$, $r_0>0$ be small enough. For all $a\in [h^{\frac 23-\ceps},a_0]$, there exists a smooth function $f_{h,a}$ such that $\varkappa(hD_{y'})f_{h,a}=g_{h,a}$ and
  \begin{equation}
    \label{eq:Jdirac}
    J(f_{h,a})(x,y)=\chi_0(hD_{x})\varkappa(hD_{y})\delta_{(a,0)} + O_{C^{\infty}(|(x,y)|\leq r_0)}(h^{\infty})\,,
  \end{equation}
where the remainder is  $O_{C^{\infty}(|(x,y)|\leq r_0)}(h^{\infty})$ uniformly in $a$.
\end{lemma}

The lemma follows from the fact, recalled above, that $J$ is an elliptic Fourier integral operator. As in \cite[Lemma 2.12]{ILLP}, $f_{h,a}$ can be computed explicitly, and we therefore recall its expression from there. The only difference lies in the symbol; this does not affect the argument, since all that is needed is its ellipticity near the boundary. 

\begin{lemma}
\label{lemp}
There exist a symbol $\tilde q_h$ of order $0$, with support near $|\theta'|=1$ and there exists $\chi\in C^{\infty}_0$ supported in a small neighborhood of $0$ (depending only on $a_0$ and $\ceps_0$), such that $f_{h,a}$ whose Fourier transform is given by
 \begin{equation}\label{hatgha}
\widehat{f_{h,a}}(\theta'/h, \alpha/h)=h^{-1+1/3}\int e^{-\frac ih\Phi(a,0,\theta',\alpha,s)}\chi(s) \tilde q_h(\theta',\alpha,s)q^{-1/6}(\theta') ds\,
\end{equation}
solves \eqref{eq:Jdirac}.
\end{lemma}
 \begin{proof}
We may invert microlocally the operator $J$ from \eqref{eq:J} by setting 
\[
J^{-1}(F)(y',\tprime)=\frac{h^{1/3}}{h^{d+1}}\int e^{\frac ih(-\psi(x,y,\theta',\alpha',s)-\frac{s^3}{3}+s\zeta(x,y,\theta',\alpha'))+y'\cdot\theta'+\tprime\alpha')}q_h(x,y,\theta',\alpha',s)q^{-1/6}(\theta')F(x,y)dxdyd\theta' d\alpha' ds,
\]
where $q_h(x,y,\theta',\alpha',s)$ is a symbol of order $0$, $q_h=\sum_{k\geq 0} h^kq_{h,k}$, with support near $\{x=0,y=0,\alpha'=0,s=0\}$ and elliptic on this set. Here the phase function is 
\begin{equation}\label{phaPhi}
\Phi(x,y,\theta,\alpha,s)=\psi(x,y,\theta,\alpha)+\frac{s^3}{3}-s\zeta(x,y,\theta,\alpha).
\end{equation}
Notice that in \cite{ILLP} we have worked with $\frac{\sigma^3}{3}-\sigma(xq^{1/3}(\theta)-\alpha)+\tau_q\Gamma(x,y,\sigma q^{1/3}(\theta)/\tau_q,\theta/\tau_q)|_{\tau_q=\tau_q(\alpha,\theta)}$ instead of $\Phi$, where $\Gamma$ was defined in \eqref{eq:GAB}. After a suitable change of variable, this phase transforms into \eqref{phaPhi}, which is the canonical form of a phase function with degenerate critical points of order two. We aim at proving that $J^{-1}\circ J(f)=f\text{ modulo } O(h^{\infty})$: the factors $h^{1/3}$ and $h^{-1/3}$ disappear and we are left with 
\begin{multline}
J^{-1}\circ J(f)=\frac{1}{h^{d+1}}\int e^{\frac ih(-\Phi(x,y,\theta',\alpha',s)+y'\cdot\theta'+\tprime\alpha')}q_h(x,y,\theta',\alpha',s)q^{-1/6}(\theta')\\
\times \frac {1}{2\pi h^d} e^{\frac ih \Phi(x,y,\theta,\alpha,\sigma)}g_h(x,y,\theta,\alpha,\sigma)q^{1/6}(\theta) \hat{f}(\theta/h,\alpha/h)d\theta d\alpha d\sigma dxdy d\theta' d\alpha' ds.
\end{multline}
We now apply stationary phase in variables $(\sigma, s, x,y,\theta',\alpha')$: one checks that critical points are non-degenerate, such that $\theta'_c=\theta$, $\alpha'_c=\alpha$, and stationary phase provides a factor $h^{d-1+1+1}=h^{d+1}$ (one factor $h^{d-1}$ from $dyd\theta'$, one factor $h$ from $dxds$ and one factor $h$ from $d\sigma d\alpha'$). The critical value of the phase is $y'\theta+\tprime \alpha$ and we obtain (modulo $O(h^{\infty})$)
\[
J^{-1}\circ J(f)(y',\tprime)=\frac{1}{h^{d+1}}\times \frac{h^{d+1}}{2\pi h^d}\int e^{\frac ih(y'\cdot\theta+\tprime \alpha)}\tilde q_h(\theta,\alpha)\hat{f}(\theta/h,\alpha/h)d\theta d\alpha,
\]
where $\tilde q_h$ is obtained from the product $q_h(x,y,\theta',\alpha',s)q^{-1/6}(\theta') g_h(x,y,\theta,\alpha,\sigma)q^{1/6}(\theta)$ after stationary phase; asking $\tilde q_h=1$ for $\theta$ such that $|\theta|\sim 1$ and $\tprime$ near $0$ allows to chose $q_h$; with $g_h$ defined in \eqref{symbgh}, we obtain $q_h$ as announced. Define 
\begin{equation}\label{gl-2}
\tilde f_{h,a}:=J^{-1}(\chi_0(hD_x)\varkappa(hD_y)\delta_{(a,0)})\,.
\end{equation}
Then, using the second line in \eqref{datachi0var}, 
\begin{multline}\label{ghaF}
\tilde f_{h,a}(y',\tprime)=\frac{1}{h^{d}}\int e^{\frac ih y'\cdot\theta'}F_{h,a}(\tprime,\theta')d\theta'\,,\,
  F_{h,a}(\tprime,\theta')=\frac{h^{1/3}}{h^{d+1}}\int e^{\frac ih (-\Phi(x,y,\theta',\alpha,s)+\tprime \alpha+(x-a)\sigma+y\cdot \theta)}\\
{}  \times q_h(x,y,\theta',\alpha,s)q^{-\frac 1 6}(\theta')\chi_0(\sigma)\varkappa(\theta)d\sigma d\theta dx dy d\alpha ds\,.
\end{multline}
This implies that 
\[
\widehat{\tilde f_{h,a}}(\theta'/h, \alpha/h)=\frac{h^{1/3}}{h^{d+1}}\int e^{\frac ih (-\Phi(x,y,\theta',\alpha,s)+(x-a)\sigma+y\cdot \theta)}\\
{}  \times q_h(x,y,\theta',\alpha,s)q^{-\frac 1 6}(\theta')\chi_0(\sigma)\varkappa(\theta)d\sigma d\theta dx dy d\alpha ds.
\]
We apply stationary phase with respect to variables $(x,\sigma,y,\theta)$: (non-degenerate) critical points are $x=a$, $y=0$, $\sigma_c=\partial_x\Phi(a,0,\theta',\alpha,s)$ and $\theta_c=\partial_y\Phi(a,0,\theta',\alpha,s)$. The resulting symbol $\tilde q_h(\theta',\alpha,s)$ is of order $0$, with support near $\{|\theta'|=1,\alpha=0,s=0\}$ and elliptic on this set, and
\begin{equation}\label{Fha}
\widehat{\tilde f_{h,a}}(\theta'/h, \alpha/h)=h^{-1+1/3}\int e^{-\frac ih\Phi(a,0,\theta',\alpha,s)} \tilde q_h(\theta',\alpha,s)q^{-1/6}(\theta') ds\,.
\end{equation}
Here $\alpha$ is bounded, as $\tau_q(\alpha,\theta')\in \mathrm{supp}\, \varkappa$; indeed, $\alpha=h^{2/3}\omega$ and  we assumed $|\omega|\leq \ceps_0h^{-2/3}$; therefore $\alpha\leq \ceps_0$ on the support of the symbol $g_h$, as well as on the support of $q_h$ and also $\tilde q_h$. The phase of $F_{h,a}$ is stationary in $\alpha$ for $s+\tprime+O(a)=0$ and in $s$ for $s^2+O(a)\sim \alpha\leq \ceps_0$ (as we will see below using the explicit form of $\Phi$) and $a\leq a_0<1$ is small enough, therefore $s^2\lesssim \ceps_0+a_0$ (otherwise non stationary phase in $s$ provides an $O(h^{\infty})$ contribution.) Hence there exists a cut-off $\chi(\tprime)\in C^{\infty}_0((-2r,2r))$, equal to $1$ on $[-r,r]$ for $r\sim \sqrt{\ceps_0}+ a_0$, such that $(1-\chi(\tprime))F_{h,a}(\tprime,\theta')=O(h^{\infty})$ in $\mathcal{S}_{\tprime}$,
 uniformly in $\theta'$ near $|\theta'|=1$. As in the end of the proof of \cite[Lemma 2.12]{ILLP}, we can show that one, modulo $O(h^{\infty})$ terms, one may take $f_{h,a}(y',\tprime)$ to satisfy \eqref{hatgha}.
This completes the proof of Lemma \ref{lemp}.
\end{proof}
\begin{dfn}\label{defparamPoisson}
    Let $f_{h,a}$ be defined in \eqref{hatgha}: using \eqref{eq:AiryPoisson}, we define $\Prond_{h,a}$ equivalently as
    \begin{align}
      \label{eq:Prond}
      \Prond_{h,a}(t,x,y) & = \langle \sum_{N\in \Z} e^{-iN \tilde L(\omega)}
    , K_\omega(f_{h,a})(t,x,y)\rangle_{\omega}\\
      \label{eq:Prond2}
      \Prond_{h,a}(t,x,y) & =2\pi  \sum_{k\in \N^*} \frac 1
      {\tilde L'(\tilde\omega_k)}  K_{\tilde\omega_k}(f_{h,a})(t,x,y)\,.
    \end{align}
    \end{dfn}
We are abusing notation here: one should consider $\Prond^{\pm}$ depending on the sign on $t$ and then obtain $2 \Prond$ from Definition \ref{dfnparametrix} as $\Prond^{+}+\Prond^{-}$. Considering $\Prond^{+}$ is enough by time symetry and we therefore drop the $+$.

We now recall (see the discussion after \eqref{eq:Kequiv}) that, using both localizations in $\eta$ and $\tau_q(\omega,\eta)$, \eqref{eq:Prond2} may be reduced to a finite sum over $k\lesssim h^{-1}$:  support considerations on $K_{\omega}$ (as a function of $\omega$) provide $|\omega|\leq \ceps_0 h^{-2/3}$; after Airy-Poisson summation, this translates into $|\tilde \omega_{k}|\leq \ceps_0 h^{-2/3}$. The zeroes $\{-\tilde\omega_k\}_{k\geq 1}$ of the derivative $Ai'(-\omega)$ of the Airy function have asymptotic $\tilde\omega_k\sim (3\pi k/2)^{2/3}$ (see \cite[(2.53)]{vaso}). We therefore introduce a cut-off in the sum over $k$, $\cutoffchi^{\flat}_{\epsilon_{0}}(h^{2/3}\tilde\omega_{k}):={\cutoffchi^{\flat}}(h^{2/3}\tilde\omega_k/(2\ceps_{0}))$. Then, \eqref{eq:Prond2} may be rewritten as a finite sum,
 \begin{equation}
      \label{eq:Prond2cut}
      \Prond_{h,a}(t,x,y):=2\pi  \sum_{k\in \N^*} \frac{{\cutoffchi^{\flat}_{\epsilon_{0}}}(h^{2/3}\tilde\omega_k
        )}
      {\tilde L'(\tilde\omega_k)}  K_{\tilde\omega_k}(f_{h,a})(t,x,y)\,.
    \end{equation}
    From $\tilde\omega_1\geq 1.01$, we remark that the cut-off function ${\cutoffchi^{\sharp}(\omega)}$, that was introduced in the definition \eqref{eq:Kequiv} of $K_{\omega}$ and which is equal to $1$ on $[1,\infty)$, is no longer needed when restricting $\omega$ to the set $\{\tilde \omega_{k}\}_{k\in \N^{*}}$.  
    But it will help on the other sum \eqref{eq:Prond}, in estimating how many $N$'s contribute significantly.  Again with \eqref{eq:AiryPoisson}, we also have
 \begin{equation}
      \label{eq:Prondcut}
      \Prond_{h,a}(t,x,y)= \langle \sum_{N\in \Z} e^{-iN \tilde L(\omega)}
    , {\cutoffchi^{\flat}}(h^{2/3}\omega/(2\ceps_{0}))K_\omega(f_{h,a})(t,x,y)\rangle_{\omega}\,.
    \end{equation}
The sum $\sum_{N\in\mathbb{Z}}$ converges in $\mathcal{D}'_{\omega}$ and ${\cutoffchi^{\flat}}(h^{2/3}\omega/(2\ceps_0))K_\omega(f_{h,a})(t,x,y)$ is smooth in $(t,x,y)$ in a neighborhood $W$ of $(0,0,0)$ and smooth and compactly supported in $\omega$. Substituting $f_{h,a}$ from \eqref{hatgha} in \eqref{eq:Kom} yields
\begin{multline}\label{Kghauseful}
K_{\omega}(f_{h,a})(t,x,y)=\frac{h^{2/3}}{2\pi h^{d+1}}\int e^{\frac ih (t\tau_q(h^{2/3}\omega,\theta)+\Phi(x,y,\theta,h^{2/3}\omega,\sigma)-\Phi(a,0,\theta,h^{2/3}\omega,s))}\\
\times g_h(x,y,\theta,h^{2/3}\omega,\sigma)\chi(s)\tilde q_h(\theta,h^{2/3}\omega,s)dsd\theta d\sigma\,.
\end{multline}
We set $\Prond_{h,a}(t,x,y)=\sum_{N\in \Z} V_N(t,x,y)$, where $V_{N}$ is defined as
\begin{align}
  \label{eq:newVN}
  \quad V_N(t,x,y) := & \int e^{-iN\tilde L(\omega)}{\cutoffchi^{\flat}}(h^{2/3}\omega/(2\ceps_0) )K_{\omega}(f_{h,a})(t,x,y)d\omega\\
 = & \begin{multlined}[t] \frac 1 {2\pi h^{d+1}} \int e^{\frac i h (
    t\tau_q(\alpha,\theta)+\Phi(x,y,\theta,\alpha,\sigma)-\Phi(a,0,\theta,\alpha,s)-Nh \tilde L(h^{-2/3}\alpha)
    )} \\
\quad\quad\quad{}\times {\cutoffchi^{\flat}}(\alpha/(2\ceps_0))\chi(s)g_h(x,y,\theta,\alpha,\sigma)\tilde q_h (\theta,\alpha,s) ds \,d\theta  d{\sigma}d\alpha\,.\end{multlined}
\end{align}
The symbol ${\cutoffchi^{\flat}}(\alpha/(2\ceps_0)) \chi(s)g_h \tilde q_h$ of $V_N$ is the same for every $N$,  is of order $0$ and is given by an asymptotic expansion with small parameter $h$ and main term equal to $1$ (indeed, since $\tilde q_h$ has been obtained by inverting $J$, whose symbol is $g_h$). Note that we do not have a finite sum over $N$: convergence should be understood in the distributional sense. The cut-off in $\alpha$ is redundant but we will leave it there to emphasize compact support in $\alpha$. As in \cite[Lemma 2.16]{ILLP}, Lemma \ref{sommeNfinie} below shows that for a generic function $f_{h}$ replacing $f_{h,a}$ the sum over $N$ converges and is $O(h^{\infty})$ for $N>h^{-1/3}$, provided $f_{h}$ is of moderate growth with respect to $h$. Practically, $f_{h}$ is an oscillatory integral with an Airy type phase and with a smooth rapidly decaying or compactly supported symbol and as such, is of moderate growth. The proof works in the same way for the Neumann case, as the only difference with respect to the Dirichlet case is the symbol.

As such, we may indeed replace $f_{h,a}$ modulo $O(h^{\infty})$ as it will only concern a (large but) finite number of terms. The main result of this section is the following proposition, which yields Theorem \ref{thmN} for initial data supported at distance greater than $h^{2/3-\epsilon}$ from the boundary.
\begin{prop}\label{propsommefinie}
Let $a\in (h^{\frac{2}{3}-\ceps},a_0]$ with small $a_0,\ceps>0$.
\begin{itemize}[leftmargin=5.5mm]
\item
For $|t|\lesssim 1$, $\Prond_{h,a}$ is essentially a finite sum in $N$ at any given time,
\begin{equation}
  \label{eq:newProndcut}
  \Prond_{h,a}(t,x,y)=\sum_{|N|\lesssim |t| a^{-1/2}} V_N(t,x,y)+O_{C^{\infty}}(h^{\infty})\,. 
  \end{equation}
  Moreover we can introduce a cut-off ${\cutoffchi}^{\sharp}(4\alpha/a)$ in the definition of $V_N$ without changing its main contribution modulo $O(h^{\infty})$ terms. 
\item At $t=0$, we have $\mathcal{P}_{h,a}(0,x,y)=\chi_0(hD_{x}){\varkappa}(hD_y)\delta_{(a,0)}+O_{C^\infty}(h^\infty)$.
\item On $\{x=0\}$, the Neumann boundary condition is satisfied.
  \item $\Prond_{h,a}$ is a parametrix in the sense of the Definition \ref{dfnparametrix}.
\end{itemize}
\end{prop}
\begin{rem}\label{rmqalpha>a}
The cut off ${\cutoffchi^{\sharp}}(h^{-2/3}\alpha)$ from $K_{\omega}$ restricts to $1\leq h^{-2/3}\alpha$. The last statement in the first part of Proposition \ref{propsommefinie} translates into the contribution of the integrals defining $V_N$ being irrelevant for small values $\alpha\leq a/2$: for $a>h^{2/3 -\ceps}$ we can further restrict to $\alpha>a/2$. This follows right away from the expression of $G_N(x,y,\theta/h,\omega)$ appearing in the definition of $K_{\omega}(f_{h,a})$ (recall \eqref{eq:Kequiv}): using \eqref{eq:defG}, $G_N$ reads as a sum of Airy functions computed at $-\zeta=x|\eta|^{2/3}e_0(x,y,\eta,\omega)-\omega$ with an elliptic $e_0$, close to $1$. These Airy functions are exponentially decreasing for $-\zeta>0$; hence, if $a>h^{2/3-\ceps}$, $\eta=\theta/h$ with $|\theta|\sim 1$ and $\omega=h^{-2/3}\alpha$, we must have $a\lesssim \alpha$ since otherwise the contribution from $G_N$ is $O(h^{\infty})$. Note that $h^{2/3-\ceps}\lesssim \alpha$ is required to perform stationary phase arguments; for $\alpha\leq h^{2/3-\ceps}$ (to be dealt with if $a\leq h^{2/3-\ceps}$ !), rescaling no longer provides a large parameter.
\end{rem}
\begin{proof}
From \eqref{condN}, it follows that $\partial_xG(0,y,\eta,\tilde\omega_k)=e^{i\psi(0,y,\eta,\tilde\omega_k)}
  (-g_0\partial_x\zeta+iq(\eta)^{-1/6}\partial_xg_1)Ai'(-\tilde\omega_k)$, which immediately yields $\partial_xG(0,y,\eta,\tilde \omega_k)=0$. Therefore, from
\eqref{eq:Prond2cut} being a finite sum, we get
  $\partial_x \Prond_{h,a}(t,x,y)_{{\textstyle |}\partial\Omega} =0$,
 hence the Neumann boundary condition holds for $\Prond_{h,a}$. From \eqref{eq:AiryPoisson}, we get that the distribution $\sum_{N\in\mathbb{Z}}e^{-iNL(\omega)}\in \mathcal{S}'(\R)$. Moreover, from the upcoming Lemma, for $|N|>h^{-1/3}$, the sum is $O(h^{\infty})$ irrespective of $f_{h,a}$. 
 \begin{lemma}
  \label{sommeNfinie}
  Let $f_{h}$ be a smooth function of $(y',\tprime)$, with compact support in $\tprime$ and of moderate growth in $h$, and $K_{\omega}(f_{h})$ be defined by \eqref{eq:Kom}. Then
\begin{equation}\label{eq:suminN}
  \langle \sum_{|N|\gtrsim h^{-1/3}}e^{-iN \tilde L(\omega)}
    , {\cutoffchi^{\flat}}(h^{2/3}\omega/\ceps_{0})K_\omega(f_{h})(t,x,y)\rangle_{\omega} = O(h^{\infty})\,.
\end{equation}
\end{lemma}
The Lemma follows exactly like in \cite[Lemma 2.16]{ILLP}.  As such, we are reduced to a finite number of $N$'s, and  from \eqref{eq:Prond} and \eqref{eq:Kom1}, it follows that, taking $W$ smaller if needed, and uniformly in $a<a_0$, one has
$(\partial^2_{t}-\Delta)\mathcal{P}_{h,a}\in O_{C^{\infty}(W)}(h^{\infty})$, not only for $x>0$ but in the full neighborhood $W$ of $(0,0,0,1)$.  Both statements on $\mathcal{P}_{h,a}$ are independent on the particular choice of the function $f_{h,a}$ such that \eqref{eq:Jdirac} holds. It remains to check that, with our choice of $f_{h,a}$ given in \eqref{eq:Jdirac}, $\mathcal{P}_{h,a}(0,x,y)$ is the right initial value: this has already been proved in \cite[Prop 2.14]{ILLP}, by first showing that the sum over $N$ is (large but) finite and that at $t=0$, in the sum over $N$ in \eqref{eq:newProndcut}, all the oscillatory integrals $V_N(0,x,y)$ for $|N|\geq 1$ provide a $O(h^{\infty})$ contribution, while $V_0(0,x,y)=J(f_{h,a})(x,y)$ which, by design, is our initial data. 
The fact that the number of $N$ is finite allows to deduce that $\mathcal{P}_{h,a}(0,x,y)=V_0(0,x,y)=J(f_{h,a})(x,y)=\chi_0(hD_x)\varkappa(hD_y)\delta_{(a,0)}$ and conclude. We do not reproduce the proof from \cite{ILLP} here as all the arguments are exactly the same.
\end{proof}

\subsection{Parametrix construction for $a< h^{2/3-\ceps}$, $0<\ceps<1/12$}\label{lowparam}
We follow the approach of \cite{ILLP}. 

\subsubsection{Spectral theory for $-\Delta_M$: the initial data in terms of model gallery modes}\label{sss231}
Let $-\Delta_M$ be the Friedlander model operator introduced in \eqref{eq:LapM},
and recall $q$ is a positive definite quadratic form.
Taking the Fourier transform in the $y$ variable, the operator $-\Delta_{M}$ becomes
$-\partial^2_x+|\eta|^2+xq(\eta)$. For  $\eta\neq 0$, this operator is a positive self-adjoint operator 
on $L^2(\mathbb{R}_+)$, with compact resolvent. The next Lemma is proved in \cite{DSS26} for the Dirichlet case and in \cite[Prop.2.8]{DSS26} for Robin boundary condition (with $q(\eta)=|\eta|^2$, but only using $q(\eta)\neq 0$):
\begin{lemma}\label{lemorthog}
There exists an Hilbert basis of $L^{2}(\mathbb{R}_{+})$ where $\{\tilde e_k(x,\eta)\}_{k\geq 0}$ are eigenfunctions of $-\partial^2_x+|\eta|^2+xq(\eta)$, with eigenvalues $\lambda_k(\eta)=|\eta|^2+\tilde\omega_k q(\eta)^{2/3}=\tau_{q}^2(\tilde\omega_k,\eta)$. These eigenfunctions are translated and rescaled Airy functions:
\begin{equation}\label{eig_k}
\tilde e_k(x,\eta)=\frac{\sqrt{2\pi} q(\eta)^{1/6}}{\sqrt{\tilde L'(\tilde\omega_k)}}
Ai\Big(xq(\eta)^{1/3}-\tilde\omega_k\Big),
\end{equation}
where $\tilde L'(\tilde\omega_k)$ (from \eqref{eq:propL2}) normalizes $\|\tilde e_k(.,\eta)\|_{L^2(\mathbb{R}_+)}=1$. Moreover, $\partial_x\tilde e_k(x,\eta)|_{x=0}=0$ for all $k\geq 1$.
\end{lemma}
Notice that, when dealing with the Dirichlet boundary condition, one need to replace $\tilde\omega_k$ by $\omega_k$ and the normalisation coefficient $\tilde L(\tilde\omega_k)$ by $L(\omega_k)=\int_0^{\infty}Ai^2(x-\omega_k)dx$ (see the discussion after Lemma \ref{lemL}).

For $a>0$, the Dirac distribution $\delta_{x=a}$ on $\mathbb{R}_+$ may be decomposed as 
 $\delta_{x=a}=\sum_{k\geq 1} \tilde e_k(x,\eta)\tilde e_k(a,\eta)$. Eigenfunctions of 
$-\partial^2_{x}+xq(\eta)$ are $\tilde e_k(x,\eta)$ with eigenvalue $\lambda_k(\eta)-|\eta|^2=\tau_{q}^2(\tilde\omega_k,\eta)-|\eta|^2=\tilde \omega_k q^{2/3}(\eta)$,  and for cut-offs $\varkappa$, $\chi_0$ to be chosen, the following spectral decomposition holds
\begin{equation}\label{diracmoddeltaM<} 
\chi_0(-h^2\partial^2_{x}+xq(h\partial_y))\varkappa(hD_y)\delta_{(a,0)}=\sum_{k\geq 1}\int
e^{iy\cdot \eta} \\ 
\chi_0\bigl(h^{\frac 23}\tilde\omega_kq^{\frac 23}(h\eta)\bigr) \varkappa(h\eta)\tilde e_k(x,\eta)\tilde e_k(a,\eta)d\eta\,.
\end{equation}
With $\chi_0\in C^{\infty}_0(-2\ceps_0, 2\ceps_0)$, $\chi_{0}=1$ on $[-\ceps_0,\ceps_0]$, \eqref{diracmoddeltaM<} is a finite sum with $O(\ceps_0/h)$ terms. 
Setting $\theta=h\eta$, from support considerations, taking the support of $\chi_0$ smaller if necessary, we can assume that $\chi^{\flat}(\alpha/\ceps_{0})\chi_0(\alpha q^{2/3}(\theta))\varkappa(\theta)=\chi_0(\alpha q^{2/3}(\theta))\varkappa(\theta)$,
where $\alpha=h^{2/3}\omega$ and $\chi^{\flat}(h^{2/3}\omega/\ceps_{0})$ is the cut-off introduced in \eqref{eq:Prond2cut} (which restricts the support of $K_{\omega}$ to values $\omega\leq \ceps_0/h^{2/3}$).

For $a\leq h^{2/3-\ceps}$, the easiest way to define an initial data is to chose the lefthand side term in \eqref{diracmoddeltaM<}, which does vanish on the boundary ($\partial_x \tilde e_k(0,\eta)=0$ for every $k\geq 1$). Using both \eqref{eq:Prondcut} and \eqref{eq:Prond2cut}, we are left to obtain a smooth function $g_{h,a}$ such that for $\Prond_{h,a}$ as in \eqref{eq:Prond2cut} to have
\begin{equation}\label{dataamic}
\Prond_{h,a}(0,x,y)=\chi_0(-h^2\partial^2_{x}+xq(h\partial_y))\varkappa(hD_y)\delta_{(a,0)}+O_{C^{\infty}}(h^{\infty}).
\end{equation}
We proceed as follows : for all $k$, we define $\mathcal{E}_{M}(\cdot,\tilde\omega_k)$  
\begin{multline}\label{defEmathcalM}
\mathcal{E}_M(x,y,a,\tilde\omega_k):=\int e^{iy\cdot \eta} 
\chi_0(h^{2/3}\tilde\omega_kq^{2/3}(h\eta)) \varkappa(h\eta)\tilde e_k(x,\eta)\tilde e_k(a,\eta)d\eta\\
=\frac{2\pi}{\tilde L'(\tilde \omega_k)}\int e^{iy\cdot \eta}\chi_0(h^{2/3}\tilde \omega_kq^{2/3}(h\eta)) \varkappa(h\eta)q^{1/3}(\eta)Ai(xq^{1/3}(\eta)-\tilde\omega_k)Ai(aq^{1/3}(\eta)-\tilde \omega_k)d\eta.
\end{multline}
For a given $K_{\ceps}$ such that $h^{-\ceps}\ll h^{-2\ceps}\lesssim K_{\ceps}\lesssim h^{-1/4+\ceps}\ll h^{-1/4}$ we split the sum over $k$ in \eqref{diracmoddeltaM<} as follows
\begin{equation}
\label{dataEMEM}
  \sum_{k\geq 1}(\cdots)=
\sum_{k\geq1}\cutoffchi^{\flat}\Big(\frac{\tilde \omega_k}{\tilde \omega_{K_{\ceps}}}\Big)\mathcal{E}_M(x,y,a,\tilde\omega_k)+\sum_{k\geq 1}\chi^{\flat}(h^{2/3}\tilde\omega_k/\ceps_{0})\cutoffchi^{\sharp}\Big(\frac{\tilde\omega_k}{\tilde\omega_{K_{\ceps}}}\Big)\mathcal{E}_M(x,y,a,\tilde \omega_k)\,.
\end{equation}

\begin{prop}\label{propdataapetitpetit1}
For all $a\in (0,h^{2/3-\ceps})$ there exists a smooth function $f_{h,a,1}$ such that
\begin{equation}\label{trucmuche}
\sum_{k\geq 1}\frac{2\pi}{\tilde L'(\tilde \omega_k)}\cutoffchi^{\flat}(h^{2/3}\tilde\omega_k/\ceps_{0})K_{\tilde\omega_k}(f_{h,a,1})(0,x,y)=\sum_{k\geq1}\cutoffchi^{\flat}\Big( \frac{\tilde \omega_k}{\tilde \omega_{K_{\ceps}}}\Big)\mathcal{E}_M(x,y,a,\tilde\omega_k)+O(h^{\infty})\,.
\end{equation}
\end{prop}
The proof of Proposition \ref{propdataapetitpetit1} is postponed to Section \ref{sectproofs}, as it requires arguments from section \ref{sectpseudocalcul}; we will introduce the cutoff $\cutoffchi^{\flat}( {\tilde\omega_k}/{\tilde\omega_{K_{\ceps}}})$ in the (LHS) term of \eqref{trucmuche} as well, due to how $f_{h,a,1}$ is obtained.
\begin{prop}\label{propdataapetitpetit2}
  For all $a\in (0,h^{2/3-\ceps})$, there exists a smooth function $f_{h,a,2}$ such that
  \begin{gather}
\label{eq:g1}\begin{multlined}[t]
\langle \sum_{N\in \Z} e^{-iN \tilde L(\omega)}
    , {\cutoffchi^{\flat}}(h^{2/3}\omega/\ceps_{0}){\cutoffchi^{\sharp}}(h^{2\ceps}\omega)K_\omega(f_{h,a,2})(0,x,y)\rangle_{\omega}\\
\quad\quad\quad\quad    {}=\sum_{k\geq 1}\cutoffchi^{\sharp}\Big(\frac{\tilde\omega_k}{\tilde \omega_{K_{\ceps}}}\Big)\chi^{\flat}(h^{2/3}\tilde\omega_k/\ceps_{0})\mathcal{E}_M(x,y,a,\tilde\omega_k)+O(h^{\infty}),
\end{multlined}\\
\label{eq:g12}
\langle \sum_{N\in \Z} e^{-iN \tilde L(\omega)}
    , {\cutoffchi^{\flat}}(h^{2/3}\omega/\ceps_{0}){\cutoffchi^{\sharp}}(\omega){\cutoffchi^{\flat}}( h^{2\ceps}\omega)K_\omega(f_{h,a,2})(0,x,y)\rangle_{\omega}=O(h^{\infty}).
\end{gather}
\end{prop}
We introduced a new cut-off ${\cutoffchi^{\sharp}}(h^{2\ceps}\omega)$ in \eqref{eq:g1} and removed ${\cutoffchi^{\sharp}}(\omega)$ (which is supported for $\omega\geq 1/2$ and identically $1$ on the support of ${\cutoffchi}^{\sharp}(h^{2\ceps}\omega)$). The proof of Proposition \ref{propdataapetitpetit2} will be provided in Section \ref{proofprop2}.
The sum of \eqref{eq:g1} and \eqref{eq:g12} yields the second term in \eqref{dataEMEM}. Finally, we have
\begin{prop}\label{propparama<}
For all $a\in (0,h^{2/3-\ceps})$, let $f_{h,a}:=f_{h,a,1}+f_{h,a,2}$, then $\Prond_{h,a}$, defined in Definition \ref{defparamPoisson} is a parametrix in the sense of Definition \ref{dfnparametrix} and \eqref{dataamic} holds.
\end{prop}
Proposition \ref{propparama<} follows easily from Propositions \ref{propdataapetitpetit1} and \ref{propdataapetitpetit2} using $K_{\omega}(f_{h,a,1}+f_{h,a,2})=K_{\omega}(f_{h,a,1})+K_{\omega}(f_{h,a,2})$ together with \eqref{eq:Prond2}.
For later purposes (dispersion for small $a$, Section \ref{secdispapetit}), taking $K_{\ceps}\sim h^{-2\ceps}$ would be enough. However, in the next subsection, following \cite{ILLP}, we aim at obtaining gallery modes for $k$ as large as possible, which turns out to be up to $K_{\ceps}\sim h^{-1/4+\ceps}$. This allows to achieve the proof of Theorem \ref{thmN} for initial data supported at distance smaller than $h^{2/3-\epsilon}$ from the boundary.

\subsubsection{Pseudo-differential calculus; construction of gallery modes.
}\label{sectpseudocalcul}
\begin{dfn}
Let $G_N$ be defined in \eqref{eq:GN}.
We set
\begin{equation}\label{defeomkgen}
\tilde e(x,y,\eta,\omega)=\frac{\sqrt{2\pi} q(\eta)^{1/6}}{\sqrt{\tilde L'(\omega)}}e^{-iy\cdot \eta}G_N(x,y,\eta,\omega).
\end{equation}
\end{dfn}
\begin{dfn}
Let $\varkappa$ be  like in Definition \eqref{dfnparametrix}. For $g\in L^2(\mathbb{R}^{d-1})$, we define
\begin{equation}\label{defFk}
F_{\tilde\omega_k}(g)(x,y): =\frac{1}{(2\pi )^{d-1}}\int e^{i(y-y')\cdot \eta}\tilde e(x,y,\eta,\tilde\omega_k)(\varkappa(hD_{y'})g)(y')d\eta dy'\,. 
\end{equation}
\end{dfn}
\begin{dfn}\label{blabla}
Let $\tilde \varkappa\in C^{\infty}_0(\mathbb{R}^{d-1})$ be such that $\tilde\varkappa=1$ on the support of $\varkappa$ and vanishing outside a neighborhood of $\mathbb{S}^{d-2}$. Let also $\chi\in C^{\infty}_0$ be a smooth cutoff supported in the ball of center $0$ and radius $1/16$ of $\mathbb{R}^{d-1}$.  
For $f\in L^2(\mathbb{R}^{d-1})$ we define an operator $\Lo$ as
\begin{equation}\label{def:L}
\Lo(f)(y):=\int e^{i(y-y')\cdot\eta-i|\eta|B_0(y',\eta/|\eta|)}\tilde\varkappa(h\eta)\chi(y')f(y')dy'd\eta,
\end{equation}
where $B_0$ is the first term in the development of $B_{\Gamma}$ in \eqref{formalseriesBA} and is homogeneous of degree $0$.
\end{dfn}
To define $f_{h,a,1}$ and $f_{h,a,2}$, we need to "invert" $F_{\tilde\omega_k}$, which requires estimating derivatives of $\tilde e(x,y,\eta,\tilde\omega_k)$ with respect to $(y,\eta)$: in the Friedlander model case the corresponding mode $\tilde e_k(x,\eta)$ from \eqref{eig_k} does not depend on $y$, but here it does and one derivative that falls on the phase $\psi$ can provide a large factor $1/h$. Since Theorem \ref{thmMelrose} does not provide additional information about the phase functions $\psi$ and $\zeta$, we need to use Proposition \ref{propimpformgamma} as in \cite{ILLP} in order to be able to describe in mode detail these phases and, in particular, their dependence on $y$.

Deriving $\tilde e$ (hence $G_N$) with respect to $y$ yields a large factor due to $\mathcolor{red}{\mathbf{\tau}} B_0$ in the phase $\mathcolor{red}{\mathbf{\tau}} \Gamma(x,y,\sigma q^{1/3}(\eta)/\mathcolor{red}{\mathbf{\tau}},\eta/\mathcolor{red}{\mathbf{\tau}})$. 
In order to get rid of the homogeneous term of degree zero $|\eta|B_0(y,\eta/|/\eta|)$, we proceed as in \cite[Section 2.3.2]{ILLP} : set $\tilde F_{\tilde\omega_k}:=F_{\tilde\omega_k}\circ \Lo$, which is a new operator whose phase function does not include the contribution $|\eta|B_0(y,\eta/|\eta|)$; we prove that, at least for $k\ll h^{-1/4}$, these operators can be inverted and this will be our main result in this section:
\begin{prop}\label{propmatrix}
Let $0<\ceps_{1}<1/4$ be small. The operators $\tilde F_{\tilde\omega_k}^*\circ \tilde F_{\tilde\omega_k}:L^2(\mathbb{R}^{d-1})\rightarrow L^2(\mathbb{R}^{d-1})$ are pseudo-differential operators that are uniformly elliptic with respect to $1\leq k\leq h^{-1/4+\ceps_{1}}$.
\end{prop}
The proof of Proposition \ref{propmatrix} follows like the one of \cite[Prop 2.30]{ILLP} ; we recall here the main steps but omit details as the argument is identical to the Dirichlet case (when instead of $\tilde\omega_k$ we work with $\omega_k$).

%
Setting $\tilde F_{\tilde\omega_k}=F_{\tilde\omega_k}\circ \Lo$, we can explicitly compute the adjoint operator $\tilde F_{\tilde\omega_k}^*$ using that for every $f\in L^2(\mathbb{R}^{d-1})$ and $\mathcal{E}\in L^2(\Omega)$ we have $<\tilde F_{\tilde\omega_k}(f),\mathcal{E}>_{L^2(\Omega)}=<f,\tilde F_{\tilde\omega_k}^*(\mathcal{E})>_{L^2(\mathbb{R}^{d-1})}$. This further yields
$\tilde F_{\tilde\omega_k}^*\circ \tilde F_{\tilde \omega_k}(f)(z)$ for $f\in L^2(\mathbb{R}^{d-1})$ and $z\in\mathbb{R}^{d-1}$. Since these computations were already carried out in \cite[Section 2.3.2]{ILLP} with $\omega_k$ in place of $\tilde \omega_k$, we only record the final formulas here. As in \cite[(2.72)]{ILLP}, we find
%
\begin{equation}\label{FistarFj}
\tilde F_{\tilde\omega_k}^*\circ \tilde F_{\tilde \omega_k}(f)(z)
=\int_{z'}M_{k}(z,z')f(z')dz' +O(h^{\infty}),
\end{equation}
where 
we have set $\breve{e}(x,y,\eta,\tilde\omega_k)=e^{-i|\eta|B_0(y,\eta/|\eta|)}\tilde e(x,y,\eta,\tilde\omega_k)$ and where, modulo $O(h^{\infty})$ terms,
\begin{gather}\label{mjk}
M_{k}(z,z')=\int_{\tiTheta}e^{i(z-z')\cdot\tiTheta} m_{k}(z',\tiTheta) d\tiTheta\,,\\
m_{k}(z,z',\tiTheta) = \begin{multlined}[t]
\frac{1}{(2\pi )^{d-1}}\int _{\zeta}\int_{w} \int_0^{\infty} e^{-iw\cdot\zeta}|\nabla_{\tiTheta}\eta|(\tiTheta,z,z'+w)|\nabla_{\tiTheta'}\eta|(\tiTheta-\zeta,z',z'+w)\\ 
\times {\varkappa}(h\eta(\tiTheta,z,z'+w)){\varkappa}(h\eta'(\tiTheta-\zeta,z',z'+w))\tilde\chi(z'+w)\chi(z)\chi(z')\\
\times \overline{\tilde{e}(x,z'+w,\eta(\tiTheta,z,z'+w),\omega_k)} \tilde{e}(x,z'+w,\eta(\tiTheta-\zeta,z',z'+w),\omega_k) dx dwd\zeta.
\end{multlined}
\end{gather}
On the support of the two cut-offs $\varkappa$, $h(\tiTheta+O(z,z'+w))\in [1/2,2]$,  $h(\tiTheta-\zeta+O(z',z'+w))\in[1/2,2]$, and on the support of $\chi(z)\chi(z'+w)\tilde\chi(z'+w)$, $|z,z'|<1/16, |z'+w|<1/8$; then, set $\tiTheta=\frac{\theta}{h}$, $\zeta=\frac{\varrho}{h}$ where $\theta\in[1/4,5/4]$ and $|\varrho|\leq 2$. In \eqref{mjk} we may replace $m_{k}(z,z',\tiTheta)$ by $m_{k}(z',z',\tiTheta)=:\tilde{m}_{k}(z',\tiTheta)$ without changing the integral modulo $O(\tiTheta^{-\infty})=O(h^{\infty})$. In the new variables (and modulo $O(h^{\infty})$ terms), the symbol of \eqref{mjk} becomes
\begin{equation}\label{eq:mktilde}
  \tilde{m}_{k}(z',\theta/h) =\frac{1}{(2\pi h )^{d-1}}\int _{\varrho}\int_{w} 
  e^{-\frac ih w\cdot \varrho}
a_{k}((z',\theta); (w,\varrho);h)
  dw d\varrho\,,
\end{equation}
where
\begin{multline}\label{defajk}
a_{k}((z',\theta); (w,\varrho);h):=|\nabla_{\tiTheta}\eta|(\theta/h,z',z'+w)|\nabla_{\tiTheta'}\eta|((\theta-\varrho)/h,z',z'+w)\\\times {\varkappa}(h\eta(\theta/h,z',z'+w)){\varkappa}(h\eta((\theta-\varrho)/h,z',z'-w))\tilde\chi(z'+w)\chi(z)\chi(z')\\
\times \int_0^{\infty}\overline{\breve{e}(x,z'+w,\eta(\theta/h,z,z'+w),\tilde\omega_k)} \breve{e}(x,z'+w,\eta((\theta-\varrho)/h,z',z'+w),\tilde\omega_k) dx\,.
\end{multline}
 Define
\begin{equation}
  \label{eq:20}
  \mathcal{S}_{\gamma}:=\{a\in C^{\infty} \,\,\text{such that}\,\, |\partial_{w}^{\beta_{1}}\partial_{\varrho}^{\beta_{2}}a(w,\varrho,h)|\leq C_{\beta} h^{-\gamma (|\beta_{1}|+|\beta_{2}|)}\}\,.
\end{equation}
The following result holds as in \cite[Prop. 2.27]{ILLP},
\begin{prop}\label{propsymbole}
Let $0<\ceps_{1}<1/4$ be small. The symbols $a_k((z',\theta);(w,\varrho);h)$ are in the class $\mathcal{S}_{\frac{1}{2}-\frac 2 3\ceps_{1}}$, uniformly with respect to $1\leq k\leq h^{-\frac 1 4+\ceps_{1}}$.
\end{prop}
\begin{proof}
The proof follows exactly like in \cite[Lemma 2.28]{ILLP}. We have to prove that for every $|\beta|\geq 1$, as $k\leq h^{-1/4+\ceps_{1}}$,
\begin{equation}\label{akderb}
  h^{|\beta|}\Big|\partial^{\beta}_w\partial^{\beta}_{\varrho}a_{k}((z',\theta);(w,\varrho);h)|_{w=0,\varrho=0}\Big|
  \lesssim C_{\beta}h^{|\beta|}\big(\tilde\omega_k/h^{1/3}\big)^{2|\beta|}\lesssim C_{\beta}\big(h^{4\ceps_{1}/3}\big)^{|\beta|}\,.
\end{equation}
This follows from \eqref{defajk}, $\tilde\omega_k\sim k^{2/3}$ and the next lemma, whose proof follows in the same way as \cite[Lemma 2.32]{ILLP} :
\begin{lemma}\label{lemestimderivekk}
Uniformly for $h^{2/3}\tilde\omega_k\ll 1$, we have:
\begin{equation}\label{estimderivekk}
\|\partial^{\beta}_{(y,\theta)}\breve{e}(.,y,\theta/h,\tilde\omega_k)\|_{L^2(x\geq 0)}\lesssim \big(\tilde\omega_k/h^{1/3}\big)^{|\beta|}\,.
\end{equation}
\end{lemma}
\end{proof}
Using the classical expansion \cite[Lemma 2.29]{ILLP} of a symbol $a\in\mathcal{S}_{1/2-\epsilon}$ with small $\epsilon>0$, yields
\begin{equation}\label{mtildeh}
\tilde{m}_{k}(z',\theta/h)= \sum_{\beta}\frac{h^{|\beta|}}{i^{|\beta|}|\beta|!}\partial^{\beta}_w\partial^{\beta}_{\varrho}a_{k}((z',\theta);(w,\varrho);h)|_{w=0,\varrho=0}\,.
\end{equation}
We may now return to the proof of our main Proposition.
\begin{proof}(of Proposition \ref{propmatrix})
  From \eqref{FistarFj} and \eqref{mjk},
  \begin{equation}
    \label{eq:21}
        \tilde F_{\tilde \omega_k}^*\circ \tilde F_{\tilde\omega_k}(f)(z)=\int e^{\frac ih (z-z')\cdot\theta}\tilde m_k(z',\theta/h)f(z')dz'\,,
      \end{equation}
 where the symbol $\tilde m_k(z',\theta/h)$ is given by \eqref{mtildeh} and where, for every $k\leq h^{-1/4-\ceps_{1}}$, \eqref{akderb} holds true. Moreover, $\tilde m_k$ is elliptic. Indeed, from \eqref{defajk}, it follows that
\[
a_{k}((z',\theta);(0,0);h)={\sigma}(z',\theta/h;0,0) \|\breve{e}(.,z',\theta/h,\tilde \omega_k) \|^2_{L^2(x>0)},
\]
where for $\tiTheta=\theta/h$ and $\zeta=\varrho/h$ like before, we define
\begin{multline}\label{sigmaandtildes}
\sigma((z',\tiTheta);(w,\zeta))=|\nabla_{\tiTheta}\eta|(\tiTheta,z',z'+w)|\nabla_{\tiTheta'}\eta|(\tiTheta-\zeta,z',z'+w)\\\times {\varkappa}(h\eta(\tiTheta,z',z'+w)){\varkappa}(h\eta(\tiTheta-\zeta,z',z'-w))\tilde\chi(z'+w)\chi(z)\chi(z').
\end{multline} 
As in Lemma \cite[Lemma 2.28]{ILLP}, we have (with $\tilde \omega_k$ in place of $\omega_k$)
$\|\breve e(\cdot,y,\eta,\tilde \omega_k)\|_{L^2(x>0)}\sim 1$. On the other hand the symbol $\sigma$ defined in \eqref{sigmaandtildes} is elliptic, therefore $a_{k}$ is elliptic and $a_{k}((z',\theta);(0,0);h)\sim 1$ for $\theta$ close to $1$; using \eqref{akderb}, $\tilde m_k$ is therefore elliptic; $\tilde F_{\tilde \omega_k}^*\circ \tilde F_{\tilde \omega_k}$ are pseudo-differential operators, uniformly elliptic for $k\leq h^{-1/4+\ceps_{1}}$.
\end{proof}

\subsubsection{Construction of quasi-modes $k\ll h^{-1/4}$ (proof of Proposition \ref{propdataapetitpetit1}).}\label{sectproofs}
It suffices to  construct a smooth function $f_{h,a,1}$ such that, for all $k$ such that $\chi_{\flat}(\tilde\omega_{k}/\tilde\omega_{K_{\ceps}})\neq 0$, 
\begin{equation}\label{eqKE}
\frac{2\pi}{L'(\omega_k)}K_{\tilde\omega_k}(f_{h,a,1})(0,x,y)=\mathcal{E}_M(x,y,a,\tilde\omega_k) +O(h^{\infty}).
\end{equation}
Indeed, for a function $f(y',\tprime)$, let $\hatdeux{f}(y',\alpha/h)$ be its Fourier transform w.r.t. $\tprime$ at $\alpha/h$. 
Using the definition of $K_{\tilde\omega}(f)$, we only need $\hatdeux{f}_{h,a,1}(.,\omega/h^{1/3})$ for $\omega\in\{\tilde\omega_k, 1\leq k\ll h^{-1/4}\}$.
\begin{lemma}\label{lemfkhatgk}
Let $h^{-2\ceps}\lesssim K_{\ceps}\lesssim h^{-1/4+\ceps}$, $\ceps>0$ small. For $1\leq k\leq 4 K_{\ceps}$, define
\begin{equation}
  \label{eq:13}
  f_{\tilde \omega_k}(y):=\Big(\tilde F^*_{\tilde\omega_k}\circ \tilde F_{\tilde\omega_k}\Big)^{-1}\Big(\tilde F^*_{\tilde\omega_k}(\mathcal{E}_M(.,\tilde\omega_k))\Big)(y)\,,
\end{equation}
then define $f_{h,a,1}$ with: $\forall k$ such that $\chi^{\flat}(\tilde \omega_k/\tilde\omega_{K_{\ceps}})\neq 0$,
\begin{equation}
  \label{eq:14}
  \hatdeux{f}_{h,a,1}(.,\tilde\omega_k/h^{1/3}) :=\frac{\sqrt{\tilde L'(\tilde\omega_k)}}{\sqrt{2\pi}}\Lo(f_{\tilde \omega_k})\,,\\
\end{equation}
and $\forall k$ such that $\chi^{\flat}(\tilde\omega_k/\tilde\omega_{K_{\ceps}})=0$,
$\hatdeux{f}_{h,a,1}(.,\omega_k/h^{1/3}) :=0$. Then Proposition \ref{propdataapetitpetit1} 
holds for $f_{h,a,1}$.
\end{lemma}
\begin{proof}
As $k\leq 4 K_{\ceps}\lesssim h^{-1/4+\ceps}\ll h^{-1/4}$, we can use Proposition \ref{propmatrix} and define $f_{\tilde\omega_k}$ by \eqref{eq:13}. For such $f_{\tilde\omega_k}$ and $\chi_{\flat}(\tilde \omega_{k}/\omega_{ K_{\ceps}})\neq 0$, we now define $\hatdeux{f}_{h,a,1}(y,\tilde \omega_k/h^{1/3})$ by \eqref{eq:14} (and zero for larger $k$'s).
  
By construction, on the support of $\chi_{\flat}(\tilde\omega_{k}/\tilde\omega_{K_{\ceps}})$, we have $F_{\tilde\omega_k}(\Lo(f_{\tilde\omega_k}))(x,y)=\tilde{F}_{\tilde\omega_k}(f_{\tilde \omega_k})(x,y)=\mathcal{E}_M(x,y,a,\tilde \omega_k)$, as we chose $f_{\tilde\omega_k}$ such that $\tilde F^*_{\tilde \omega_k}\circ \tilde F_{\tilde \omega_k}(f_{\tilde \omega_k})=\tilde F^*_{\tilde \omega_k}(\mathcal{E}_M(.,\tilde \omega_k))$. In turn, we have
\begin{equation}\label{eqKEfgfg}
F_{\tilde \omega_k}\Big(\frac{2\pi}{\sqrt{\tilde L'(\tilde \omega_k)}}\hatdeux{f}_{h,a,1}(.,\tilde \omega_k/h^{1/3})\Big)(x,y)=\mathcal{E}_{M}(x,y,a,\tilde \omega_k)+O(h^{\infty})\,,
\end{equation}
(inverting $\tilde F^{*}_{\omega_{k}}$); using \eqref{eq:Kequiv} and \eqref{defFk}, $\frac{2\pi}{\tilde L'(\tilde\omega_k)}K_{\tilde\omega_k}(f_{h,a,1})(0,x,y)=\mathcal{E}_{M}(x,y,a,\tilde\omega_k)+O(h^\infty)$. 
\end{proof}
We obtain the explicit form of $\hatdeux{f}_{h,a,1}$ as a corollary of Lemma \ref{lemfkhatgk}, which follows as in \cite[Corollary 2.31]{ILLP}.
\begin{cor}\label{lemgomegak}
We keep the notations from the proof of Lemma \ref{lemfkhatgk}.
Let 
\begin{equation}\label{Io}
I_a(\eta,\tilde \omega_k):=\int_{x,y} e^{-iy\cdot \eta}\overline{\tilde e(x,y,\eta,\tilde\omega_k)}\mathcal{E}_M(x,y,a,\tilde \omega_k)dxdy.
\end{equation}
For $1\leq k\leq K_{\ceps}$, $f_{h,a,1}$ (from Lemma \ref{lemfkhatgk}) may be rewritten
\[
\hatdeux{f}_{h,a,1}(y',\tilde\omega_k/h^{1/3})=\frac{\sqrt{\tilde L'(\tilde \omega_k)}}{\sqrt{2\pi}}\int e^{iy'\cdot\eta}\varkappa(h\eta)r(\eta,\tilde\omega_k)I_a(\eta,\tilde\omega_k)d\eta,
\]
where $r(.,\omega_k)$ is an elliptic symbol of order $0$ and main contribution $\frac{1}{\tilde m_k(y,\eta+|\eta|\partial_yB_0(y,\eta/|\eta|))}$ with $\tilde m_k$ defined in \eqref{mtildeh}. 
\end{cor}

\subsubsection{Proof of Proposition \ref{propdataapetitpetit2}}\label{proofprop2}
Our goal is to obtain $f_{h,a,2}$ such that \eqref{eq:g1} holds. Observe that the sum in the second line of \eqref{eq:g1}, involving $\mathcal{E}_M$ for large $\tilde\omega_k$, is, using \eqref{eq:AiryPoisson},
\begin{equation}\label{EMomegamare}
\sum_{k\geq 1}\cutoffchi^{\sharp}\Big( \frac{\tilde\omega_k}{\tilde\omega_{K_{\ceps}}}\Big)\cutoffchi^{\flat}(\frac{h^{2/3}\tilde\omega_k}{\ceps_{0}})\mathcal{E}_M(x,y,a,\tilde\omega_k)
=\langle \sum_{N\in \Z} e^{-iN \tilde L(\omega)}
    , {\cutoffchi^{\flat}}(h^{2/3}\omega/\ceps_{0})\cutoffchi^{\sharp}\Big(\frac{ \omega}{\tilde \omega_{K_{\ceps}}}\Big)\frac{\tilde L'(\omega)}{2\pi}\mathcal{E}_M(x,y,a,\omega))\rangle_{\omega},
\end{equation}
and non-stationary phase (with respect to $\omega$) easily applies for all $|N|\geq 2$ in the second line of \eqref{EMomegamare}, providing an $O(h^{\infty})$ contribution (this is just the model case). Therefore, we are left to obtain $f_{h,a,2}$ such that
\begin{multline}\label{tosolve}
\int {\cutoffchi^{\flat}}(h^{2/3}\omega)\cutoffchi^{\sharp}(h^{2\ceps} \omega)\Big(1+\sum_{\pm} e^{\pm i \tilde L(\omega)}\Big)K_\omega(f_{h,a,2})(0,x,y)d\omega\\+\sum_{|N|\geq 2}\int e^{- iN \tilde L(\omega)}{\cutoffchi^{\flat}}(h^{2/3}\omega/\ceps_{0})\cutoffchi^{\sharp}( h^{\ceps}\omega)K_\omega(f_{h,a,2})(0,x,y)d\omega \\
=\int  {\cutoffchi^{\flat}}(h^{2/3}\omega/\ceps_{0})\cutoffchi^{\sharp}\Big(\frac{ \omega}{\tilde\omega_{K_{\ceps}}}\Big)\Big(1+\sum_{\pm} e^{\pm i \tilde L(\omega)}\Big)\frac{\tilde L'(\omega)}{2\pi}\mathcal{E}_M(x,y,a,\omega)d\omega+O(h^{\infty}).
\end{multline}
Let us analyze the last line of \eqref{tosolve}, corresponding to the sum over model gallery modes. Here $\mathcal{E}_M$ is a product of two Airy functions $e^{iy\cdot \eta}\mathrm{Ai}(-\zeta_M(x,\eta,\omega))\mathrm{Ai}(-\zeta_M(a,\eta,\omega))$, where the phases $\zeta_M=\omega-xq^{1/3}(\eta)$ and $\psi_M(y,\eta)=y\cdot \eta$ are such that \eqref{systeikeq} holds with $<.,.>$ replaced by the scalar product obtained by polarization of the principal symbol $\xi^2+|\eta|^2+xq(\eta)$ of the model Laplace operator $\Delta_M$. Using the definition of $L$ in Lemma \ref{lemL},
$
1+\sum_{\pm}e^{\pm i\tilde L(\omega)}=1-\Big(\frac{A'_+}{A'_-}\Big)(\omega)-\Big(\frac{A'_-}{A'_+}\Big)(\omega)
$. As $h^{2/3}\tilde\omega_{K_{\ceps}}\sim (hK_{\ceps})^{2/3}\gtrsim (h^{1-2\ceps})^{2/3}\gg h^{2/3-\ceps}$, it follows that $h^{2/3}\omega$ is much larger than $a$ on the support of the symbol of the integral in the last line of \eqref{tosolve}, and we can use \eqref{eq:Apm} to write $Ai(-\zeta_M(a,\eta,\omega))=\sum_{\pm}A_{\pm}(\zeta_M(a,\eta,\omega))$. The phase of $\mathcal{E}_M$ is now
\begin{equation}\label{phaEm}
y\cdot \eta+\xi^3/3+\xi(xq^{1/3}(\eta)-\omega)\pm \frac 23 (\omega-aq^{1/3}(\eta))^{3/2}.
\end{equation}
\begin{lemma}
In the integral defining $\mathcal{E}_M(.,\omega)$, the usual stationary phase in $\xi$ applies. Moreover, for the phase corresponding to $N=0$ in the second line of \eqref{EMomegamare}, we have
\[
\phi_{M,\pm,\mp}(x,y,\eta,\omega):= y\cdot \eta\pm\frac 23 \Big((\omega-xq^{1/3}(\eta))^{3/2}-(\omega-aq^{1/3}(\eta))^{3/2}\Big).
\]
In the same way, the phases corresponding to $N=\pm1 $ in the second line of \eqref{EMomegamare} are
\[
\phi_{M,\pm,\pm}(x,y,\eta,\omega)\mp\frac 43 \omega^{3/2}:= y\cdot \eta\pm \frac 23\Big((\omega-xq^{1/3}(\eta))^{3/2}+(\omega-aq^{1/3}(\eta))^{3/2}-2\omega^{3/2}\Big).
\]
Moreover, for $x>2h^{2/3-\ceps}$, the integral in the second line of \eqref{tosolve} is $O(h^{\infty})$.
\end{lemma}
\begin{proof}
The proof follows as in \cite[Lemma 3.2]{ILLP}. Indeed, for $N=0$ the corresponding phases and symbols are exactly the same, while for $N=\pm 1$, the only difference with respect to \cite[Lemma 3.2]{ILLP} is the fact that we deal with factors $\Big(\frac{A'_+}{A'_-}\Big)^{\pm 1}(\omega)$ instead of $\Big(\frac{A_+}{A_-}\Big)^{\pm 1}(\omega)$. From \eqref{eq:Apm}, it follows that $\chi(\omega)-\tilde\chi(\omega)=\pi/2-\frac{1}{4\omega^{3/2}}+O(1/\omega^{3})$ and therefore
\[
\frac{A'_+}{A'_-}(\omega)=e^{-2i\tilde\chi(\omega)}=e^{-2i\chi(\omega)+2i(\pi/2+O(1/\omega^{3/2}))}=-\frac{A_+}{A_-}(\omega)e^{iO(1/\omega^{3/2})},
\]
where $\omega\geq \tilde\omega_{K_\epsilon}\gg h^{-\epsilon}$ is large. This allows to conclude again exactly like in \cite[Lemma 3.2]{ILLP}. 
\end{proof}
Putting all this together, the integral in the last line in \eqref{tosolve} reads as 
\[
\int  {\cutoffchi^{\flat}}(h^{2/3}\omega/\eps_{0})\cutoffchi^{\sharp}\Big(\frac{ \omega}{\tilde\omega_{K_{\ceps}}}\Big)\Big(1+\sum_{\pm} e^{\pm i \tilde L(\omega)}\Big)\frac{\tilde L'(\omega)}{2\pi}\mathcal{E}_M(x,y,a,\omega)d\omega=E_{M,+}(x,y,a)+E_{M,-}(x,y,a),
\]
where we have set, modulo $O(h^{\infty})$,
\begin{multline}\label{EM1}
E_{M,+}(x,y,a):=\int  {\cutoffchi^{\flat}}(h^{2/3}\omega/\eps_{0})\cutoffchi^{\sharp}({\omega}/{\tilde\omega_{K_{\ceps}}}) \int q(\eta)^{1/6}\varkappa(h\eta)\chi_0(h^{2/3}\omega q^{2/3}(h\eta))\\
\times e^{iy\cdot \eta}\Big(A_+(\zeta_M(x,\eta,\omega))-\Big(\frac{A'_+}{A'_-}\Big)(\omega)A_-(\zeta_M(x,\eta,\omega))\Big)A_-(\zeta_M(a,\eta,\omega))d\eta d\omega,
\end{multline}
\begin{multline}\label{EM2}
E_{M,-}(x,y,a):=\int  {\cutoffchi}^{\flat}(h^{2/3}\omega/\eps_{0})\cutoffchi^{\sharp}({\omega}/{\tilde \omega_{K_{\ceps}}})\int q(\eta)^{1/6}\varkappa(h\eta)\chi_0(h^{2/3}\omega q^{2/3}(h\eta))\\
\times e^{iy\cdot \eta}\Big(A_-(\zeta_M(x,\eta,\omega))-\Big(\frac{A'_-}{A'_+}\Big)(\omega)A_+(\zeta_M(x,\eta,\omega))\Big)A_+(\zeta_M(a,\eta,\omega))d\eta d\omega,
\end{multline}
where the phase functions of the Airy terms in the second line of \eqref{EM1} are $\phi_{M,+,-}$ and $\phi_{M,-,-}+\tilde L(\omega)$, while the phase functions of the Airy terms in the second line of \eqref{EM2} are $\phi_{M,-,+}$ and $\phi_{M,+,+}-\tilde L(\omega)$. Moreover, for $x>2h^{2/3-\ceps}$, $E_{M,\pm}(x,y,a)=O(h^{\infty})$ and $E_{M,\pm}(0,y,a)=0$. This means that we can introduce a smooth cut-off $\chi_1( x/h^{2/3-\ceps})$ with $\chi_1\in C^{\infty}_0$ equal to $1$ on $[-1,1]$ and equal to $0$ for $x\geq 2h^{2/3-\ceps}$ such that $E_{M,\pm}(x,y,a)=\chi_1(x/h^{2/3-\ceps})E_{M,\pm}(x,y,a)+O(h^{\infty})$, and therefore we need to construct $f_{h,a,2}$ such that \eqref{tosolve} holds with the last line replaced by $\chi_1(x/h^{2/3-\ceps})(E_{M,+}(x,y,a)+E_{M,-}(x,y,a))$ (instead of $E_{M,+}(x,y,a)+E_{M,-}(x,y,a)$). 

We now go back to \eqref{tosolve}: the symbol of its left hand side has support in $\omega\gtrsim h^{-2\ceps}$, while the right hand side is essentially supported for $x,a\lesssim h^{2/3-\ceps}$. For such values of $x$ and $\omega$ we have $\zeta(x,y,\eta,\omega)=\omega-x|\eta|^{2/3}e_0(x,y,\eta/|\eta|,\omega_k/|\eta|^{2/3})\geq \omega/2$; using \eqref{eq:GN} we write
\[
G_N(x,y,\eta,\omega)=e^{i\psi}\sum_{\pm}\Big(g_0A_{\pm}(\zeta)+i q(\eta)^{-1/6}g_1A'_{\pm}(\zeta)\Big)=:G_{N,\pm}(x,y,\eta,\omega).
\]
\begin{prop}\label{propapetitg2}
There exists smooth functions $f_{h,a,2,\pm}$ such that, with $f_{h,a,2}:=\sum_{\pm} f_{h,a,2,\pm}$,
\begin{multline}\label{g2-}
\int \cutoffchi^{\flat}(h^{2/3}\omega/\ceps_{0})\cutoffchi^{\sharp}( h^{2\ceps}\omega)\Big(G_{N,+}(x,y,\eta,\omega)-\Big(\frac{A'_+}{A'_-}\Big)(\omega)G_{N,-}(x,y,\eta,\omega)\Big)\\
q(\eta)^{1/6}\varkappa(h\eta)\varkappa(h\tau_{q}(\omega,\eta))\hat{f}_{h,a,2,-}(\eta,\omega/h^{1/3})d\eta d\omega
=E_{M,+}(x,y,a)+O(h^{\infty}),
\end{multline}
\begin{multline}\label{g2+}
\int\cutoffchi^{\flat}(h^{2/3}\omega/\ceps_{0})\cutoffchi^{\sharp}( h^{2\ceps}\omega) \Big(G_-(x,y,\eta,\omega)-\Big(\frac{A'_-}{A'_+}\Big)(\omega)G_{N,+}(x,y,\eta,\omega)\Big)\\
\times q(\eta)^{1/6}\varkappa(h\eta)\varkappa(h\tau_{q}(\omega,\eta))\hat{f}_{h,a,2,+}(\eta,\omega/h^{1/3})d\eta d\omega
=E_{M,-}(x,y,a)+O(h^{\infty}).
\end{multline}
\begin{multline}\label{g2+0}
\int\cutoffchi^{\flat}(h^{2/3}\omega/\ceps_{0})\cutoffchi^{\sharp}( h^{2\ceps}\omega) \Big(G_{N,+}(x,y,\eta,\omega)-\Big(\frac{A_+}{A_-}\Big)(\omega)G_{N,-}(x,y,\eta,\omega)\Big)\\
q(\eta)^{1/6}\varkappa(h\eta)\varkappa(h\tau_{q}(\omega,\eta))\hat{f}_{h,a,2,+}(\eta,\omega/h^{1/3})d\eta d\omega
=O(h^{\infty}),
\end{multline}
\begin{multline}\label{g2-0}
\int\cutoffchi^{\flat}(h^{2/3}\omega/\ceps_{0})\cutoffchi^{\sharp}( h^{2\ceps}\omega) 
\Big(G_{N,-}(x,y,\eta,\omega)-\Big(\frac{A_-}{A_+}\Big)(\omega)G_{N,+}(x,y,\eta,\omega)\Big)\\
\times q(\eta)^{1/6}\varkappa(h\eta)\varkappa(h\tau_{q}(\omega,\eta))\hat{f}_{h,a,2,-}(\eta,\omega/h^{1/3})d\eta d\omega
=O(h^{\infty}),
\end{multline}
\begin{equation}\label{toshowsecond}
  \sum_{|N|\geq 2 }\int e^{- iN \tilde L(\omega)}\int\cutoffchi^{\flat}(h^{2/3}\omega/\ceps_{0})\cutoffchi^{\sharp}( h^{2\ceps}\omega)
  K_\omega(f_{h,a,2})(0,x,y)d\omega =O(h^{\infty}).
\end{equation}
\end{prop}
\begin{proof}
It is enough to prove that the operators in the first lines of \eqref{g2+} and \eqref{g2-} are invertible. Indeed, once $f_{h,a,2,\pm}$ are defined, \eqref{g2+0}, \eqref{g2-0} and \eqref{toshowsecond} follow by non-stationary phase. More precisely, \eqref{g2-} forces the $\omega$-derivative of $f_{h,a,2,-}$ to be $ 2\sqrt{\omega}+O(a)$, since otherwise the phase of \eqref{g2-} would be non-stationary in $\omega$. Substituting such a function 
into \eqref{g2-0} therefore produces a phase whose $\omega$-derivative behaves like $4\sqrt{\omega}$ hence  a $O(h^{\infty})$ contribution by repeated integrations by parts, since $\omega$ is large on the support of the symbol. The same argument yields \eqref{toshowsecond} : for $|N|\geq 2$ the phases are non-stationary in $\omega$, and each integration by parts gains a factor $(N\sqrt{\omega})^{-1}$, allowing summation over $N$. We are thus reduced to constructing $f_{h,a,2,-}$ satisfying \eqref{g2-}; the construction of $f_{h,a,2,+}$ solving \eqref{g2+} is identical.
\begin{prop}\label{proptildeJ}
Let $\tilde J_+:=J_++R_+$, with
\begin{multline}
  J_+(f)(x,y)=\int \cutoffchi^{\flat}(h^{2/3}\omega/\ceps_{0})\cutoffchi^{\sharp}(h^{2\ceps}\omega)
  G_{N,+}(x,y,\eta,\omega)\chi_1(x/h^{2/3-\ceps}) \\
\times q(\eta)^{1/6}\varkappa(h\eta)\varkappa(h\tau_{q}(\omega,\eta))e^{i(y'\cdot\eta+\tprime \omega/h^{1/3})} f(y',\tprime )d\eta d\omega dy'd\tprime.
\end{multline}
\begin{multline}
  R_+(f)(x,y)=\int \cutoffchi^{\flat}(h^{2/3}\omega/\ceps_{0})\cutoffchi^{\sharp}(h^{2\ceps}\omega)
  \Big(\frac{A'_+}{A'_-}\Big)(\omega)G_{N,-}(x,y,\eta,\omega)\chi_1(x/h^{2/3-\ceps})  \\
\times q(\eta)^{1/6}\varkappa(h\eta)\varkappa(h\tau_{q}(\omega,\eta))e^{i(y'\cdot\eta+\tprime \omega/h^{1/3})} f(y',\tprime )d\eta d\omega dy'd\tprime.
\end{multline}
The operator $\tilde J_+$ is well defined from $\mathcal{S}'_{y',\tprime}$ into the space of functions of $(x,y)$ near $(0,0)$, and with $h$ as small parameter, $J_+$ is an elliptic semi-classical Fourier integral operator. Moreover,  $\|J_+^{-1}\circ R_+\|_{\mathcal{L}(L^{2})}=O(h^{\infty})$, hence $\tilde J_+$ is invertible and $\tilde J^{-1}_+=\Big(I+J^{-1}_+\circ R_+\Big)^{-1}\circ J^{-1}_+$.
\end{prop}
If we now chose $f_{h,a,2,-}(y',\tprime):=\tilde J_+^{-1}(E_{M,+})$, this achieves the proof of Proposition \ref{propapetitg2}.
\end{proof}
\begin{proof}(of Proposition \ref{proptildeJ}, which follows as in \cite[Prop. 2.34]{ILLP})
The operator $J_+$ is easily elliptic and invertible with phase function $\psi+\frac 23\zeta^{3/2}$ with $\psi$ and $\zeta$ defined in Theorem \ref{thmMelrose}. The oscillatory phase function of $R_+$ is $\psi-\frac 23\zeta^{3/2}+\frac 43 \omega^{3/2}$. Therefore the phase function of $J^{-1}_+\circ R_+$ is given by
\begin{equation}\label{phaBR}
-\psi(x,y,\eta,\omega)-\frac 23\zeta^{3/2}(x,y,\eta,\omega)+\psi(x,y,\tilde \eta,\tilde\omega)-\frac 23 \zeta^{3/2}(x,y,\tilde\eta,\tilde \omega)+\frac 43 \tilde\omega^{3/2},
\end{equation}
where $x,y$ are now integration variables. The derivative of \eqref{phaBR} with respect to $x$ is
\[
-\partial_x\zeta(x,y,\eta,\omega)\sqrt{\zeta(x,y,\eta,\omega)}-\partial_x\zeta(x,y,\tilde \eta,\tilde\omega)\sqrt{\zeta(x,y,\tilde\eta,\tilde\omega)}-\partial_x\psi(x,y,\eta,\omega)+\partial_x\psi(x,y,\tilde\eta,\tilde\omega),
\]
where $|\eta|,|\tilde \eta|\sim 1/h$, $\omega,\tilde\omega\geq h^{-2\ceps}$ and $xq^{1/3}(\eta)\leq h^{2/3-\ceps-2/3}=h^{-\ceps}$ on the support of the symbol. As $\zeta=\omega-x|\eta|^{2/3}e_0(x,y,\eta/|\eta|,\omega/|\eta|^{2/3})$ with $e_0$ elliptic and close to $1$, the derivatives of the two terms involving $\zeta$ in \eqref{phaBR} are such that 
\[
\Big|\partial_x\zeta(x,y,\eta,\omega)\sqrt{\zeta(x,y,\eta,\omega)}+\partial_x\zeta(x,y,\tilde \eta,\tilde\omega)\sqrt{\zeta(x,y,\tilde\eta,\tilde\omega)}\Big|\gtrsim \frac 12 (\sqrt{\omega}|\eta|^{2/3}+\sqrt{\tilde \omega}|\tilde\eta|^{2/3}).
\]
On the other hand, using \eqref{formpsigamma}, we obtain as in \cite[Prop. 2.35]{ILLP} the following bounds
\[
\Big|\partial_x\psi(x,y,\eta,\omega)-\partial_x\psi(x,y,\tilde \eta,\tilde\omega)\Big|\lesssim  \omega |\eta|^{1/3}+\tilde\omega |\tilde\eta|^{1/3},
\]
where we have used that $q^{2/3}(\eta)/\tau_{q}(\omega,\eta)=|\eta|^{4/3}q^{2/3}(\eta/|\eta|)/(|\eta|\tau_{q}(\omega/|\eta|^{2/3},\eta/|\eta|))\sim |\eta|^{1/3}$. On the support of the symbol $\chi(h^{2/3}\omega)\chi(h^{2/3}\tilde\omega)$ we have $|\omega|,|\tilde\omega|\leq \ceps_0h^{-2/3}$, which means that the main contribution of the derivative of \eqref{phaBR} comes from the terms involving $\zeta$ and behaves like $\sim (\sqrt{\omega}|\eta|^{2/3}+\sqrt{\tilde \omega}|\tilde\eta|^{2/3})$, as for $|\eta|,|\tilde\eta|\simeq 1/h$ we have $\sqrt{\omega}|\eta|^{2/3}\gg \omega|\eta|^{1/3}$  ($\omega h^{2/3}\ll 1$). To perform non stationary phase and obtain an $O(h^{\infty})$ contribution, we check that taking one derivative with respect to $x$ of the symbol provides a factor $O(h^{\ceps})$. Indeed, ${h^{2/3}}\partial_x(\chi(x/h^{2/3-\ceps}))/{\sqrt{\omega}}\sim h^{\ceps}/\sqrt{\omega}\leq h^{\frac 53 \ceps}$, which completes the proof.
\end{proof}
To complete the proof of Proposition \ref{propdataapetitpetit2}, it remains to prove that, for $f_{h,a,2,-}:=\tilde J^{-1}_+(E_{M,+})$, \eqref{eq:g12} holds: but then in \eqref{eq:g12} one obtains a vanishing symbol as $\chi^{\flat}(h^{2\ceps}\omega)\chi^{\sharp}( h^{2\ceps}\omega)=0$.\qed

\section{Dispersion estimates}
\label{sec:dispersion-estimates}

Let $\ceps$ be small, $0<\ceps<1/12$ be consistent with the parametrix construction in subsection \ref{lowparam} (to have an overlap between both regimes where we get dispersion estimates by different arguments). 
For $a\in(0,h^{2/3-\ceps})$, $0<\ceps<1/12$, the proof of the dispersive bound follows exactly as in \cite[Section 4]{ILLP}, using Propositions \ref{propdataapetitpetit1} and \ref{propdataapetitpetit2}, together with Remark \ref{rmqtransversesmall}. Let $a\geq h^{2/3-\ceps}$.
We will use the parametrix as a sum over $N$ to obtain the following dispersion estimates, restricting to positive times for the sake of simplicity.
\begin{thm}\label{dispintermediaire}
There exist $a_{0}$, $c$, $C$, $\ceps$ such that for all $|(t,x,y,h,a)|<a_{0}$, one has
\begin{itemize}
\item for $t\leq c \sqrt a$,
  \begin{equation}
    \label{eq:1}
|  \Prond_{h,a}(t,x,y)|\leq C h^{-d}  \min \left(1, (h/t)^{\frac{d-1} 2}\right) \,;
  \end{equation}
\item For $a\geq h^{2/3-\ceps}$, $t>c\sqrt a$,
  \begin{equation}
    \label{eq:2}
|  \Prond_{h,a}(t,x,y)|\leq C h^{-d}  \left(\frac h t\right)^{\frac{d-2} 2} \left(  (\max(a,x))^{\frac 1 4} \left(\frac h t \right)^{\frac 1 4} +h^{\frac 1 3}\right)\,;
  \end{equation}
\item For $a\leq h^{1/3+\ceps}$,
  \begin{equation}
    \label{eq:3}
|  \Prond_{h,a}(t,x,y)|\leq C h^{-d}  \left(\frac h t\right)^{\frac{d-2} 2+\frac 1 3} \,.
  \end{equation}
\end{itemize}
\end{thm}
The first estimate, \eqref{eq:1}, is just the (short time) dispersion for a free wave. On this timescale, the wave has at most one reflection and singularities have not appeared yet. One should point out that (a suitable version of) such dispersion is already proved in \cite{blsmso08}, for a more general boundary.

The second estimate, \eqref{eq:2}, is proved using the parametrix as a sum over reflected waves. The first term is due to swallowtail singularities (and always larger than the corresponding factor in the free dispersion) and the second term is due to the presence of cusps appearing after each swallowtail singularity, between two consecutive reflections; notice that here we use the parametrix construction in the region $a\geq h^{2/3-\ceps}$.

The third estimate, \eqref{eq:3}, will be proved using the parametrix as a sum over quasi-modes, and we postpone its proof to the last section, where we deal with decay of such quasi-modes.

The next  result allows to deal with the parametrix "behind the wave front" and follows like in \cite[Lemma 3.2]{ILLP}.
\begin{lemma}
\label{derriere}
There exist $c_{0}$ and $T_{0}$ such that , with $\mathcal{B}=\{ 0\leq x\leq a, |y|\leq c_{0} t, 0<h\leq t \}$, 
  \begin{equation}
    \label{eq:2bis}
\forall t\in [0, T_{0}]\,,\quad\quad   \sup_{x,y,t\in \mathcal{B}}  |  \Prond_{h,a}(t,x,y)|\leq C h^{-d}  O(( h/ t)^{\infty})\,.
  \end{equation}
\end{lemma}

\subsection{Number of waves contributing to $\Prond_{h,a}$}\label{sectcardN1}

We use the dyadic localization introduced in \cite[Section 3.1]{ILLP}. Let $\phi\in C^\infty_0(\R)$ be even, with $\phi=1$ on $[-1,1]$ and $\phi=0$ outside $(-3/2,3/2)$, and set $\chi_1=\phi-\phi(2\cdot)$. For dyadic $\gamma\leq \ceps_0$, we denote by $K_{\omega,\gamma}$ the operator obtained from $K_\omega$ in \eqref{eq:Kequiv} by inserting the cutoff
$\chi_1\Big(\frac{\omega}{\gamma|\eta|^{2/3}}\Big)$,
and by $\Prond_{h,a,\gamma}$ the corresponding localization of $\Prond_{h,a}$. Thus
\[
\Prond_{h,a}=\sum_\gamma \Prond_{h,a,\gamma},
\qquad h^{2/3-\ceps}\lesssim\gamma\leq\ceps_0,
\]
where the sum is dyadic; moreover, the terms with $\gamma\ll a$ are $O(h^\infty)$ by non-stationary phase.

Using \eqref{eq:Prondcut} and the Airy-Poisson formula \eqref{eq:AiryPoisson}, we write
\begin{equation}\label{defProndgamma}
\Prond_{h,a,\gamma}(t,x,y)
=\sum_{N\in\Z}V_{N,\gamma}(t,x,y),
\end{equation}
where
\begin{equation}
\label{eq:newVNgam}
V_{N,\gamma}(t,x,y)
=\frac1{2\pi h^{d+1}}
\int e^{\frac ih\Phi_{N,a,\gamma}}
\chi_1\Big(\frac{\alpha}{|\theta|^{2/3}\gamma}\Big)
{\cutoffchi^{\flat}}\Big(\frac{\alpha}{\ceps_0}\Big)
{\cutoffchi^{\sharp}}\Big(\frac{\alpha}{h^{2/3}}\Big) \notag\
\chi(s)g_h(x,y,\theta,\alpha,\sigma)
\tilde q_h(\theta,\alpha,s)
,ds,d\theta,d\sigma,d\alpha ,
\end{equation}
with
\begin{equation}\label{PhiNagamma}
\Phi_{N,a,\gamma}
=t\tau_q(\alpha,\theta)
+\Phi(x,y,\theta,\alpha,\sigma)
-\Phi(a,0,\theta,\alpha,s)
-\frac43N\alpha^{3/2}
+NhB_{\tilde L}(\alpha^{3/2}/h).
\end{equation}
The amplitude in \eqref{eq:newVNgam} is a symbol of order zero, uniformly in $N$.

The localization and counting arguments of \cite[Section 3.1]{ILLP} apply without change to \eqref{eq:newVNgam}. Indeed, the phase functions $\Phi$ are the same as in the Dirichlet case, while $B_L$ is replaced by $B_{\tilde L}$. The arguments in \cite[Propositions 3.4-3.5]{ILLP} only use the corresponding symbol estimates and the asymptotic behavior
$|B'_{\tilde L}(\omega^{3/2})|
\sim |B'_L(\omega^{3/2})|
\sim \omega^{-3/2}$;
the different sign of the leading term of $B_{\tilde L}$ plays no role. The change in the amplitude is likewise harmless, since the new amplitude is a uniformly bounded symbol of order zero.

We therefore retain the following consequences. First, at fixed $|t|\lesssim T_0$,
\begin{equation}\label{eq:24}
\sum_{|N|\geq 4|t|/\sqrt\gamma}
V_{N,\gamma}(t,x,y)=O(h^\infty).
\end{equation}
In particular, only $O(|t|/\sqrt\gamma)$ reflected waves may contribute.
For fixed $(t,x,y)$, let
\begin{equation}\label{Ncal}
\mathcal{N}(t,x,y)
=\{N\in\Z:\ \Phi_{N,a,\gamma}
\text{ has a critical point in } (\sigma,s,\alpha,\theta)\}.
\end{equation}
We use the enlargement $\mathcal N_d^1(t,x,y)$ introduced in
\cite[Section 3.1]{ILLP}; equivalently, it is obtained by taking the union of $\mathcal N(t',x',y')$ over the $\gamma$-neighborhood $\mathcal C_\gamma(t,x,y)$ defined there. Then we have
\begin{prop}\label{propnofw}
The following holds 
\begin{equation}\label{eq:113}
|\mathcal N_d^1(t,x,y)|
\lesssim 1+\gamma^{-1/2}|t|
\big(\gamma^{3/2}/h\big)^{-2},
\end{equation}
and
\begin{equation}\label{eq:outsideN1}
\sum_{\substack{N\notin\mathcal N_d^1(t,x,y)\
|N|\lesssim |t|/\sqrt\gamma}}
V_{N,\gamma}(t,x,y)
=O(h^\infty).
\end{equation}
\end{prop}
These are precisely \cite[Propositions 3.4-3.5]{ILLP}, whose proofs apply verbatim with $B_L$ replaced by $B_{\tilde L}$.

We now complete the proof of Theorem \ref{dispintermediaire}. We decompose
\[
\Prond_{h,a}=\sum_{a\lesssim\gamma\ll1}\Prond_{h,a,\gamma}
\]
and distinguish the tangent regime $\gamma\sim a$ from the transverse regime $\gamma\geq4a$. We write $\Prond_{h,a,a}$ and $V_{N,a}$ for the corresponding terms when $\gamma\sim a$.
The term $N=0$ satisfies the usual free-space dispersive estimate
\begin{equation}\label{eq:36}
|V_0(t,x,y)|
\leq Ch^{-d}\left(\frac ht\right)^{\frac{d-1}{2}},
\qquad t\gtrsim h.
\end{equation}
The remaining terms satisfy the following estimates.

\begin{prop}\label{propN=1}
Assume $\ceps>0$, $a\in[h^{2/3-\ceps},a_0]$ and $h\in(0,1)$. Then, for $t\gtrsim h$,
\begin{equation}\label{eq:37}
|V_{\pm1}(t,x,y)|
\leq Ch^{-d}\left(\frac ht\right)^{\frac{d-2}{2}}
\left[
\left(\frac ht\right)^{1/2}
+a^{1/4}\left(\frac ht\right)^{1/4}
+h^{1/3}
\right].
\end{equation}
\end{prop}

\begin{prop}\label{proptransv}
Assume $\ceps>0$, $a\in[h^{2/3-\ceps},a_0]$, $\gamma\geq4a$, and set
$\lambda_\gamma=\gamma^{3/2}/h$. Then, for $t\gtrsim\sqrt\gamma$,
\begin{equation}\label{eq:37bis}
\left|\sum_{|N|\geq2}V_{N,\gamma}(t,x,y)\right|
\leq Ch^{-d}\left(\frac ht\right)^{\frac{d-2}{2}}
h^{1/3}
\left(
\frac{\gamma^{1/4}}{t^{1/2}}
+\lambda_\gamma^{-3/2}
\right).
\end{equation}
\end{prop}

\begin{prop}\label{proptang}
Assume $\ceps>0$, $a\in[h^{2/3-\ceps},a_0]$, and set
$\lambda=a^{3/2}/h$. Then, for $0\leq x\leq2a$,
\begin{equation}\label{eq:37ter}
\left|\sum_{|N|\geq2}V_{N,a}(t,x,y)\right|
\leq Ch^{-d}\left(\frac ht\right)^{\frac{d-2}{2}}
\left[
a^{1/4}\left(\frac ht\right)^{1/4}
+\frac{h^{1/3}}{\lambda^{4/3}}
\right].
\end{equation}
\end{prop}

These are the Neumann analogues of the corresponding estimates in
\cite[Section 3]{ILLP}. Their proofs use the same stationary-phase analysis.
Indeed, by construction the phase functions $\Phi$ are exactly those of the
Dirichlet parametrix, and the Neumann amplitude occurring in
\eqref{eq:newVNgam} is again a symbol of order zero with the required uniform
bounds. The only change in the phase is
$B_L(\alpha^{3/2}/h)\longrightarrow
B_{\tilde L}(\alpha^{3/2}/h)$.
From \eqref{eq:propL}, $B'_{\tilde L}$ has the same symbol behavior as $B'
_L$modulo the sign, which is irrelevant for out purpose.
In particular, the estimates for its derivatives used in
\cite[Section 3]{ILLP} are unchanged. The sign difference in their leading
terms does not enter any of the non-degeneracy or stationary-phase
arguments. We therefore only indicate below the few points needed to
identify the estimates with those of \cite{ILLP}.

\subsection{The tangent part $\gamma\sim a$}

Let $\gamma\sim a$ and $\lambda=a^{3/2}/h$. We proceed as in \cite[Section 3.2]{ILLP}, keeping track of the only change in the phase, namely the replacement of $B_L$ by $B_{\tilde L}$.
We first apply stationary phase to $\Psi_{N,a,a}$ with respect to $A$.

\begin{lemma}\label{lemAc}
The equation $\partial_A\Psi_{N,a,a}=0$ has at most one solution $A_c$ on the support of the symbol. This critical point is non-degenerate and
$|\partial_A^2\Psi_{N,a,a}||_{A=A_c}\sim Na^{3/2}$.
Consequently, stationary phase in $A$ yields the factor
\begin{equation}\label{eq:169}
(\lambda |N|)^{-1/2}.
\end{equation}
Moreover, if $N\sqrt a\lesssim1$, then
\begin{equation}\label{eq:146bis}
A_c=\Ab_0^2-2\Ab_0\Ab_1+
\left(\frac{\Sigma-S}{2N}\right)^2
+O\left(a\left(\frac{\Sigma}{N},\frac{S}{N}\right)^2\right)
+\frac{f_0}{N\lambda^2},
\end{equation}
where
\[
\Ab_0=
\frac{q^{2/3}(\vartheta_c)|_{\Sigma=S=0}}
{4N\sqrt a}\big(t+E_0(y,a)\big),
\qquad
\Ab_1=
\frac{\Sigma}{2N}(1-xE_1)
-\frac{S}{2N}(1-aE_2),
\]
with $E_0=O(|y|^2+a|y|)$, $E_{1,2}=O(1)$, while $f_0$ admits an asymptotic expansion in $\lambda^{-1}$.
\end{lemma}

\begin{proof}
This is the analogue of the corresponding result in \cite[Section 3.2]{ILLP}. The equation determining $A_c$ has exactly the same structure, except that
$1-\frac34B_L'(\rho\lambda A^{3/2})
\quad\text{is replaced by}\quad
1-\frac34B_{\tilde L}'(\rho\lambda A^{3/2})$.
The asymptotic expansion of $B_{\tilde L}$ from Lemma \ref{lemL} gives the same symbol bounds as for $B_L$, and therefore the proof of the existence, uniqueness and non-degeneracy of $A_c$ is unchanged. The same computation gives \eqref{eq:146bis}; the contribution containing $\rho$ is entirely contained in the remainder $f_0/(N\lambda^2)$.
\end{proof}
Let
\begin{equation}\label{eq:23}
\phi_{N,a}(t,x,y,\Sigma,S,\rho)
:=\Psi_{N,a,a}(t,x,y,\Sigma,S,A_c,\rho)
\end{equation}
be the critical value after stationary phase in $A$, and denote by
$\sigma_{V,h,a}(\Sigma,S,\rho)$ the resulting symbol after stationary phase in $\vartheta$ and $A$. It is a symbol of order zero, uniformly in $N$, with compact support in $(\Sigma,S,\rho)$.

For $|N|\lesssim\lambda$, the factor
$e^{iNB_{\tilde L}(\rho\lambda A^{3/2})}$ may be absorbed into the symbol as it doesn't oscillate. For larger $N$ we retain it in the phase. In particular, when $|N|>C\lambda^2$, the computation of \cite[Section 3.2]{ILLP} gives
\begin{equation}\label{eq:54}
\frac1h\left|\partial_\rho^2\phi_{N,a}\right|
\sim\frac{|N|}{\lambda}.
\end{equation}
Indeed, the only contribution to the second $\rho$ derivative comes from
$NhB_{\tilde L}(\rho\lambda A^{3/2})$; its sign plays no role. Hence stationary phase in $\rho$ yields the additional factor $(|N|/\lambda)^{-1/2}$ as in \cite[Section 3.2]{ILLP}.

We recall the two oscillatory estimates that will be used in the summation.

\begin{lemma}\label{lemestimVN}
There exists $C>0$, independent of $N$, such that:
\begin{enumerate}
\item if $|N|\geq\lambda^{1/3}$,
\begin{equation}\label{eq:64}
\left|
\int e^{\frac ih\phi_{N,a}}
\sigma_{V,h,a}(\Sigma,S,\rho) d\Sigma dS d\rho
\right|
\leq C\lambda^{-2/3};
\end{equation}
\item if $1\leq |N|<\lambda^{1/3}$ and
$\Lambda=\lambda |N|^{-3}\geq1$, then
\begin{equation}\label{eq:77}
\left|
\int e^{\frac ih\phi_{N,a}}
\sigma_{V,h,a}
\left(\frac{\sigma}{|N|},\frac{s}{|N|},\rho\right)
d\sigma ds
\right|
\leq C\Lambda^{-3/4}.
\end{equation}
\end{enumerate}
\end{lemma}

\begin{proof}
These are the analogues of the corresponding estimates in
\cite[Section 3.2]{ILLP}. For $|N|<\lambda^{1/3}$,
$NB_{\tilde L}=O(1)$ and its exponential can be included in the symbol; the resulting phase is linear in $\rho$, and the proof of \eqref{eq:77} is exactly the same.
For $|N|\geq\lambda^{1/3}$ the proof of \eqref{eq:64} uses only the symbol bounds on $B_L$ and, for $|N|>\lambda^2$, the non-degeneracy in $\rho$ expressed by \eqref{eq:54}. Both properties hold with $B_{\tilde L}$ in place of $B_L$, so the same argument applies.
\end{proof}

\subsubsection{Summation over the reflected waves}

We now prove Proposition \ref{proptang}. By Proposition
\ref{propnofw}, the terms
$N\notin\Nrond_d^1(t,x,y)$ contribute $O_{C^\infty}(h^\infty)$.
At a critical point of the phase defining $V_{N,a}$,
\begin{equation}\label{eq:Ntime}
|N|\sim\frac{t}{\sqrt a}.
\end{equation}
Set $\sharp N=|\Nrond_d^1(t,x,y)|$.
Assume first that $\sharp N=O(1)$. If
$1\leq|N|\leq\lambda^{1/3}$, collecting stationary phase in
$\vartheta$, \eqref{eq:169}, and \eqref{eq:77}, we obtain
\begin{align}
|V_{N,a}(t,x,y)|
&\lesssim
h^{-d}\left(\frac ht\right)^{\frac{d-2}{2}}
\frac{a^2}{h}
\frac1{|N|^{1/2}\lambda^{1/2}}
|N|^{1/4}\lambda^{-3/4} \notag\
&\lesssim
h^{-d}\left(\frac ht\right)^{\frac{d-2}{2}}
\frac{a^{1/8}h^{1/4}}{|N|^{1/4}} \notag\
&\lesssim
h^{-d}\left(\frac ht\right)^{\frac{d-2}{2}}
a^{1/4}\left(\frac ht\right)^{1/4},
\label{eq:tangent-smallN}
\end{align}
where we used \eqref{eq:Ntime} in the last inequality.
If $|N|\geq\lambda^{1/3}$, we instead use \eqref{eq:64} and obtain
\begin{equation}\label{eq:tangent-largeN}
|V_{N,a}(t,x,y)|
\lesssim
h^{-d}\left(\frac ht\right)^{\frac{d-2}{2}}
\frac{a^2}{h}
\frac1{|N|^{1/2}\lambda^{1/2}}
\lambda^{-2/3}.
\end{equation}
Using $|N|\geq\lambda^{1/3}$, this is bounded by the first term in
\eqref{eq:37ter}, or by
$h^{-d}\left(\frac ht\right)^{\frac{d-2}{2}}h^{1/3}$,
as required.

It remains to consider the case $\sharp N\gg1$. By \eqref{eq:113},
$\sharp N\lesssim
1+\frac{t}{\sqrt a\,\lambda^2}$.
Hence $\sharp N\gg1$ implies $t\gtrsim\sqrt a,\lambda^2$; moreover,
by \eqref{eq:Ntime}, $|N|\sim t/\sqrt a$. Using \eqref{eq:64},
\eqref{eq:169} and the above cardinality bound, we obtain
\begin{equation}
\left|
\sum_{\substack{N\in\Nrond_d^1(t,x,y)\
|N|\sim t/\sqrt a}}V_{N,a}
\right|
\lesssim
h^{-d}\left(\frac ht\right)^{\frac{d-2}{2}}
a^{1/2}\lambda^{1/3}\frac{\sharp N}{|N|} 
\lesssim
h^{-d}\left(\frac ht\right)^{\frac{d-2}{2}}
h^{1/3}\lambda^{-4/3}.
\label{eq:tangent-sum}
\end{equation}
Together with the preceding bounds, this yields
\[
\left|\sum_{|N|\geq2}V_{N,a}(t,x,y)\right|
\lesssim
h^{-d}\left(\frac ht\right)^{\frac{d-2}{2}}
\left[
a^{1/4}\left(\frac ht\right)^{1/4}
+\frac{h^{1/3}}{\lambda^{4/3}}
\right],
\]
which is \eqref{eq:37ter} and completes the proof of Proposition
\ref{proptang}.

\subsection{The almost transverse part $4a\leq\gamma\ll1$}

We return to $\Psi_{N,a,\gamma}(t,x,y,\Sigma,S,A,\rho)$, obtained after stationary phase in $\vartheta$. As in \cite[Section 3.3]{ILLP}, when $4a\leq\gamma$ we first apply stationary phase in $S$.

\begin{lemma}\label{lemSAtransv}
The phase $\Psi_{N,a,\gamma}$ has two distinct critical points $S_\pm$ in $S$, satisfying
\begin{equation}\label{critScpm}
S_{\pm}^{2}+ \frac{a}{\gamma}q^{1/3}(\vartheta_{c})
\left(1+\partial_{\Xi}A_{\Gamma}
\left(a,0,\Xi,\frac{\vartheta_{c}}{\tau_{q}(\gamma A,\vartheta_{c})}\right)
\Big|_{\Xi=\frac{\sqrt{\gamma}S{\pm}q^{1/3}(\vartheta_c)}
{\tau_q(\gamma A,\vartheta_c)}}\right)=A.
\end{equation}
Moreover,
$\partial_S^2\Psi_{N,a,\gamma}|_{S=S_\pm}\sim1$,
 $S_\pm=S_0\pm\sqrt{A-O(a/\gamma)}$,
 $S_0=O(a/\sqrt\gamma)$.
For each $S_\pm$, the resulting phase has a non-degenerate critical point $A_\pm$ in $A$, with
$\partial_A^2\Psi_{N,a,\gamma}|_{(S_\pm,A_\pm)}
\sim N$.
Consequently, stationary phase in $S$ and $A$ yields respectively the factors
\begin{equation}\label{eq:SAfactors}
\lambda_\gamma^{-1/2},
\qquad
(\lambda_\gamma |N|)^{-1/2},
\qquad \lambda_\gamma=\gamma^{3/2}/h.
\end{equation}
\end{lemma}

\begin{proof}
The analysis in $S$ is exactly the one in \cite[Section 3.3]{ILLP}, since it only involves the phase functions inherited from the Dirichlet construction. The equation determining $A_\pm$ has the same form as in the Dirichlet case, except that
$1-\frac34B_L'(\rho\lambda_\gamma A^{3/2})
\quad\text{is replaced by}\quad
1-\frac34B_{\tilde L}'(\rho\lambda_\gamma A^{3/2})$.
The symbol estimates for $B_{\tilde L}$ are the same as those used for $B_L$, hence the non-degeneracy in $A$ and the stationary-phase factors are unchanged.
\end{proof}

For $|N|>\lambda_\gamma^2$ we additionally apply stationary phase in $\rho$, exactly as in the tangent regime. If
$\phi_{N,a,\gamma,\pm}(t,x,y,\Sigma,\rho)
:=
\Psi_{N,a,\gamma}
(t,x,y,\Sigma,S_\pm,A_\pm,\rho)$,
then, as in \eqref{eq:54},
\begin{equation}\label{eq:rhotransv}
\frac1h
\left|
\partial_\rho^2\phi_{N,a,\gamma,\pm}
\right|
\sim\frac{|N|}{\lambda_\gamma}.
\end{equation}
Thus stationary phase in $\rho$ yields the additional factor
$(|N|/\lambda_\gamma)^{-1/2}$.
Let $\sigma_{V,h,\gamma,\pm}(\Sigma,\rho)$ denote the symbol obtained after stationary phase in $\vartheta$, $S$ and $A$. It is a symbol of order zero, uniformly in $N$, with compact support in $\Sigma$. We are therefore left with oscillatory integrals of the form
$\int e^{\frac ih\phi_{N,a,\gamma,\pm}}
\sigma_{V,h,\gamma,\pm}(\Sigma,\rho)\,d\Sigma\,d\rho$.

The following phase properties are those used in the Dirichlet case.

\begin{lemma}\label{lemtransversephases}
For $|N|\geq2$ and $t\sim4N\sqrt\gamma$, the phase
$\phi_{N,a,\gamma,\pm}$ has at most one degenerate critical point in $\Sigma$, of order two. If $|N|\geq\lambda_\gamma^2$, the same property holds for the critical value obtained after stationary phase in $\rho$.
\end{lemma}

\begin{proof}
This is \cite[Lemmas 3.34 and 3.37]{ILLP}. The proofs only involve the geometric part of the phase and the symbol estimates of $B_L$. They therefore apply with $B_{\tilde L}$ in place of $B_L$.
\end{proof}

For completeness, the exceptional terms $N=0,\pm1$ satisfy the same estimates as in \cite[Lemmas 3.35-3.36]{ILLP}: for $N=0$ and $|t|>\gamma$,
$\partial_\Sigma^2\phi_{0,a,\gamma,\pm}\sim\frac{t}{\sqrt\gamma}$, $|\Sigma|\lesssim1$,
while for $N=\pm1$ the phase has at most one critical point of order two. These estimates give Proposition \ref{propN=1}; below we only consider $|N|\geq2$.

\subsubsection{Summation over $N$ and proof of Proposition \ref{proptransv}}

By \eqref{eq:outsideN1}, terms with
$N\notin\Nrond_d^1(t,x,y)$ contribute $O_{C^\infty}(h^\infty)$.
Set again 
$\sharp N=|\Nrond_d^1(t,x,y)|$.
At a critical point we have
\begin{equation}\label{eq:Ntransv}
|N|\sim\frac{t}{\sqrt\gamma}.
\end{equation}

Assume first that $\sharp N=O(1)$. For
$2\leq|N|\lesssim\lambda_\gamma^2$, collecting stationary phase in
$\vartheta$, the two factors in \eqref{eq:SAfactors}, and degenerate stationary phase in $\Sigma$, we obtain
\begin{equation}
|V_{N,\gamma}(t,x,y)|
\lesssim
\frac1{h^d}\left(\frac ht\right)^{\frac{d-2}{2}}
\frac{\gamma^2}{h}
\frac1{\lambda_\gamma^{1/2}}
\frac1{|N|^{1/2}\lambda_\gamma^{1/2}}
\frac1{\lambda_\gamma^{1/3}}
\lesssim
\frac1{h^d}\left(\frac ht\right)^{\frac{d-2}{2}}
\frac{h^{1/3}}{\sqrt{|N|}}.
\label{eq:87}
\end{equation}
Using \eqref{eq:Ntransv}, this gives
\begin{equation}\label{eq:87bis}
|V_{N,\gamma}(t,x,y)|
\lesssim
\frac1{h^d}\left(\frac ht\right)^{\frac{d-2}{2}}
h^{1/3}\frac{\gamma^{1/4}}{t^{1/2}}.
\end{equation}

For $|N|\geq\lambda_\gamma^2$, we use in addition
\eqref{eq:rhotransv}; hence
\begin{equation}
|V_{N,\gamma}(t,x,y)|
\lesssim
\frac1{h^d}\left(\frac ht\right)^{\frac{d-2}{2}}
\frac{\gamma^2}{h}
\frac1{\lambda_\gamma^{1/2}}
\frac1{|N|^{1/2}\lambda_\gamma^{1/2}}
\left(\frac{|N|}{\lambda_\gamma}\right)^{-1/2}
\frac1{\lambda_\gamma^{1/3}}
\lesssim
\frac1{h^d}\left(\frac ht\right)^{\frac{d-2}{2}}
\frac{h^{1/3}\lambda_\gamma^{1/2}}{|N|}.
\label{eq:89}
\end{equation}

It remains to consider the case where several reflected waves overlap.
By \eqref{eq:113},
$\sharp N
\lesssim
1+\frac{t}{\sqrt\gamma\,\lambda_\gamma^2}$.
Together with \eqref{eq:Ntransv} and \eqref{eq:89}, this yields
\begin{equation}
\left|
\sum_{N\in\Nrond_d^1(t,x,y)}V_{N,\gamma}
\right|
\lesssim
\frac1{h^d}\left(\frac ht\right)^{\frac{d-2}{2}}
\frac{h^{1/3}\lambda_\gamma^{1/2}}{|N|}
\sharp N \lesssim
\frac1{h^d}\left(\frac ht\right)^{\frac{d-2}{2}}
\frac{h^{1/3}}{\lambda_\gamma^{3/2}}.
\label{eq:90}
\end{equation}
Combining \eqref{eq:87bis} and \eqref{eq:90}, we obtain
\[
\left|\sum_{|N|\geq2}V_{N,\gamma}(t,x,y)\right|
\lesssim
\frac1{h^d}\left(\frac ht\right)^{\frac{d-2}{2}}
h^{1/3}
\left(
\frac{\gamma^{1/4}}{t^{1/2}}
+\frac1{\lambda_\gamma^{3/2}}
\right),
\]
which proves Proposition \ref{proptransv}.

\begin{rmq}\label{rmqtransversesmall}
The argument above only requires
$\gamma\gtrsim h^{2/3-\ceps}$, and not
$a\gtrsim h^{2/3-\ceps}$. In particular, both the transverse dispersive estimate and the counting bound \eqref{eq:113} remain valid for arbitrary $a>0$. This will be used in the small-$a$ regime below.
\end{rmq}

\subsection{The transverse contribution}

It remains to consider the part of the initial data which is microlocalized away from the glancing region. More precisely, with the notation of Lemma~\ref{lemgab}, we replace the first term in the right-hand side of \eqref{eq:Jdirac} by
$(1-\chi_0(hD_x))\varkappa(hD_y)\delta_{(a,0)}$.
The corresponding contribution to the wave flow is represented, as above, by a sum over the successive reflections at the boundary. The analogue of \eqref{eq:newVN} is obtained by replacing the glancing cutoff
${\cutoffchi^{\flat}}\bigl(\alpha/(2\ceps_0)\bigr)$
by $1-{\cutoffchi^{\flat}}\bigl(\alpha/(2\ceps_0)\bigr)$.
Thus, up to a harmless modification of \(\ceps_0>0\), the corresponding oscillatory integrals are supported in the transverse region
\[
\alpha\geq\ceps_0.
\]
On this support the associated bicharacteristics meet the boundary transversally, uniformly away from glancing, and the phases have only non-degenerate critical points. Moreover, by finite speed of propagation, on every fixed sufficiently small time interval only finitely many terms in the sum over N can contribute. Standard stationary phase can therefore be applied to each of these terms and yields the same decay as for the free wave propagator,
\[
\bigl|\Prond^{\mathrm{tr}}_{h,a}(t,x,y)\bigr|
\lesssim h^{-d}\min\left\{1,
\left(\frac{h}{|t|}\right)^{\frac{d-1}{2}}\right\}.
\]

For $a<h^{2/3-\ceps}$, we proceed exactly as in the Dirichlet case. Instead of prescribing the initial data in \eqref{dataamic}, we require
\[
\Prond_{h,a}(0,x,y)
=
\bigl(1-\chi_0(-h^2\partial_x^2+xq(h\partial_y))\bigr)
\varkappa(hD_y)\delta_{(a,0)}
+O_{C^\infty}(h^\infty).
\]
Using the spectral decomposition in terms of the Neumann quasimodes $\tilde e_k$, this selects the modes satisfying, up to a harmless modification of the cutoff,
\[
h^{2/3}\tilde\omega_k\geq\ceps_0.
\]
The resulting parametrix is then treated exactly as the contribution corresponding to
$\tilde\omega_k>\tilde\omega_{K_\epsilon}$ in the Dirichlet analysis. Indeed, the subsequent argument only uses that the Airy parameter $\tilde\omega_k$ is sufficiently large, so that the relevant oscillatory integrals can be handled by the usual stationary phase arguments. They therefore satisfy the free-space dispersive bound as well.

Hence the transverse contribution exhibits no additional dispersive loss, and the boundary-induced loss is entirely contained in the glancing contribution.

\section{Index of notations}
Below is a commented list of the main notations, with reference to their very first occurence. 
\subsection{General notations (used consistently throughout the paper)}
\begin{itemize}
\item $(\Omega,g)$ = $d$ dimensional manifold, $d\geq 2$, $\Delta_g$ its Laplace Beltrami operator, section \ref{intro}.
  \item $\upgamma(d)$ encodes the loss in Strichartz estimates w.r.t. the case without boundary, Theorem \ref{thStri}.
\item $(x,y)$, boundary normal coordinates; $t$ the time variable; locally, $\Omega=\{(x,y) : x>0, y\in\mathbb{R}^{d-1}\} $, section \ref{parconstruction}.
\item $(\xi,\eta,\tau)$, dual variables: $(x,y,t,\xi,\eta,\tau)\in T^*(\Omega\times\mathbb{R}_t)$.
For $(x,y)$ near $(0,y)$, the metric is $\xi^2+R(x,y,\eta)$. In a neighborhood of $(0,0)\in \partial\Omega$, 
\[
  R_0(y,\partial_y):=R(0,y,\eta)
  \,, \quad R_1(y,\eta):=\partial_x R(0,y,\partial)
  \,,
\]
section \ref{parconstruction} and \eqref{eq:R01}.
\item $\Delta_M=\partial^2_x+\sum_j\partial^2_{y_j}+x\sum_{j,k=1}^{d-1}R^{j,k}_1(0)\partial_{y_j}\partial_{y_k}$: model Laplace operator, \eqref{eq:LapM};  Multipliers 
 \[
q(\eta)=\sum_{j,k=1}^{d-1}R^{j,k}_1(0)\eta_j\eta_k,\quad \tau_q(\omega,\eta):=\sqrt{|\eta|^2+\tilde\omega q(\eta)^{2/3}}\,.
 \]
\item $\{\tilde e_k(x,\eta)\}_{k\geq 0}$: in the spectral decomposition of $-\Delta_M$ (section \ref{sss231}), an explicit orthonormal base of eigenfunctions satisfying the Neumann boundary condition, associated to eigenvalues $\lambda_k(\eta)$, where
\[
  \lambda_k(\eta)=|\eta|^2+\tilde\omega_k q(\eta)^{2/3}=\tau_q^2(\tilde\omega_k,\eta).
\]
\item $\{-\tilde \omega_k\}_{k\geq 0}$: zeros of the derivative of the Airy function in decreasing order. Everywhere in the paper $\omega >1$ and serves as a substitute to the $\xi$ variable: if $Q_{y}$ is the differential operator with symbol $q$, $\alpha=h^{2/3}\omega$ quantizes the operator $x-Q_{y}^{-1}\partial_{x}^{2}$.
\item $s$, $\sigma$: integration variables in Airy type oscillatory integrals, \eqref{eq:Gosc}, \eqref{eq:15}.
\item  $(a,b)$ coordinates of the source point, mostly set with $b=0$, Theorem \ref{disper}.
\item  $h\in (0,1)$ (Theorem \ref{disper}), $\gamma\in (0,1)$ with $1/\gamma \in 2^{\mathbb{N}}$ (Section \ref{sectcardN1}): small parameters.
    \item $\lambda=a^{3/2}/h$, $\lambda_{\gamma}=\gamma^{3/2}/h$: large parameters, Section \ref{sectcardN1}.
    \item $(X,Y,T)$, rescaled coordinates (using some combination of $a$, $h$ or $\lambda$, $\lambda_{\gamma}$ as rescaling parameters), Section \ref{sectcardN1}; $\Sigma$, $S$, $A$: rescaled variables in Airy-type oscillatory integrals, Section \ref{sectcardN1}.
      \item  $\omega$: \eqref{eq:LapM}, parameter and integration variable, successively rescaled to $\alpha$ and then $A$ (Section \ref{sectcardN1}). Stationary phases in oscillatory integrals are performed with respect to $\alpha$, $\sigma$, $s$ or their rescaled versions $A$, $\Sigma$, $S$, less frequently $\eta$, with a large parameter being $1/h$, $\lambda$ or $\lambda_{\gamma}$.
    \item $\theta$: (Section \ref{parconstruction}) rescaled $\eta$, near $\mathbb{S}^{d-1}$, and $\rho=|\theta|$, $\vartheta=\theta/\rho$.
\end{itemize}
\subsection{Localisations in phase space} 
\begin{itemize}
\item We localize $\tau_q(\omega,\eta)\sim 1/h$ and $|\eta|\sim 1/h$.
For small $x$, this corresponds to large frequencies $-\Delta\sim -\Delta_M \sim 1/h^2$ and "tangent" directions: the number of reflections on the boundary may be quite large.
\item A further localization is to values $\omega/|\eta|^{2/3}\sim \gamma$. Informally, it relates to  the angle of incidence at the boundary for a ray starting tangentially from $(\gamma,0)$.  %
\item Cut-offs : $\varkappa\geq 0$ is a cut-off function in $ C_{0}^{\infty}(\mathbb{R}^{m})$ with $m=1$ or with $m=d-1$, localizing around  a small neighbourhood of $1$ (for $d=1$), or near $\mathbb{S}^{m-1}$  for $m=d-1$; $\varkappa_{1}$ is a 1-d $\varkappa$. We also have $\chi^{\flat}\in C^{\infty}(\mathbb{R})$ such that $\chi^{\flat}=1$ on $(-\infty, 1]$ and $\chi^{\flat}=0$ on $[2,\infty)$ and $\chi^{\sharp}=1-\chi^{\flat}$. Also, $\chi_0\in C^{\infty}_0(\mathbb{R})$ is supported is a small, fixed neighborhood of $0$.
\end{itemize}
\subsection{Operators, kernels and quasimodes}
\begin{itemize}
\item $G_D(x,y,\eta,\omega)$: a quasimode for Dirichlet, \eqref{eq:eqG} or \eqref{eq:Gosc}; satisfies \eqref{eq:eqG}  $-\Delta G=\tau_q^2 G+O_{C^{\infty}}(\tau_q^{-\infty})$; $G_N(x,y,\eta,\omega)$ a quasimode for Neumann, \eqref{eq:GN}; satisfies \eqref{eq:GN-quasimode}.
\item $K_{\omega}(f)(t,x,y)$: operator related to wave flow \eqref{eq:Kequiv}, acting on smooth $f$.
\item $J(f)(x,y)$: Fourier integral operator, \eqref{eq:J}.
\item $\mathcal{P}_{h,a}(t,x,y)$, \eqref{eq:Prond} and \eqref{eq:Prond2}, our parametrix for the wave equation
\item $V_{N}$: a wave in the expansion over $N$ of $\mathcal{P}_{h,a}$, \eqref{eq:newVN}.
\item $V_{N,\gamma}$: further localized with $\chi_{1}(\omega/(\gamma |\eta|^{2/3}))$, \eqref{eq:newVNgam}.
  \item $\mathcal{P}_{h,a,\gamma}$: the corresponding sum over $N$, \eqref{defProndgamma}. 
  \item $\mathcal{E}_M(\cdot,\omega_k)$:  galery modes for the model Laplacian $\Delta_{M}$, \eqref{defEmathcalM}.
  \item $\tilde e_{k}(x,\eta)$: eigenfunctions of $\mathcal{F}_{y}(\Delta_{M})$, \eqref{eig_k}.
    \item  $\tilde e(x,y,\eta,\omega)$: quasimodes for $\Delta_{g}$, \eqref{defeomkgen}.
\item $f_{h,a}$, $f_{h,a,j}$, $j\in\{1,2\}$: functions to serve as arguments to $J$ and $K_{\omega}$ to construct a suitable smoothed out Dirac data, Propositions \ref{propdataapetitpetit1} and \ref{propdataapetitpetit2}.
\item $F_{\omega_k}(g)(x,y)$: operator acting on functions $g\in L^2(\mathbb{R}^{d-1})$, average (with density $\hat g(\eta)$) of quasimodes $\tilde e(x,y,\eta,\omega_{k})$, \eqref{defFk}.
\item $\Lo(f)(y)$: operator actiong on $f\in L^2(\mathbb{R}^{d-1})$ which allows to "get rid" of the term $B_0$ in the phase of $\tilde e(x,y,\eta,\omega)$,  \eqref{def:L}.
\item $\tilde F_{\omega_k}(f)(x,y)=F_{\omega_k}\circ \Lo(f)(x,y)$: its main property is that it can be inverted.
\end{itemize}
\subsection{Phase functions and canonical transformation} 
\begin{itemize}
\item $\zeta(x,y,\eta,\omega)$, $\psi(x,y,\eta,\omega)$ : the phase functions of $G(x,y,\eta,\omega)$ from Theorem \ref{thmMelrose}.
\item $\Sigma_0$: \eqref{eq:18}, a neighborhood of a glancing point in the model case.
  \item $ \canonchi_M$, \eqref{eq:melrose},the canonical transformation defined in a conic neighborhood of $\Sigma_0$ mapping the model case (variables $(X_{M},Y_{M}, \Xi, \Theta)$) to the general case (variables $(x,y,\xi,\eta)$).
  \item $\varphi_{\Gamma}(x,y,\Xi,\Theta)=x\Xi+y\Theta+\Gamma(x,y,\Xi,\Theta)$: Proposition \ref{lemgamma}, the generating function for $\canonchi_M$, with $\Gamma(x,y,\Xi,\Theta)=B_{\Gamma}(y,\Theta)+xA_{\Gamma}(x,y,\Xi,\Theta)$ from \eqref{eq:GAB}.
    \item $A_{\Gamma}$, $B_{\Gamma}$: phase functions that are formal series \eqref{formalseriesBA}, defined near $\mathcal{GL}=\{x=0,\Xi=0,\varrho-1=0\}$ and for $(y,\vartheta)$ near $\{0\}\times \mathbb{S}^{d-1}$, where $\Theta=\varrho\vartheta$. Their explicit form is given in \eqref{AGam}.
\item $\mathcal{H}_{\geq j}=\{ F \text{ such that } F=\sum_{k\geq j} F_k, \text{ with } F_k \text{ homogeneous of degree } k\geq 1\}$, where a monomial of the form $x^a(\varrho-1)^b\Xi^c$ is homogeneous of degree $k$ if $c+2(a+b)=k$.

\item $\Lp(y,\vartheta)$ (which defines $A_1=\Xi \Lp$), $\alpha(y,\vartheta)$, $\beta(y,\vartheta)$, $\mu(y,\vartheta)$ (which define $A_2=\alpha(y,\vartheta)x+\beta(y,\vartheta)(\varrho-1)+\mu(y,\vartheta)\Xi^2$) so that $A_{\Gamma} =A_1+A_2+\mathcal{H}_{\geq 3}$: other functions related to $\Gamma$.
  \item $B_{\Gamma}(y,\Theta)=B_0(y,\vartheta)+(\varrho-1)B_2(y,\vartheta)+\mathcal{H}_{j\geq3}$, whose properties are stated in Proposition \ref{propimpformgamma}.
\item $\mathcal{L}=\{(x,y,\vartheta,\varrho,\Xi), p_{M,2}(X_M,Y_M,\Xi,\Theta)=0\}$.
\item $\Phi_{N,a,\gamma}$: the phase function of $V_{N,\gamma}$ defined in \eqref{PhiNagamma}.
%
%
    \end{itemize}

%

\def\cprime{$'$} \def\cprime{$'$}

\end{document}